\title{Higher geometric sheaf theories}
\author{Raffael Stenzel
}
\renewcommand\footnotemark{}
\documentclass[12pt,a4paper]{article}
\usepackage[utf8]{inputenc}
\usepackage[english]{babel}
\usepackage{xcolor}

\usepackage{amssymb,mathtools,amsmath,amsthm,stmaryrd}
\usepackage[all,cmtip,2cell]{xy} 

\usepackage[draft=false, hidelinks, linkcolor = blue]{hyperref}
\usepackage{cleveref}
\usepackage{tensor}
\usepackage{url}
\usepackage{bookmark}

\usepackage{graphicx}
\usepackage{lmodern}
\usepackage{enumitem}
\usepackage{array}

\usepackage[left=2.5cm,right=2.5cm,top=2.5cm,bottom=2.5cm]{geometry}
\def\defthm#1#2#3#4{
  \newtheorem{#1}[theorem]{#3}
  \newtheorem*{#1*}{#3}
  \newtheorem{#2}[theorem]{#4}
  \newtheorem*{#2*}{#4}
  \crefname{#1}{#3}{#4}
  \crefname{#2}{#4}{#4}  
}

\newtheoremstyle{mythm}%
{10pt}
{}
{\itshape}
{}
{\bf}
{.}
{.5em}
{}%

\newtheoremstyle{mydef}%
{10pt}
{3pt}
{}
{}
{\bf}
{.}
{.5em}
{}%

\newtheoremstyle{myrmk}%
{10pt}
{3pt}
{}
{}
{\bf}
{.}
{.5em}
{}%

\theoremstyle{mythm}
\newtheorem{theorem}{Theorem}[section]
\newtheorem*{theorem*}{Theorem}

\defthm{corollary}{corollaries}{Corollary}{Corollaries}
\defthm{lemma}{lemmata}{Lemma}{Lemmata}
\defthm{proposition}{propositions}{Proposition}{Propositions}
\defthm{axiom}{axioms}{Axiom}{Axioms}
\defthm{propdef}{props/defs}{Proposition/Definition}{Prositions/Definitions}
\defthm{defcor}{defscors}{Corollary/Definition}{Corollaries/Definitions}
\defthm{conjecture}{conjectures}{Conjecture}{Conjectures}

\theoremstyle{mydef}
\defthm{definition}{definitions}{Definition}{Definitions}

\theoremstyle{myrmk}
\defthm{acknowledgment}{acknowledgments}{Acknowledgment}{Acknowledgments}
\defthm{remark}{remarks}{Remark}{Remarks}
\defthm{motivation}{motivations}{Motivation}{Motivations}
\defthm{construction}{constructions}{Construction}{Constructions}
\defthm{example}{examples}{Example}{Examples}
\defthm{question}{questions}{Question}{Questions}
\defthm{observation}{observations}{Observation}{Observations}
\defthm{claim}{claims}{Claim}{Claims}
\defthm{notation}{notations}{Notation}{Notations}

\newtheorem*{replemmax}{\reptitle}
 {\end{replemmax}}

\newtheorem*{repthmx}{\reptitle}
 {\end{repthmx}}

\newtheorem*{repcorx}{\reptitle}
 {\end{repcorx}}

\crefname{section}{Section}{Sections}
\crefname{theorem}{Theorem}{Theorems}

\makeatletter
\renewenvironment{proof}[1][\proofname] {\par\pushQED{\qed}\normalfont\topsep6\p@\@plus6\p@\relax\trivlist\item[\hskip\labelsep\bf#1\@addpunct{.}]\ignorespaces}{\popQED\endtrivlist\@endpefalse}
\makeatother

\newcommand{\sprime}{^{\prime}}

\newcommand{\pbs}{\scalebox{1.5}{\rlap{$\cdot$}$\lrcorner$}}

\newcommand{\adj}{\rotatebox[origin=c]{180}{$\vdash$}\hspace{0.1pc}}

\newcommand{\Address}{{
  \bigskip
  \footnotesize

\textsc{Max Planck Institute for Mathematics, Vivatsgasse 7, 53111 Bonn, Germany}\par\nopagebreak
  \textit{E-mail address}: \texttt{stenzel@mpim-bonn.mpg.de}

}}

\begin{document}
\maketitle

\abstract{
We introduce the notion of a higher covering diagram in a base $\infty$-category $\mathcal{C}$. 
The theory of higher covering diagrams in $\mathcal{C}$ will be shown to recover various 
descent conditions known from the $\infty$-categorical literature in a uniform manner. In fact, higher covering 
diagrams always assemble to what we refer to as a structured colimit
pre-topology on the base $\mathcal{C}$. It hence always defines a sub-canonical sheaf theory over $\mathcal{C}$, and 
indeed defines the canonical such whenever $\mathcal{C}$ has pullbacks.
This ``higher geometric'' sheaf theory will be shown to differ from the usual infinitary-coherent sheaf theory by a cotopological localization whenever $\mathcal{C}$ is infinitary-coherent itself. We prove that 
this localization is generally non-trivial. For instance, every $\infty$-topos is the theory of 
higher geometric sheaves over itself, but the according infinitary-coherent sheaf theory over it is generally 
strictly larger. The higher geometric sheaves are hence characterized by a 
limit preservation property that is generally not captured by the classical 
sheaf condition. 
We define an $\infty$-category of higher geometric $\infty$-categories, and 
show that the (opposite of the) $\infty$-category of $\infty$-toposes embeds fully faithfully therein. We show that 
the higher $\kappa$-geometric sheaf theory on a higher $\kappa$-geometric $\infty$-category defines the free
$\infty$-topos generated by it, and consequently that it faithfully generalizes Lurie's definition of a ``sheaf'' 
over an $\infty$-topos.
}

\section{Introduction}\label{secintro}

\begin{notation*}
As is often custom, the prefix ``$(\infty,1)$'' will be abbreviated to ``$\infty$'' throughout this paper. The
$\infty$-category of spaces will be denoted by $\mathcal{S}$, the $\infty$-category of functors between two
$\infty$-categories $\mathcal{C}$, $\mathcal{D}$  will be denoted by $\mathrm{Fun}(\mathcal{C},\mathcal{D})$ or by
$\mathcal{D}^{\mathcal{C}}$ depending on the context. The $\infty$-category of presheaves
$\mathrm{Fun}(\mathcal{C}^{op},\mathcal{S})$ over a small $\infty$-category $\mathcal{C}$ will be denoted by
$\hat{\mathcal{C}}$. The functor $\mathsf{C}\colon\mathcal{D}\rightarrow\mathrm{Fun}(\mathcal{C},\mathcal{D})$ will 
denote the pre-composition with the functor $\mathcal{C}\rightarrow\Delta^0$; in particular,
$\mathsf{C}(B)\colon\mathcal{C}\rightarrow\mathcal{D}$ denotes the constant diagram with value $B\in\mathcal{D}$.

To avoid terminological confusion, we will refer to the $\infty$-categories and sheaf theories that this paper is 
about as \emph{higher} geometric rather than just geometric or $\infty$-geometric. The terminology is justified by 
the totality of the results in this paper. For the sake of further evident distinction, we will use exclusively the 
term ``$\kappa$-coherent'' to refer to the notion often synonymously called ``$\kappa$-coherent'',
``$\kappa$-geometric'' or ``$\kappa$-ary regular'' in the literature.
\end{notation*}

\paragraph{A motivation}
The central notion of a higher covering diagram as to be introduced in this paper may appear rather technical on 
first sight. To outline the idea unobscured by a machinery of formal constructions, we recall that sheaves on a 
topological space 
$X$ are by definition sheaves on the frame $\mathcal{O}(X)$ of open sets equipped with its canonical Grothendieck 
topology \cite{johnstone_ss, mlmsheaves}. Generally, the canonical topology on any frame $\mathcal{C}$ is generated 
by covers of the form $(U_i\leq B\mid i\in I)$ in $\mathcal{C}$ where $I$ is a set such that $\bigcup_{i\in I}U_i=B$. 
In other words, a cover of an object $B$ is determined by a set-indexed diagram
$U_{\bullet}\colon I\rightarrow \mathcal{C}_{/B}$ that is jointly isomorphic to $B$ (or, equivalently, jointly 
effective epimorphic over $B$).

More generally, every ordinary category $\mathcal{C}$ exhibits a canonical Grothendieck topology; that 
is, a largest Grothendieck topology $J$ on $\mathcal{C}$ such that all representable presheaves over $\mathcal{C}$ 
are $J$-sheaves. As stated in \cite[Section C2.1]{elephant}, this topology consists exactly of those sieves which are 
universally effective epimorphic. Here, a sieve on an object $B\in\mathcal{C}$ is a $(-1)$-truncated discrete 
fibration $S\hookrightarrow\mathcal{C}_{/B}$. A sieve $S\hookrightarrow\mathcal{C}_{/B}$ is effective epimorphic if 
it is colimiting when considered as a cocone over $B$. A sieve $S\hookrightarrow\mathcal{C}_{/B}$ is universally 
effective epimorphic if for all $f\colon C\rightarrow B$ the sieve $f^{\ast}S\hookrightarrow\mathcal{C}_{/C}$ 
obtained by base change is effective epimorphic.

We further recall that an ordinary category $\mathcal{C}$ is said to be $\kappa$-coherent if it is regular and its 
subobject-posets $\mathrm{Sub}(B)$ for $B\in\mathcal{C}$ have pullback-stable $\kappa$-small unions
\cite[Definition 5.10]{shulmanexact}. \footnote{At times $\mathcal{C}$ is further required to be well-powered, see 
e.g.\ \cite[Section 1.4]{elephant}.}
A category $\mathcal{C}$ is infinitary-coherent if it is $\kappa$-coherent for all cardinals $\kappa$
(that is, if it is regular and $\mathrm{Sub}\colon\mathcal{C}^{op}\rightarrow\mathrm{Cat}$ factors through the 
category of frames). In such $\mathcal{C}$, all effective epimorphic sieves are universally effective epimorphic.
For the rest of this motivation, we will ignore matters of size for the sake of convenience, and in particular use 
the word ``coherent'' as an umbrella term for ``regular'', ``coherent'', ``$\kappa$-coherent'' and
``infinitary-coherent''. Given an ordinary coherent category $\mathcal{C}$, the coherent Grothendieck topology at an 
object $B\in\mathcal{C}$ is generated by the jointly effective epimorphic (set-indexed) families
$U=\{U_i\rightarrow B\mid i\in I\}$. That means, a sieve $S\hookrightarrow\mathcal{C}_{/B}$ is covering if it 
contains all maps $f\colon C\rightarrow B$ that factor through one of the components of some jointly effective 
epimorphic family $U$ over $B$.
\begin{align*}
\xymatrix{
 & U_i\ar[d] \\
C\ar@{-->}@/^/[ur]^{\exists ?}\ar[r]_f  & B
}
\end{align*}
The property of a sieve to be covering does not record the explicit factorizations 
themselves but rather the mere existence of such. In particular, a covering sieve does not have the space to 
distinguish between different such factorizations. This justifies the discrete indexing of the generating covering families in the first 
place, as any further structure would be forgotten by virtue of the eventual propositional truncation anyway. If one 
is to record the factorizations explicitly however, the diagrammatic shape of $U\colon I\rightarrow\mathcal{C}_{/B}$ 
does become relevant for the higher homotopical structure of the associated presheaf
$(\mathcal{C}_{/B})^{op}\rightarrow\mathcal{S}$ that maps an arrow $f\in\mathcal{C}_{/B}$ to the space of such 
factorizations. With this in mind, we are motivated to study ``structured'' topologies $T$ on $\infty$-categories
$\mathcal{C}$ whose objects can be presented by not-necessarily $(-1)$-truncated but still colimiting right 
fibrations $U\colon I\twoheadrightarrow\mathcal{C}_{/B}$.
In fact, keeping in mind that any Grothendieck pre-topology can be completed to a Grothendieck topology, we will 
rather define structured pre-topologies $T$ which do not necessarily consist of fibrations per se. The according 
sheaves, defined as the presheaves local for the class
\begin{align}\label{coverintro}
\mathrm{Cov}_T=\{\mathrm{colim}yU\rightarrow yB\mid U\in T\},
\end{align}
are exactly those presheaves over $\mathcal{C}$ which take the chosen colimits to limits.
The objects of a cover $\mathrm{colim}yU\rightarrow yB$ in $\mathrm{Cov}_T$ over an arrow $f\colon C\rightarrow B$ is 
an explicit lift of $f$ to a component $U_i$. Given two such lifts $u\colon C\rightarrow U_i$,
$v\colon C\rightarrow U_j$, the space of identifications between $u$ and $v$ is computed as the equalizer of the two 
maps in $\hat{\mathcal{C}}_{/yC}$, which a priori can be of any homotopy type. 

Clearly not every class $T$ of colimiting diagrams in an $\infty$-category $\mathcal{C}$ is of such form that
the $\infty$-category $\mathrm{Sh}_T(\mathcal{C})$ of $T$-sheaves is an $\infty$-topos.
To briefly exemplify this, let $\mathcal{C}$ be a presentable $\infty$-category. We then may consider the 
class $T$ of all small diagrams in $\mathcal{C}$. Then the Yoneda embedding induces an equivalence $y\colon\mathcal{C}\rightarrow\mathrm{Sh}_T(\mathcal{C})$. In particular, the sheaf 
theory $\mathrm{Sh}_T(\mathcal{C})$ is an $\infty$-topos if and only if $\mathcal{C}$ was an $\infty$-topos in the 
first place. This in fact is exactly the definition of the $\infty$-category of sheaves on an $\infty$-topos in 
\cite{luriehtt}. This means however that $\mathrm{Sh}_T(\mathcal{C})$ cannot be an $\infty$-topos if $\mathcal{C}$ 
is, for instance, the frame $\mathcal{O}(X)$ of opens on a topological space $X$.

We therefore define the notion of a structured colimit pre-topology $T$ on a base
$\infty$-category $\mathcal{C}$, and proceed to construct a structured colimit pre-topology
$\mathrm{Geo}(\mathcal{C})$ of higher covering diagrams in $\mathcal{C}$ which is provably the largest such 
whenever $\mathcal{C}$ has pullbacks.
%
Its associated sheaf theory $\mathrm{Sh}_{\mathrm{Geo}}(\mathcal{C})$ of higher geometric sheaves will be shown 
to be the canonical sheaf theory over $\mathcal{C}$ accordingly.
The key fact that the localization $\hat{\mathcal{C}}\rightarrow\mathrm{Sh}_{\mathrm{Geo}}(\mathcal{C})$ is left 
exact is achieved by requiring that the property of being ``higher covering'' 
is stable under base change as well as under the construction of higher parametrized path-objects. The latter is to 
say that a higher covering diagram $U\colon I\rightarrow\mathcal{C}_{/B}$ not only covers $B$ itself -- in the sense 
that it is colimiting over $B$ -- but that it also covers all parametrized path-objects $U_i\times_B U_j$ of $B$ in 
canonical fashion (as well as their higher path-objects in turn, see Definition~\ref{defcovfunctors}). Universally 
jointly effective epimorphic families then correspond exactly to those higher covering diagrams which cover the
path-objects $U_i\times_B U_j$ trivially, see Section~\ref{secregtop}. The left exact localization
$\hat{\mathcal{C}}\rightarrow\mathrm{Sh}_{\mathrm{Geo}}(\mathcal{C})$ will generally not be topological, and hence 
not be given by a Grothendieck topology on $\mathcal{C}$. 
Rather, its topological fragment will be generated exactly by the universally jointly effective epimorphic families 
(whenever $\mathcal{C}$ is coherent). In this sense, the higher geometric sheaf theories yield a 
counterpart to the set-valued coherent sheaves on 1-categories. This raises the question whether these 
sheaf theories may be shown to arise as classifying $\infty$-toposes of something that may be referred to as 
geometric homotopy type theories in the future.

\paragraph{Summary of results}
In Section~\ref{secbases}, we recall the notion of a modulator from \cite{as_soa}, state a few useful lemmata about 
them, and further recall a diagonal criterion for modulators to generate an $\infty$-topos by way of localization 
from \cite{abfjsheavesI}. Section~\ref{secdesc} introduces the central notions of this paper. Here, we discuss
semi-descent and descent diagrams in general $\infty$-categories (Section~\ref{secsubdesc}), and use those to define 
and study structured colimit pre-topologies (Section~\ref{secsubstrcolimtop}) as well as higher covering diagrams 
(Section~\ref{secsubinftydesc} and Section~\ref{secsubhcd}). We show that the class of higher covering diagrams 
defines the canonical sheaf theory over any $\infty$-category with pullbacks up to a cardinality caveat
(Theorem~\ref{prophcdmaxtop}), and 
discuss examples. We furthermore point out an interplay between higher covering diagrams and
descent diagrams (Remark~\ref{remdeschcdequiv}), which shows that the two notions are closely related.  
In Section~\ref{secexttop} and Section~\ref{secregtop}, we show that various classic doctrines of category theory 
and their sheaf theories arise as special cases of higher geometric $\infty$-categories and their sheaf theories as 
to be introduced in Section~\ref{secsites}.
In Section~\ref{secexttop}, we show that extensivity of an $\infty$-category equates to the condition that all 
(finite) discrete diagrams factor through a higher covering diagram (Corollary~\ref{corextdescalt}). We show that the 
extensive sheaf theory over an extensive $\infty$-category is hence generated by the higher covering diagrams indexed 
by a (finite) set (Lemma~\ref{lemmaextsheaveschar}). In particular, the according localization is topological and 
sub-canonical. Furthermore, we show that in the finite case the resulting $\infty$-topos is 
hypercomplete (Corollary~\ref{corexthyper}), and hence has enough points whenever $\mathcal{C}$ is lextensive 
(Corollary~\ref{corextpoints}).
In Section~\ref{secregtop}, we show that an $\infty$-category $\mathcal{C}$ with pullbacks is $\kappa$-coherent
(or regular as a special case) if and only if all according symmetric \v{C}ech nerve diagrams in $\mathcal{C}$ factor 
through a higher covering diagram (Theorem~\ref{corcharordgeocat}). We show that the $\kappa$-coherent sheaf theory 
over a $\kappa$-coherent $\infty$-category is hence generated by all higher covering diagrams indexed by the sorted 
Lawvere theories for $\kappa$-small set-sized collections of objects (after removing the terminal object), see 
Theorem~\ref{thmocdsheaves} and Remark~\ref{remocdlvthy}. 
In particular, the according localization is topological and sub-canonical. Furthermore, we show that the resulting
$\infty$-topos is generally not hypercomplete for any $\kappa$, and that it hence generally does not have enough 
points (Proposition~\ref{propregpts}).
In Section~\ref{seccohtop}, we show that the higher $\kappa$-geometric sheaf theory over a $\kappa$-coherent
$\infty$-category $\mathcal{C}$ is a cotopological localization of the $\kappa$-coherent sheaf theory 
whenever $\kappa$ is uncountable (Proposition~\ref{propcohcotoptopfac}).
We show that every $\infty$-topos is the $\infty$-topos of higher geometric sheaves over itself 
(Theorem~\ref{remcovtoposesexple}), and that the infinitary-coherent sheaf theory over an $\infty$-topos is generally 
strictly larger (Proposition~\ref{propcovnotequcoh}). In particular, the $\infty$-category of higher geometric 
sheaves is generally not hypercomplete either (Corollary~\ref{corcovhyper}).
In Section~\ref{secsites} we define the $\infty$-category of higher $\kappa$-geometric $\infty$-categories. We show 
that the (opposite of the) $\infty$-category of $\infty$-toposes embeds fully faithfully in the $\infty$-category of 
higher geometric $\infty$-categories (Proposition~\ref{propgeofunctors}), and that the higher $\kappa$-geometric 
sheaf theory $\mathcal{C}\rightarrow\mathrm{Sh}_{\mathrm{Geo}_{\kappa}}(\mathcal{C})$ over a small higher
$\kappa$-geometric $\infty$-category $\mathcal{C}$ is the free $\infty$-topos generated by $\mathcal{C}$ 
(Corollary~\ref{propfreegeotop}). Section~\ref{secappcof} is a short appendix on cofinality and cofinal equivalence 
that will be of relevance for the constructions in Section~\ref{secdesc}.

\begin{acknowledgments*}
The first version of this paper was written with the gratefully acknowledged support of the Grant Agency of the Czech 
Republic under the grant 19-00902S. The second version of this paper was written as a guest at the Max Planck Institute for Mathematics 
in Bonn, Germany, whose hospitality is greatly appreciated. This second version has considerably benefited 
from discussions with Mathieu Anel, as well as from helpful comments of an anonymous referee. The author 
also would like to thank Nathanael Arkor, John Bourke, Jonas Frey and Nima Rasekh for much appreciated comments and 
conversations.
\end{acknowledgments*}

\section{Modulators and localizations}\label{secbases}

We recall the notion of a modulator from \cite{as_soa} applied to our basic case of interest. That is, we 
fix a small $\infty$-category $\mathcal{C}$ and consider modulators for the $\infty$-category $\hat{\mathcal{C}}$ 
locally presented by the representables on $\mathcal{C}$.

A \emph{modulator} $M=\{M(B)\mid B\in\mathcal{C}\}$ on $\mathcal{C}$ is a collection of sets of objects
$M(B)\subset\hat{\mathcal{C}}_{/yB}$ such that each $M(B)$ contains the identity $1_{yB}\in\hat{\mathcal{C}}_{/yB}$, 
and such that the canonical inclusion 
\[\xymatrix{
M\ar@/_/[dr]\ar@{^(->}[r] & \hat{\mathcal{C}}_{/y}\ar@{->>}[d]\ar@{}[dr]|(.3){\pbs}\ar[r] & \hat{\mathcal{C}}^{\Delta^1}\ar@{->>}[d]^t \\
 & \mathcal{C}\ar[r]_y & \hat{\mathcal{C}}
}\] 
defines a full subfibration $M\twoheadrightarrow\mathcal{C}$ of the pullback
$y^{\ast}t\colon\hat{\mathcal{C}}_{/y}\twoheadrightarrow\mathcal{C}$.
A modulator $M$ on $\mathcal{C}$ is fiberwise left exact if each fiber $M(B)\subset\hat{\mathcal{C}}_{/yB}$ is closed 
under finite limits.

We recall that a class $K\subseteq\hat{\mathcal{C}}^{\Delta^1}$ of arrows is saturated if it contains all 
equivalences, and is both closed under composition and colimits. Every class $K\subseteq\hat{\mathcal{C}}^{\Delta^1}$ 
of arrows is contained in a smallest saturated class $\mathrm{Sat}(K)$ which will be referred to as the saturation of 
$K$. By \cite[Section 3.3]{as_soa}, the saturation $\mathrm{Sat}(M)$ of a modulator $M$ on $\hat{\mathcal{C}}$ is 
always closed under base change. A saturated class $K\subseteq\hat{\mathcal{C}}^{\Delta^1}$ of arrows is strongly 
saturated if it satisfies the 2-out-of-3 property \cite[Definition 2.2.4]{abfjsheavesI}. Every class 
$K\subseteq\hat{\mathcal{C}}^{\Delta^1}$ of arrows is contained in a smallest strongly saturated class which will be 
referred to as the strong saturation of $K$ \cite[Definition 2.2.6]{abfjsheavesI}. The following theorem is 
essentially \cite[Theorem 4.1.9]{abfjsheavesI}.

\begin{theorem}\label{thmlexlocgen}
Let $M$ be a modulator on $\mathcal{C}$ such that the set
\[\Delta(M):=\{\Delta(m)\colon X\rightarrow X\times_{yB}X\mid B\in\mathcal{C}, (m\colon X\rightarrow yB)\in M(B)\}\] 
of diagonals of maps in $M$ is contained in $\mathrm{Sat}(M)$. Then the accessible reflective localization
$\hat{\mathcal{C}}\rightarrow\hat{\mathcal{C}}[M^{-1}]$
is left exact.
\end{theorem}
\begin{proof}
Under the given assumption, \cite[Theorem 4.1.9]{abfjsheavesI} states that $\mathrm{Sat}(M)$ is 
left exact when considered as a full sub-$\infty$-category of $\hat{\mathcal{C}}^{\Delta^1}$. It follows that the 
saturation of $M$ coincides with the strong saturation of $M$. In 
particular, the latter is stable under base change. Thus, the localization
$\hat{\mathcal{C}}\rightarrow\hat{\mathcal{C}}[M^{-1}]=\hat{\mathcal{C}}[\mathrm{Sat}(M)^{-1}]$ is left exact by 
\cite[Theorem 4.2.10]{abfjsheavesI}. 
\end{proof}

\begin{notation*} 
All localizations considered in this paper are automatically reflective. Accordingly, throughout the rest of the 
paper, the term ``localization'' will implicitly refer to ``reflective localization''.
\end{notation*}

In the coming sections we will be interested in modulators of the form
\[\mathrm{Cov}_T(B)=\{\mathrm{colim}yU\rightarrow yB\mid U\in T(B)\}\]
for $B\in\mathcal{C}$ and suitable sets $T(B)$ of colimiting cocones $U\colon I\rightarrow\mathcal{C}_{/B}$. We 
first note that every modulator is of the form $\mathrm{Cov}_T$ for some set $T$ of diagrams.

\begin{lemma}\label{lemmastcolim}
Suppose $M=\{M(B)\mid B\in\mathcal{C}\}$ is a collection of classes of objects $M(B)\subset\hat{\mathcal{C}}_{/yB}$. 
For $B\in\mathcal{C}$ let
\[\mathrm{Un}[M](B):=\{\mathrm{Un}(m)\twoheadrightarrow\mathcal{C}_{/B}\mid m\in M(B)\}\]
be the class of right fibrations obtained via Unstraightening under the equivalence
$\hat{\mathcal{C}}_{/yB}\simeq\widehat{\mathcal{C}_{/B}}$. Then $M=\mathrm{Cov}_{\mathrm{Un}[M]}$.
\end{lemma}
\begin{proof}
We are to construct an equivalence
\[\mathrm{colim}(\mathrm{Un}(X)\twoheadrightarrow\mathcal{C}\xrightarrow{y}\hat{\mathcal{C}})\simeq X\] 
for all $\infty$-categories $\mathcal{C}$ and all presheaves $X\colon\mathcal{C}^{op}\rightarrow\mathcal{S}$. This 
will then in particular apply to the $\infty$-category $\mathcal{C}_{/B}$ and the presheaf
$m\in\widehat{\mathcal{C}_{/B}}$. 
Therefore, recall that the Unstraightening of 
$X\colon\mathcal{C}^{op}\rightarrow\mathcal{S}$ can be computed as the pullback of the universal
right fibration $\pi\colon\mathcal{S}_{\ast}^{op}\rightarrow\mathcal{S}^{op}$ along $X^{op}$
\cite[Section 3.3.2]{luriehtt}. We may thus consider the diagram
\[\xymatrix{
\mathcal{C}_{/X}\ar@{->>}[d]\ar[r]\ar@{}[dr]|(.3){\pbs} & \hat{\mathcal{C}}_{/X}\ar@{->>}[d]\ar[r]\ar@{}[dr]|(.3){\pbs} & \mathcal{S}_{\ast}^{op}\ar@{->>}[d]^{\pi} \\
\mathcal{C}\ar[r]_y & \hat{\mathcal{C}}\ar[r]_{yX} & \mathcal{S}^{op}
}\]
of right fibrations. Here, $\mathcal{S}$ is the $\infty$-category of large spaces (so $yX$ is $\mathcal{S}$-valued). 
The right hand side square is cartesian via \cite[Lemma 5.1.5.2]{luriehtt} and the fact that the Yoneda 
embedding is fully faithful. The vertical right fibration on the left hand side is the $\infty$-category
of elements of $X$ and defined so that the left hand side square is cartesian. The composition of the vertical two 
functors on the bottom is equivalent to $X^{op}$ itself by the Yoneda lemma. It follows that
$\mathrm{Un}(X)\simeq\mathcal{C}_{/X}$ over $\mathcal{C}$. In particular, it follows that
\[\mathrm{colim}(\mathrm{Un}(X)\twoheadrightarrow\mathcal{C}\xrightarrow{y}\hat{\mathcal{C}})\simeq\mathrm{colim}(\mathcal{C}_{/X}\twoheadrightarrow\mathcal{C}\rightarrow\hat{\mathcal{C}}).\]
The latter colimit returns $X$ precisely because the Yoneda embedding left Kan extends to the identity along itself 
\cite[Lemma 5.1.5.3]{luriehtt}. 
\end{proof}

Furthermore, recall that every accessible left exact localization of an $\infty$-topos factors through an essentially 
unique topological localization followed by a cotopological localization
\cite[Proposition 6.5.2.19, Remark 6.5.2.20]{luriehtt}. In order to understand the topological part of a left exact 
localization generated by some modulator $M$, it is useful to understand the associated Grothendieck topology in 
terms of the modulator $M$. Let us first recall the construction of $(-1)$-truncation via \v{C}ech-nerves which will 
be of relevance on multiple occasions in this paper.

Given an $\infty$-category $\mathcal{C}$ with pullbacks, the \emph{\v{C}ech-nerve} $\check{C}(f)$ of a map
$f\colon E\rightarrow B$ in $\mathcal{C}$ (if it exists) is given by the right Kan extension of the edge
$\{f\}\colon(\Delta^1)^{op}\rightarrow\mathcal{C}$ to the opposite of the category $\Delta_+$ of augmented simplicial 
sets along the (opposite of the) fully faithful inclusion $\Delta^1\hookrightarrow\Delta_+$, $i\mapsto i-1$ 
\cite[Section 6.1.2]{luriehtt}. Thus, $\check{C}(f)$ is an augmented simplicial object
$\check{C}(f)\colon\Delta_+^{op}\rightarrow\mathcal{C}$ which restricts to $f$ on degree $\leq 0$, together with 
equivalences 
\[\check{C}(f)_n\xrightarrow{\simeq}E\times_B E\dots\times_B E\]
induced by the points $[0]\rightarrow [n]$ for all $n\geq 1$.
The \v{Cech}-nerve $\check{C}(f)$ of a map $f$ in an $\infty$-category $\mathcal{C}$ plays the role of the kernel 
pair associated to a map in a 1-category. Whenever $\mathcal{C}$ is an $\infty$-topos for example (or more generally 
whenever $\mathcal{C}$ is regular, see Section~\ref{secregtop}), it will be used to compute the $(-1)$-truncation 
$f_{-1}\colon|\check{C}(f)|\hookrightarrow B$ of $f$ as the natural map out of the colimit of the underlying 
simplicial object of $\check{C}(f)$.

\begin{lemma}\label{propmodbase}
Let $\mathcal{C}$ be a small $\infty$-category and let $M$ be a modulator on $\mathcal{C}$.
\begin{enumerate}
\item The collection
\[M_{-1}(B):=\{f_{-1}\colon |\check{C}(f)|\hookrightarrow yB\mid f\in M(B)\}\]
of sieves obtained by $(-1)$-truncation of the maps in $M$ generates a Grothendieck topology $J$ whose sheaves 
are exactly the $(M_{-1})$-local objects.
\item The Grothendieck topology $J$ consists exactly of those monomorphisms with representable codomain which are 
contained in the saturation $\mathrm{Sat}(M)$. 
\end{enumerate}
\end{lemma}
\begin{proof}
For Part 1, since $(-1)$-truncation in $\hat{\mathcal{C}}$ is pullback-stable, the class $M_{-1}$ is a modulator which consists of 
monomorphisms. Hence, by \cite[Corollary 3.4.14]{as_soa}, the (topological) saturation $\mathrm{Sat}(M_{-1})$ that it 
generates is left exact. We therefore obtain a Grothendieck topology $J$ which consists of the maps in
$\mathrm{Sat}(M_{-1})$ with representable codomain. As both $M_{-1}$ and $J$ are generating sets of
$\mathrm{Sat}(M_{-1})$, the statement follows.
Part 2 follows directly from \cite[Proposition 4.1.14]{abfjsheavesII}.
\end{proof}

\begin{corollary}\label{cormodbase}
Let $\mathcal{C}$ be a small $\infty$-category, let $M$ be a modulator on $\mathcal{C}$ such that the localization
$\hat{\mathcal{C}}\rightarrow\mathrm{Sh}_M(\mathcal{C})$ at $M$ is left exact, and let $J$ be the associated 
Grothendieck topology from Lemma~\ref{propmodbase}.1. Then the factorization of the left exact localization
$\hat{\mathcal{C}}\rightarrow\mathrm{Sh}_{M}(\mathcal{C})$ into a topological localization followed by a 
cotopological localization is given by
\[\hat{\mathcal{C}}\rightarrow\mathrm{Sh}_{J}(\mathcal{C})\rightarrow\mathrm{Sh}_{M}(\mathcal{C}).\] 
\end{corollary}
\begin{proof}
This follows from Lemma~\ref{propmodbase}.2 along the lines of the proof of \cite[Proposition 6.5.2.19]{luriehtt}. 
Indeed, the only part left to show is that the latter localization is cotopological. Therefore we have to prove that 
whenever $f\colon X\hookrightarrow Y$ is an inclusion of $J$-sheaves which is mapped to 
an equivalence in $\mathrm{Sh}_{M}(\mathcal{C})$, then $f$ was an equivalence in $\mathrm{Sh}_J(\mathcal{C})$ 
already. 
Thus, given such an inclusion $f$ between $J$-sheaves, it follows that all pullbacks to representables of $f$ in
$\hat{\mathcal{C}}$ are inclusions which are each mapped to equivalences in $\mathrm{Sh}_{M}(\mathcal{C})$. That 
means, they are all elements of the strong saturation of $M$ \cite[Definition 2.2.6]{abfjsheavesI}, and hence in 
particular, elements of the saturation $\mathrm{Sat}(M)$ of $M$. Hence, they are contained in $J$ by 
Lemma~\ref{propmodbase}.2. It follows that all pullbacks to representables of $f$ are mapped to equivalences in
$\mathrm{Sh}_J(\mathcal{C})$, and hence so is $f$.
\end{proof}

%

\begin{remark}\label{remidmod}
In an earlier version of this paper the author defined the notion of an Id-modulator on an $\infty$-category
$\mathcal{C}$; that is, a modulator $M$ on $\mathcal{C}$ such that for every $m\colon X\rightarrow yB$ in $M(B)$ and 
every pair of sections $s_1,s_2$ of $m$, there is an equalizer $\mathrm{Equ}_{yC}(s_1,s_2)\rightarrow yB$ in
$\hat{\mathcal{C}}_{/yB}$ again contained in $M(B)$. Or in other words, such that for every $m\colon X\rightarrow yB$ 
in $M(B)$ the diagonal $\Delta(m)$ is locally again contained in $M$ (up to equivalence).
One can show that every modulator on $\mathcal{C}$ which is fiberwise left exact is an Id-modulator, that all higher 
diagonals of maps in an Id-modulator $M$ are again locally contained in $M$, and that a ``transitive'' modulator $M$ 
is an Id-modulator if and only if it is fiberwise left exact. 
In the meantime, the primary purpose of this definition (to prove Theorem~\ref{thmlexlocgen} for Id-modulators) has 
become obsolete in light of the stronger statement provided by \cite[Theorem 4.1.9]{abfjsheavesI}. Therefore, a 
formal discussion will be omitted. It may be worth pointing out however that the proof of
\cite[Theorem 4.1.9]{abfjsheavesI} is fairly long and intricate, and that Theorem~\ref{thmlexlocgen} is much easier 
to prove for Id-modulators $M$ directly. Indeed, one can show by way of a straight-forward recursive argument that 
the fiberwise finite limit closure $M^{\mathrm{lex}}(B)\subset\hat{\mathcal{C}}_{/yB}$ of an Id-modulator $M$ is 
contained in the saturation $\mathrm{Sat}(M)$. This implies that the latter is left exact by
\cite[Proposition 3.4.9]{as_soa}. Furthermore, all modulators we 
consider in this paper to generate $\infty$-toposes from are induced from well-structured colimit pre-topologies as 
to be introduced in Section~\ref{secsubstrcolimtop} and Section~\ref{secsubhcd}. Such modulators are always
Id-modulators (Remark~\ref{remwicolimittop}). Therefore, the notion of an Id-modulator may serve well both for 
expository as well as for alternative meta-theoretical propositions after all.
\end{remark}

We end this section with one more definition for deliberate use in the next three sections.

\begin{definition}
Let $\mathcal{C}$ be a locally small $\infty$-category. A left exact localization
$\hat{\mathcal{C}}\rightarrow\mathcal{E}$ is \emph{sub-canonical} if the Yoneda embedding
$y\colon\mathcal{C}\rightarrow\hat{\mathcal{C}}$ factors through the associated right adjoint inclusion
$\mathcal{E}\hookrightarrow\hat{\mathcal{C}}$.
\end{definition} 

\section{Higher covering diagrams and descent}\label{secdesc}

In this section we introduce a theory of higher covering diagrams (Section~\ref{secsubhcd}) which forms the core notion of this paper. To do so in due generality, we first relativize the descent conditions as usually imposed 
on an $\infty$-category to one diagram at a time in Section~\ref{secsubdesc}. In Section~\ref{secsubstrcolimtop} we 
propose a definition of a structured colimit pre-topology. This allows us to generate sheaf theories from suitable 
classes of diagrams directly rather than to repeatedly manipulate associated modulators.
In Section~\ref{secsubinftydesc} we define a stability condition on diagrams which gives rise to the notion of a 
higher covering diagram under one additional tameness assumption. In Section~\ref{secsubhcd} we then 
define higher covering diagrams. We show that they form a structured colimit pre-topology which presents the 
canonical sheaf theory over any small $\infty$-category $\mathcal{C}$ with pullbacks (up to a size caveat).

\subsection{Descent diagrams}\label{secsubdesc}

In this section we specify the notion of descent as defined in \cite{aneljoyaltopos} for a single diagram at a time, 
and further relativize it so to capture instances thereof in $\infty$-categories which do not have descent globally 
and which do not necessarily exhibit arbitrary finite limits or colimits in the first place. 

\begin{definition}\label{defpredescdiag}
Let $\mathcal{C}$ be an $\infty$-category and let $B\in\mathcal{C}$. A diagram $U\colon I\rightarrow\mathcal{C}_{/B}$ 
is
\begin{enumerate}
\item \emph{colimiting} if it is so when considered as a cocone from
$sU\colon I\rightarrow\mathcal{C}_{/B}\rightarrow\mathcal{C}$ to the object $B$, and
\item \emph{decomposable} if $\mathcal{C}$ has all pullbacks along the components $U_i\colon s(U_i)\rightarrow B$ for 
each $i\in I$.
\end{enumerate}
\end{definition}


For a given decomposable diagram $U\colon I\rightarrow\mathcal{C}_{/B}$ and a given arrow
$f\colon C\rightarrow B$ in an $\infty$-category $\mathcal{C}$ we can construct a base change
\begin{align}\label{equbasechange}
f^{\ast}U\colon I\rightarrow\mathcal{C}_{/C}
\end{align}
of $U$ as follows. Consider the full sub-$\infty$-category
$(\mathcal{C}_{/B})_f\subseteq\mathcal{C}_{/B}$ spanned by those objects $D\in\mathcal{C}_{/B}$ such that a pullback 
$f^{\ast}D\in\mathcal{C}_{/C}$ does exist. Then there is a base change functor
$f^{\ast}\colon(\mathcal{C}_{/B})_f\rightarrow\mathcal{C}_{/C}$ constructed as the right adjoint to the 
postcomposition $\Sigma_f\colon\mathcal{C}_{/C}\rightarrow\mathcal{C}_{/B}$ relative to the inclusion
$(\mathcal{C}_{/B})_f\hookrightarrow\mathcal{C}_{/B}$ (see e.g.\ \cite[Definition 2.23]{rs_ext} for a brief 
discussion of relative adjoints in this context). The diagram $U\colon I\rightarrow\mathcal{C}_{/B}$ factors through 
$(\mathcal{C}_{/B})_f$ by assumption, and hence gives rise to a functor (\ref{equbasechange}) via post-composition.

In fact, more generally, for every decomposable diagram $U\colon I\rightarrow\mathcal{C}_{/B}$ one can 
construct a product functor
\begin{align}\label{equpredescpbfunct}
U\times_B -\colon\mathrm{Fun}(I,\mathcal{C}_{/B})\rightarrow\mathrm{Fun}(I,\mathcal{C}_{/B})
\end{align}
which maps a diagram $V\colon I\rightarrow\mathcal{C}_{/B}$ and an object $i\in I$ to the fiber product
$U_i\times_C V_i$. Whenever $\mathcal{C}$ has pullbacks, this product functor exists formally 
because in this case $\mathrm{Fun}(I,\mathcal{C}_{/B})$ has all products. Otherwise, we may embed
$\mathcal{C}_{/B}$ in the left exact $\infty$-category $\hat{\mathcal{C}}_{/yB}$ via its Yoneda embedding. The 
restriction of the product functor
\[yU\times_{yB} -\colon\mathrm{Fun}(I,\hat{\mathcal{C}}_{/yB})\rightarrow\mathrm{Fun}(I,\hat{\mathcal{C}}_{/yB})\]
to $\mathrm{Fun}(I,\mathcal{C}_{/B})\hookrightarrow\mathrm{Fun}(I,\hat{\mathcal{C}}_{/yB})$ then factors to give a 
product functor (\ref{equpredescpbfunct}) in $\mathcal{C}$.

\begin{definition}\label{defpredescdiagstable}
A colimiting and decomposable diagram $U\colon I\rightarrow\mathcal{C}_{/B}$ is a \emph{pre-descent diagram} if for 
every morphism $f\colon C\rightarrow B$ the base change $f^{\ast}U\colon I\rightarrow\mathcal{C}_{/C}$ has a colimit.
\end{definition}

The base change $f^{\ast}U\colon I\rightarrow\mathcal{C}_{/C}$ of a pre-descent diagram
$U\colon I\rightarrow\mathcal{C}_{/B}$ is always colimiting over some base
$\mathrm{colim}f^{\ast}U\in\mathcal{C}_{/C}$ by definition. However it is not necessarily a pre-descent diagram again 
itself, because the canonical morphism $\mathrm{colim}f^{\ast}U\rightarrow C$ need not be an equivalence.

\begin{definition}
A pre-descent diagram $U\colon I\rightarrow\mathcal{C}_{/B}$ is a \emph{semi-descent diagram} if for every
$f\colon C\rightarrow B$ in $\mathcal{C}$ the base change $f^{\ast}U\colon I\rightarrow\mathcal{C}_{/C}$ is again a
pre-descent diagram.
\end{definition}

As iterated pullbacks compose, a pre-descent diagram $U\colon I\rightarrow\mathcal{C}_{/B}$ is a semi-descent 
diagram if and only if for every $f\colon C\rightarrow B$ the base change
$f^{\ast}U\colon I\rightarrow\mathcal{C}_{/C}$ is again colimiting. This in turn holds if and only if the colimit of 
$U$ is universal in the usual sense \cite[Section 6.1.1.(ii)]{luriehtt}. It is also easy to see that
$U\colon I\rightarrow\mathcal{C}_{/B}$ is a semi-descent diagram if and only if for every $f\colon C\rightarrow B$ 
the base change $f^{\ast}U\colon I\rightarrow\mathcal{C}_{/C}$ is again a semi-descent diagram. Thus, the notions of 
pullback-stable pre-descent diagram, semi-descent diagram, and pullback-stable semi-descent diagram all coincide.

Given any decomposable diagram $U\colon I\rightarrow\mathcal{C}_{/B}$, there is a canonical functor 
\begin{align}\label{equpredefpredescent}
\mathrm{res}_U\colon\mathcal{C}_{/B}\rightarrow\mathrm{Desc}(U)
\end{align}
of $\infty$-categories, where the codomain $\mathrm{Desc}(U)\subseteq\mathrm{Fun}(I,\mathcal{C}_{/B})_{/U}$ denotes 
the full sub-$\infty$-category spanned by the cartesian natural transformations over $U$. It maps an object
$f\colon C\rightarrow B$ to the cartesian natural transformation $\mathrm{res}_U(f)$ given pointwise by its 
associated pullbacks along the components $U_i\colon sU_i\rightarrow B$. It can be formally defined in this 
generality as the composition
\[\mathrm{res}_U\colon\mathcal{C}_{/B}\xrightarrow{\mathsf{C}}\mathrm{Fun}(I,\mathcal{C}_{/B})\xrightarrow{U\times_B -}\mathrm{Fun}(I,\mathcal{C}_{/B})_{/U}\]
via the product functor (\ref{equpredescpbfunct}) which is easily seen to factor through $\mathrm{Desc}(U)$.
%
It directly generalizes the construction of the same functor $\mathrm{res}_U$ from
\cite[Section 3.3.2]{aneljoyaltopos} in case $U$ is colimiting and $\mathcal{C}$ has all finite limits and colimits. 
This functor will be used to define descent in the obvious way. First however, to relativize the notion of descent to 
suitable classes of diagrams, we make the following additional definitions. Therefore, we refer to the definition of 
cofinal equivalence from Definition~\ref{defcofequiv}.

\begin{definition}\label{defsemidescentclass}
Let $\mathcal{C}$ be an $\infty$-category and $T=\{T(B)\mid B\in\mathcal{C}\}$ be a class of diagrams of type
$I\rightarrow\mathcal{C}_{/B}$ for $B\in\mathcal{C}$ and $I\in\mathrm{Cat}_{\infty}$. Say that
$T$ is \emph{cofinally stable (under base change)} at a diagram $U\colon I\rightarrow\mathcal{C}_{/B}$ 
in $T(B)$ if for all $f\colon C\rightarrow B$ in $\mathcal{C}$ there is a diagram contained in $T(C)$ that is 
cofinally equivalent to the base change $f^{\ast}U\colon I\rightarrow\mathcal{C}_{/C}$. Say that $T$ is 
\emph{cofinally stable (under base change) in $\mathcal{C}$} if it is cofinally stable at all diagrams in $T$. 
\end{definition}

\begin{notation}
Let $T=\{T(B)\mid B\in\mathcal{C}\}$ be class of diagrams as in Definition~\ref{defsemidescentclass}, and suppose 
$U\colon I\rightarrow\mathcal{C}_{/B}$ is a diagram. 
Let
\[\mathrm{Desc}_T(U)\subseteq\mathrm{Desc}(U)\]
denote the full sub-$\infty$-category spanned by those cartesian natural transformations $V\rightarrow U$ such that 
there is some object $C\rightarrow B$ together with a factorization $V\colon I\rightarrow\mathcal{C}_{/C}$ which is 
cofinally equivalent to some diagram in $T(C)$. 
\end{notation}

Whenever $T$ is cofinally stable under base change at a given diagram $U\colon I\rightarrow\mathcal{C}_{/B}$ in
$T(B)$, the functor $\mathrm{res}_U$ of (\ref{equpredefpredescent}) factors through $\mathrm{Desc}_T(U)$.
Clearly, every class $T$ of diagrams that is actually stable under base change at a diagram $U\in T$ is cofinally 
stable under base change at $U$. In particular, the class $\mathrm{Rex}(\mathcal{C})$ of colimiting diagrams in
$\mathcal{C}$ is cofinally stable at any given pre-descent diagram $U\colon I\rightarrow\mathcal{C}_{/B}$.
We obtain the following straight-forward generalization of Anel and Joyal's characterization of descent in 
\cite[Section 3.3.2]{aneljoyaltopos} .

\begin{lemma}\label{lemmaexostableadj}
Let $\mathcal{C}$ be an $\infty$-category and let $U\colon I\rightarrow\mathcal{C}_{/B}$ be a pre-descent diagram.
\begin{enumerate}
\item The functor
\[\mathrm{res}_U\colon\mathcal{C}_{/B}\rightarrow\mathrm{Desc}_{\mathrm{Rex}}(U)\] has a left adjoint
$\mathrm{glue}_U\colon\mathrm{Desc}_{\mathrm{Rex}}(U)\rightarrow\mathcal{C}_{/B}$.
\item $U$ is a semi-descent diagram if and only if the counit of this adjunction is an equivalence.
That is, if and only if the functor (\ref{equpredefpredescent}) is fully faithful.
\end{enumerate}
\end{lemma}
\begin{proof}
For Part 1 we are to show that for every cartesian natural transformation
$\alpha\colon V\rightarrow U$ in $\mathrm{Desc}_{\mathrm{Rex}}(U)$ the pullback
\[\xymatrix{
\alpha_{/\mathrm{res}_U}\ar@{->>}[d]\ar[r]\ar@{}[dr]|(.3){\pbs} & \mathrm{Desc}_{\mathrm{Rex}}(U)_{\alpha/}\ar@{->>}[d] \\
\mathcal{C}_{/B}\ar[r]_{\mathrm{res}_U} & \mathrm{Desc}_{\mathrm{Rex}}(U)
}\]
of $\infty$-categories has an initial object \cite[Proposition 6.1.11]{cisinskibook}. To do so, consider
$\mathrm{Desc}_{\mathrm{Rex}}(U)$ fully embedded in the slice $\mathrm{Fun}(I,\mathcal{C}_{/B})_{/U}$, so
$\mathrm{res}_U$ is the composition 
\[\mathcal{C}_{/B}\xrightarrow{\mathsf{C}}\mathrm{Fun}(I,\mathcal{C}_{/B})\xrightarrow{U\times_B -}\mathrm{Fun}(I,\mathcal{C}_{/B})_{/U}.\]
There is an intermediate homotopy-cartesian square
\[\xymatrix{
\mathrm{Fun}(I,\mathcal{C}_{/B})_{V/}\ar@{->>}[d]\ar[r] & (\mathrm{Fun}(I,\mathcal{C}_{/B})_{/U})_{\alpha/}\ar@{->>}[d] \\
\mathrm{Fun}(I,\mathcal{C}_{/B})\ar[r]_{U\times_B -} & \mathrm{Fun}(I,\mathcal{C}_{/B})_{/U}.
}\]
We obtain a resulting homotopy-cartesian square of the from
\[\xymatrix{
\alpha_{/\mathrm{res}_U}\ar@{->>}[d]\ar[r] & \mathrm{Fun}(I,\mathcal{C}_{/B})_{V/}\ar@{->>}[d] \\
\mathcal{C}_{/B}\ar[r]_(.4){\mathsf{C}} & \mathrm{Fun}(I,\mathcal{C}_{/B}).
}\]
It follows that $\alpha_{/\mathrm{res}_U}$ is equivalent to the $\infty$-category of cocones over
$V\colon I\rightarrow\mathcal{C}_{/B}$. Now, $V$ admits a colimit $\mathrm{colim}V\in\mathcal{C}_{/B}$ by assumption,  
and so the $\infty$-category $\alpha_{/\mathrm{res}_U}$ is equivalent to the under-category
$(\mathcal{C}_{/B})_{\mathrm{colim}V/}$. As such it has an initial object.

For Part 2, we just note that the counit of the resulting adjunction at an object $f\in\mathcal{C}_{/B}$ is the 
natural map $\mathrm{colim}f^{\ast}U\rightarrow f$ over $B$ induced by the cocone
$f^{\ast}U\colon I\rightarrow\mathcal{C}_{/\mathrm{dom}f}$. This cocone is colimiting for every object 
$f\in\mathcal{C}_{/B}$ exactly if $U$ is a semi-descent diagram. It follows that
$\mathrm{res}_U\colon\mathcal{C}_{/B}\rightarrow\mathrm{Desc}_{\mathrm{Rex}}(U)$ is fully faithful if and only if $U$ 
is a semi-descent diagram. As the inclusion $\mathrm{Desc}_{\mathrm{Rex}}(U)\hookrightarrow\mathrm{Desc}(U)$ is 
itself fully faithful, it follows that the push-forward
$\mathrm{res}_U\colon\mathcal{C}_{/B}\rightarrow\mathrm{Desc}(U)$ is fully faithful if and only if $U$ is a
semi-descent diagram. 
\end{proof}

\begin{definition}
Let $\mathcal{C}$ be an $\infty$-category and let $U\colon I\rightarrow\mathcal{C}_{/B}$ be a semi-descent diagram. 
The diagram $U$ is a \emph{descent diagram} if the functor
\begin{align*}
\mathrm{res}_U\colon\mathcal{C}_{/B}\rightarrow\mathrm{Desc}(U)
\end{align*}
is essentially surjective.
\end{definition}

Whenever $U\colon I\rightarrow\mathcal{C}_{/B}$ is a descent diagram, it follows that
$\mathrm{Desc}_{\mathrm{Rex}}(U)=\mathrm{Desc}(U)$, and so the left adjoint $\mathrm{glue}_U$ from 
Lemma~\ref{lemmaexostableadj}.1 is defined on all of $\mathrm{Desc}(U)$.

\begin{example}\label{explelogoi}
Let $\mathcal{C}$ be a cocomplete and finitely complete $\infty$-category. Then every colimiting diagram
$U\colon I\rightarrow\mathcal{C}_{/B}$ is a pre-descent diagram, and
$\mathrm{Desc}_{\mathrm{Rex}}(U)=\mathrm{Desc}(U)$. 
For such an $\infty$-category $\mathcal{C}$, all small colimiting diagrams $U\colon I\rightarrow\mathcal{C}_{/B}$ are
semi-descent diagrams if and only if all small colimits in $\mathcal{C}$ are universal. That is, if and only if the 
counit of the adjunction in Lemma~\ref{lemmaexostableadj}.1 is an equivalence for all small colimiting diagrams
$U\colon I\rightarrow\mathcal{C}_{/B}$. All such diagrams are descent diagrams in $\mathcal{C}$ if and only if the
$\infty$-category $\mathcal{C}$ has descent in the sense of \cite[Section 3.3.2]{aneljoyaltopos}.
\end{example}

\begin{definition}\label{defdescentclass}
Let $\mathcal{C}$ be an $\infty$-category and $T=\{T(B)\mid B\in\mathcal{C}\}$ be a class of diagrams of type
$I\rightarrow\mathcal{C}_{/B}$ for $B\in\mathcal{C}$ and $I\in\mathrm{Cat}_{\infty}$.
\begin{enumerate}
\item $T$ is a \emph{semi-descent class} if every diagram in $T$ is a semi-descent diagram and $T$ is
cofinally stable under base change in $\mathcal{C}$. 
\item $T$ is a \emph{descent class} if it is a semi-descent class and for all diagrams
$U\colon I\rightarrow\mathcal{C}_{/B}$ in $T$ the functor
\[\mathrm{res}_U\colon\mathcal{C}_{/B}\rightarrow\mathrm{Desc}_T(U)\]
is an equivalence. In that case we say that any such given $U$ has descent with respect to $T$.
\end{enumerate}
\end{definition}

Thus, in the terminology of Definition~\ref{defdescentclass} a diagram $U\colon I\rightarrow\mathcal{C}_{/B}$ is a 
descent diagram if and only if $U$ has descent with respect to the class of all small diagrams in $\mathcal{C}$.
The following reflection principle (barring the wording and some details) is an observation Mathieu Anel made the 
author aware of.

\begin{lemma}[Reflection of descent classes]\label{lemmadescentdiaganel}
Suppose $F\colon\mathcal{C}\rightarrow\mathcal{D}$ is a fully faithful pullback-preserving functor. 
\begin{enumerate}
\item Let $U\colon I\rightarrow\mathcal{C}_{/B}$ be a pre-descent diagram. If $F$ preserves the colimit of $U$ and 
the push-forward $FU\colon I\rightarrow\mathcal{D}_{/F(B)}$ is a semi-descent diagram, then $U$ was a semi-descent 
diagram already.
\item Suppose $T$ is a cofinally stable class of pre-descent diagrams, and the functor
$F$ preserves the colimit of all diagrams contained in $T$. If the image $F[T]$ is contained in a descent class $S$, 
then $T$ was a descent class already.
\end{enumerate}
\end{lemma}

\begin{proof}
Given a pre-descent diagram $U\colon I\rightarrow\mathcal{C}_{/B}$, we obtain a commutative square as follows.
\begin{align}\label{diaglemmadescentdiaganel}
\begin{gathered}
\xymatrix{
\mathcal{C}_{/B}\ar[r]^(.4){\mathrm{res}_U}\ar@{^(->}[d]_{F_{/B}} & \mathrm{Desc}(U)\ar@{^(->}[d]^{((F_{/B})_{\ast})_{/U}}\\
\mathcal{D}_{/F B}\ar[r]_(.4){\mathrm{res}_{F U}}& \mathrm{Desc}(F U)
}
\end{gathered}
\end{align}
The sliced embedding $F_{/B}\colon\mathcal{C}_{/B}\rightarrow\mathcal{D}_{/F B}$ is again fully faithful, and 
so is the push-forward
$(F_{/B})_{\ast}\colon\mathrm{Fun}(I,\mathcal{C}_{/B})\hookrightarrow\mathrm{Fun}(I,\mathcal{D}_{/F B})$. In 
particular, so is the sliced push-forward
$((F_{/B})_{\ast})_{/U}\colon\mathrm{Fun}(I,\mathcal{C}_{/B})_{/U}\hookrightarrow\mathrm{Fun}(I,\mathcal{D}_{/F B})_{/F U}$. As $F$ preserves pullbacks, both vertical functors in 
Diagram~(\ref{diaglemmadescentdiaganel}) are well-defined and fully faithful. The square commutes 
again by the assumption that the embedding $F$ preserves pullbacks.
It follows that fully faithfulness of $\mathrm{res}_{F U}$ implies fully faithfulness of $\mathrm{res}_{U}$.
This proves Part 1. Under the assumptions of Part 2, it follows from Part 1 that $T$ is a semi-descent class. For any given diagram $U\colon I\rightarrow\mathcal{C}_{/B}$ in $T$, consider the square
\[\xymatrix{
\mathcal{C}_{/B}\ar@{^(->}[d]_{F_{/B}} & \mathrm{Desc}_T(U)\ar@{^(->}[d]^{((F_{/B})_{\ast})_{/U}}\ar[l]_(.55){\mathrm{glue}_U}\\
\mathcal{D}_{/F C} & \mathrm{Desc}_{S}(F U),\ar[l]^(.55){\mathrm{glue}_{F U}}
}\]
where the horizontal arrows are the left adjoints of Lemma~\ref{lemmaexostableadj}.1 restricted to
$\mathrm{Desc}_T(U)$ and $\mathrm{Desc}_{S}(F U)$, respectively. The square commutes by assumption on $F$ together 
with Lemma~\ref{lemmacofequiv}. Thus, fully faithfulness of
$\mathrm{glue}_{F U}$ implies fully faithfulness of $\mathrm{glue}_U$ in the same way. 

\end{proof}

We end this section with the statement of various stability properties of the classes of semi-descent and descent 
diagrams in an $\infty$-category $\mathcal{C}$.

%

\begin{lemma}[Equivalences]\label{lemmadescdiagequiv}
Let $\mathcal{C}$ be an $\infty$-category, let $B\in\mathcal{C}$ be an object, and let
$U\colon I\rightarrow\mathcal{C}_{/B}$ be a diagram.
\begin{enumerate}
\item Suppose $\phi\colon J\rightarrow I$ is an equivalence of $\infty$-categories. Then
$U\colon I\rightarrow\mathcal{C}_{/B}$ is a pre-descent/semi-descent/descent diagram if and only if 
$U\phi\colon J\rightarrow\mathcal{C}_{/B}$ is so.
\item Suppose $f\colon C\rightarrow B$ is an equivalence in $\mathcal{C}$. Then
$U\colon I\rightarrow\mathcal{C}_{/B}$ is a pre-descent/semi-descent/descent diagram if and only if
$f^{\ast}U\colon I\rightarrow\mathcal{C}_{/C}$ is so.
\end{enumerate}
\end{lemma}
\begin{proof}
Straight-forward.
\end{proof}

\begin{lemma}[Terminal objects and base change]\label{lemmadescdiagbase}
Let $\mathcal{C}$ be an $\infty$-category.
\begin{enumerate}
\item For every $B\in\mathcal{C}$ the diagram 
$\{1_B\}\colon\Delta^0\rightarrow\mathcal{C}_{/B}$ is a descent diagram.
\item Descent diagrams are stable under base change in $\mathcal{C}$, i.e.\ whenever
$U\colon I\rightarrow\mathcal{C}_{/B}$ is a descent diagram then for every $f\colon C\rightarrow B$ the base change 
$f^{\ast}U\colon I\rightarrow\mathcal{C}_{/C}$ is again a descent diagram.
\end{enumerate}
\end{lemma}
\begin{proof}
Part 1 is straight-forward. 
For Part 2, suppose $U\colon I\rightarrow\mathcal{C}_{/B}$ is a descent diagram and let $f\colon C\rightarrow B$ be a 
morphism in $\mathcal{C}$. We have already noted that $f^{\ast}U\colon I\rightarrow\mathcal{C}_{/C}$ is a
semi-descent diagram, so we are left to show that the functor
\[\mathrm{res}_{f^{\ast}U}\colon\mathcal{C}_{/C}\rightarrow\mathrm{Desc}(f^{\ast}U)\]
is essentially surjective. Therefore let $\alpha\colon V\rightarrow f^{\ast}U$ be a cartesian natural transformation 
and let $\mathrm{res}_U(f)\colon f^{\ast}U\rightarrow U$ be the canonical cartesian natural transformation induced by 
$f$. The composition $\mathrm{res}_U(f)\circ\alpha\colon V\rightarrow U$ is again cartesian, and so we obtain a 
colimit $\mathrm{glue}_U(\mathrm{res}_U(f)\circ\alpha)\colon\mathrm{colim}V\rightarrow B$ of the push-forward
$\Sigma_f V\colon I\rightarrow\mathcal{C}_{/C}\rightarrow\mathcal{C}_{/B}$ in $\mathrm{C}_{/B}$. The cocone
$V\colon I\rightarrow\mathcal{C}_{/C}$ yields a map
$\mathrm{glue}_{f^{\ast}U}(V)\colon\mathrm{colim}V\rightarrow C$ over $B$. We obtain a diagram of natural transformations in $\mathrm{Fun}(I,\mathcal{C}_{/B})$ as follows.
\[\xymatrix{
V\ar@/_2pc/[ddr]_(.3){\alpha}|(.475)\hole|(.55)\hole\ar@/^2pc/[rrrd]\ar[dr] & & & \\
& \mathsf{C}(\mathrm{glue}_{U}(\Sigma_f V))^{\ast}U \ar@{}[drr]|(.3){\pbs}\ar[d]_{\mathrm{res}_{f^{\ast}U}(\mathrm{glue}_{f^{\ast}U}(V))}\ar[rr]^(.6){\mathrm{glue}_{U}(\Sigma_f V))^{\ast}U} & & \mathsf{C}(\mathrm{colim} V)\ar[d]_{\mathsf{C}(\mathrm{glue}_{f^{\ast}U}(V))}\ar@/^2pc/[dd]^{\mathsf{C}(\mathrm{glue}_{U}(\Sigma_f V)))} \\ 
 & \mathsf(f)^{\ast}U\ar[d]_{\mathrm{res}_U(f)}\ar[rr]^{f^{\ast}U}\ar@{}[drr]|(.3){\pbs} & & \mathsf{C}(C)\ar[d]_{\mathsf{C}(f)} \\
  & U\ar[rr] & & \mathsf{C}(1)
}\]
The natural transformation $V\rightarrow\mathsf{C}(\mathrm{glue}_{U}(\Sigma_f V))^{\ast}U$ is the source of the unit 
of the adjunction $\mathrm{glue}_U\adj \mathrm{res}_U$ and hence a natural equivalence. It follows that $\alpha$ is 
equivalent to $\mathrm{res}_{f^{\ast}U}(\mathrm{glue}_{U}(\Sigma_f V)))$, and so $\mathrm{res}_{f^{\ast}U}$ is 
essentially surjective.
\end{proof}

We will be particularly interested in $\infty$-categories $\mathcal{C}$ with pullbacks. In this case, we furthermore 
obtain straight-forward proofs of the following two closure properties.

\begin{lemma}[Dependent composition]\label{lemmadescdepcomp}
Suppose $\mathcal{C}$ has pullbacks and let $U\colon I\rightarrow\mathcal{C}_{/B}$ be a (semi-)descent diagram. Let 
$J\twoheadrightarrow I$ be a cocartesian fibration and let $(\mathcal{C}_{/B})_{/U}\twoheadrightarrow I$ be the 
cocartesian fibration obtained by pullback of the (cocartesian)
target fibration $t\colon(\mathcal{C}_{/B})^{\Delta^1}\twoheadrightarrow\mathcal{C}_{/B}$ along
$U\colon I\rightarrow\mathcal{C}_{/B}$. Suppose $V\colon J\rightarrow\mathcal{C}_{/sU}$ is a cocartesian functor over 
$I$ such that for all $i\in I$ the fiber $V_i\colon J_i\rightarrow\mathcal{C}_{/U_i}$ is a (semi-)descent diagram. 
Then the composition
\[J\xrightarrow{V}(\mathcal{C}_{/B})_{/U}\xrightarrow{U_{\ast}}\mathcal{C}_{/B}\]
is again a (semi-)descent diagram.
\end{lemma}
\begin{proof}
We first are to show that the composition $U_{\ast}V$ is a semi-descent diagram whenever $U$ is so and $V$ is so 
pointwise. The fact that it is colimiting (over $B$) follows from \cite[Proposition 4.3.3.10]{luriehtt}, as 
$J\twoheadrightarrow I$ is cocartesian and for every point $\{i\}\colon\Delta^0\rightarrow I$ the colimit of the 
restriction $V(i)\colon J(i)\rightarrow\mathcal{C}_{/B}$ is $U_i$ by assumption. Furthermore, given a map
$f\colon C\rightarrow B$, the diagram
\[\xymatrix{
J\ar[r]^V\ar[dr]_{f^{\ast}V} & \mathcal{C}_{/U}\ar[r]^{U_{\ast}}\ar[d]^{f^{\ast}} & \mathcal{C}_{/B}\ar[d]^{f^{\ast}} \\
 & \mathcal{C}_{/f^{\ast}U}\ar[r]_{f^{\ast}U_{\ast}} & \mathcal{C}_{/C} 
}\]
commutes. As $U$ is a semi-descent diagram, so is $f^{\ast}U$. And similarly, as $V$ is pointwise a semi-descent 
diagram, so is the composition $f^{\ast}V$. It follows that the base change $f^{\ast}U_{\ast}V$ is again colimiting 
over $C$, and so $U_{\ast}V$ is a semi-descent diagram. 

Now, assume $U$ is a descent diagram and $V$ is so pointwise. If $\mathcal{C}$ has pullbacks, the restriction functor
$\mathrm{res}_U\colon\mathcal{C}_{/B}\rightarrow\mathrm{Desc}(U)$ is the natural map
\[\mathrm{res}_U\colon\mathcal{C}_{/B}\rightarrow\underset{i\in I}{\mathrm{lim}}\mathcal{C}_{/U_i}\]
induced between limits as argued in the Appendix. Again via (the dual of) \cite[Proposition 4.3.3.10]{luriehtt}, 
one constructs the following factorization.
\[\xymatrix{
\mathcal{C}_{/B}\ar[r]^(.4){\mathrm{res}_U}\ar@/_1pc/[drr]_{\mathrm{res}_{U_{\ast}V}} & \underset{i\in I}{\mathrm{lim}}\mathcal{C}_{/U_i} \ar[r]^(.4){\underset{i\in I}{\mathrm{lim}}\mathrm{res}_{V_i}} & \underset{i\in I}{\mathrm{lim}}\underset{j\in J(i)}{\mathrm{lim}}\mathcal{C}_{/V_j}\ar@{}[d]_[@!-90]{\simeq}\\
 & &  \underset{j\in J}{\mathrm{lim}}\mathcal{C}_{/U_{\ast}V_j} 
}\]
By assumption, both $\mathrm{res}_U$ and all $\mathrm{res}_{V_i}$ are equivalences. As equivalences of
$\infty$-categories are closed under limits and composition, the statement follows.
\end{proof}

\begin{corollary}[Products]\label{cordescstabprod}
Suppose $\mathcal{C}$ has pullbacks. Whenever $U\colon I\rightarrow\mathcal{C}_{/B}$ and $V\colon I\rightarrow\mathcal{C}_{/B}$ are descent diagrams, then 
so is the product 
\[U\times_{B}V\colon I\times I\rightarrow\mathcal{C}_{/B}.\]
\end{corollary}
\begin{proof}
The diagram $U\times_{B}V$ is the dependent composition
\[I\times I\xrightarrow{U^{\ast}V}(\mathcal{C}_{/B})_{/U}\xrightarrow{U_{\ast}}\mathcal{C}_{/B}.\]
For every $i\in I$ the base change $U_{i}^{\ast}V\colon I\rightarrow(\mathcal{C}_{/B})_{/U_i}$ is a descent diagram 
by Lemma~\ref{lemmadescdiagbase}.2. The statement hence follows from Lemma~\ref{lemmadescdepcomp}.
\end{proof}

\begin{lemma}[Cofinal equivalence]\label{lemmadesccofstable}
Suppose $\mathcal{C}$ has pullbacks. Suppose $U\colon I\rightarrow\mathcal{C}_{/B}$ and
$V\colon J\rightarrow\mathcal{C}_{/B}$ are cofinally equivalent pre-descent diagrams. 
Then $U$ is a (semi-)descent diagram if and only if $V$ is a (semi-)descent diagram.
\end{lemma}
\begin{proof}
Let $W\colon K\rightarrow\mathcal{C}_{/B}$ be an extension of both $U$ and $V$ along a pair of cofinal functors
$\phi\colon I\rightarrow K$ and $\psi\colon J\rightarrow K$. If $\mathcal{C}$ has all pullbacks, then $W$ is again 
a pre-descent diagram. By Lemma~\ref{lemmaresfunct} we obtain a diagram 
\[\xymatrix{
\mathrm{Desc}(U)& \mathrm{Desc}(W)\ar[l]_{\mathrm{res}_{\phi}}^{\simeq}\ar[r]^{\mathrm{res}_{\psi}}_{\simeq} & \mathrm{Desc}(V)\\
 & \mathcal{C}_{/B}\ar[u]|{\mathrm{res}_W}\ar[ul]^{\mathrm{res}_U}\ar[ur]_{\mathrm{res}_V} & 
}\]
in $\mathrm{Cat}_{\infty}$ whose top horizontal arrows are equivalences. It follows that $\mathrm{res}_U$ is fully 
faithful (an equivalence) if and only if $\mathrm{res}_V$ is fully faithful (an equivalence).

\end{proof}

\subsection{Structured colimit pre-topologies}\label{secsubstrcolimtop}

In Section~\ref{secbases} we imposed conditions on the family of sets
\[\mathrm{Cov}_{T}(B):=\{\mathrm{colim}yU\in\hat{\mathcal{C}}_{/yB}| U\in T(B)\}\]
for sets $T$ of diagrams $U\colon I\rightarrow\mathcal{C}_{/B}$ to control the localization
$\hat{\mathcal{C}}\rightarrow\hat{\mathcal{C}}[\mathrm{Cov}_{T}^{-1}]$ at the union
$\mathrm{Cov}_{T}=\bigcup_{B\in\mathcal{C}}\mathrm{Cov}_{T}(B)$. 
In this section, we express conditions on the set $T$ itself to do so. Therefore, we first set up some preliminary 
technical constructions.
Lemma~\ref{lemmacolimittop1} introduces for every suitable diagram $U\colon I\rightarrow\mathcal{C}_{/B}$ in an
$\infty$-category $\mathcal{C}$ a natural transformation $\phi_U$ of span-extensions over
$U\times_B U\colon I\times I\rightarrow\mathcal{C}_{/B}$. Its components
$\phi_U(i,j)\in\mathcal{C}_{/U_i\times_B U_j}$ for pairs $i,j\in I$ will be required to be an equivalence in the 
definition of a higher covering diagram (Definition~\ref{defcovfunctors}). Albeit entirely formal, it is the central 
construction of Section~\ref{secdesc} and hence will be introduced in due detail in the course of the following three 
lemmata.


Thus, let $I$ be a small $\infty$-category and consider the free span $D^1= (1\leftarrow 0\rightarrow 2)$ in the 
category of simplicial sets. The simplicial set $D^1$ is identical to the outer horn $\Lambda^2_0$; the 
alternate choice of notation will become evident in Section~\ref{secsubinftydesc}. The $\infty$-category $I^{D^1}$ of 
spans in $I$ fits into a cartesian square of the following form.
\begin{align}\label{diaglambda}
\begin{gathered}
\xymatrix{
I^{D^1}\ar[r]^(.45){\lambda_{I}}\ar@{->>}[d]_{(\mathrm{ev}_0,\mathrm{ev}_{(1,2)})}\ar@{}[dr]|(.3){\pbs} & (I^{\partial\Delta^1})^{\Delta^1}\ar@{->>}[d]^{(s,t)}\\
I\times I^{\partial\Delta^1}\ar[r]_(.45){(\mathsf{C},1)} & I^{\partial\Delta^1}\times I^{\partial\Delta^1}
}
\end{gathered}
\end{align}
The functor $\lambda_I$ simply maps a span $i\leftarrow k\rightarrow j$ to the arrow of pairs
$(k,k)\rightarrow (i,j)$.

\begin{lemma}\label{lemmalambda}
For every small $\infty$-category $I$ and every cocomplete $\infty$-category $\mathcal{C}$ the natural transformation 
$\lambda_I$ induces a 2-cell
\begin{align}\label{diagapp1}
\begin{gathered}
\xymatrix{
\mathrm{Fun}(I^{\partial\Delta^1},\mathcal{C})\ar@{=}[d]\ar[r]^(.55){\mathsf{C}^{\ast}} & \mathrm{Fun}(I,\mathcal{C})\ar[d]^{\mathrm{ev}_0^{\ast}}\ar@{}[dl]|{\Downarrow\mu}\\
\mathrm{Fun}(I^{\partial\Delta^1},\mathcal{C})\ar[r]_(.55){\mathrm{ev}_{(1,2)}^{\ast}} & \mathrm{Fun}(I^{D^1},\mathcal{C})\\
}
\end{gathered}
\end{align}
in the $(\infty,2)$-category $\mathrm{Cat}_{\infty}$ of $\infty$-categories that satisfies the Beck-Chevalley 
condition with respect to the associated global left Kan extensions. 
\end{lemma}
\begin{proof}
First, for a given cocomplete $\infty$-category $\mathcal{C}$, the global left Kan extensions along both
$\mathrm{ev}_0\colon I^{D^1}\rightarrow I$ and $\mathrm{ev}_{(1,2)} \colon I^{D^1}\rightarrow I^{\partial\Delta^1}$ 
exist by \cite[Proposition 4.3.3.10]{luriehtt}. 
The fact that the square (\ref{diaglambda}) is cartesian equivalently means that the associated 2-cell
\[\xymatrix{
I^{D^1}\ar[r]^{\mathrm{ev}_{(1,2)}}\ar[d]_{\mathrm{ev}_0}\ar@{}[dr]|{\underset{\lambda_{I}}{\Rightarrow}} & I^{\partial\Delta^1}\ar@{=}[d]\\
I\ar[r]_{\mathsf{C}} & I^{\partial\Delta^1}
}\]
is a comma square in the cartesian closed $\infty$-cosmos $\mathrm{Cat}_{\infty}$ \cite[Definition 2.3.1]{rvkanext}. 
It hence induces a 2-cell $\mu$ as in (\ref{diagapp1}) that satisfies the according Beck-Chevalley condition by
\cite[Proposition 5.3.9]{rvkanext}.
\end{proof}

Explicitly, the $2$-cell $\mu$ in Lemma~\ref{lemmalambda} is given by the composition $\mu$ in the square
\begin{align}\label{diagmu}
\begin{gathered}
\xymatrix{
\mathrm{Fun}((I^{\partial\Delta^1})^{\Delta^1},\mathcal{C}^{\Delta^1})\ar[rr]^{\lambda_I^{\ast}} & &  \mathrm{Fun}(I^{D^1},\mathcal{C}^{\Delta^1})\ar@{->>}[d]^{(s,t)}\\
\mathrm{Fun}(I^{\partial\Delta^1},\mathcal{C})\ar[rr]_(.4){(\mathrm{ev}_0^{\ast}\circ \mathsf{C}^{\ast},\mathrm{ev}_{(1,2)}^{\ast})}\ar[urr]^(.4){\mu}\ar[u]^{(\cdot)^{\Delta^1}} & & \mathrm{Fun}(I^{D^1},\mathcal{C})\times\mathrm{Fun}(I^{D^1},\mathcal{C}).
}
\end{gathered}
\end{align}


Furthermore, if $\mathcal{C}$ has finite products, there is a product functor $-\times -\colon\mathcal{C}\times\mathcal{C}\rightarrow\mathcal{C}$ together with an additional homotopy-cartesian square
\[\xymatrix{
\mathcal{C}^{D^1}\ar[r]^{\pi}\ar@{->>}[d]_{(\mathrm{ev}_0,\mathrm{ev}_{(1,2)})} & \mathcal{C}^{\Delta^1}\ar@{->>}[d]^{(s,t)} \\
\mathcal{C}\times\mathcal{C}^{\partial\Delta^1}\ar[r]_(.55){(1,-\times -)}& \mathcal{C}^{\partial\Delta^1}
}\]
where the top arrow $\pi$ maps a span of the form $C\leftarrow A\rightarrow B$ to the universal map
$A\rightarrow B\times C$. The composite natural transformation
\[\Delta\colon\mathcal{C}\xrightarrow{\mathsf{C}}\mathcal{C}^{D^1}\xrightarrow{\pi}\mathcal{C}^{\Delta^1}\]
takes an object $C\in\mathcal{C}$ to its diagonal $\Delta(C)\colon C\rightarrow C\times C$. 
For any $\infty$-category $I$, the push-forward with $\Delta\colon\mathcal{C}\rightarrow\mathcal{C}^{\Delta^1}$ 
induces a 2-cell
\[\xymatrix{
\mathrm{Fun}(I,\mathcal{C})\ar[rr]^{-\times -\circ(-)^{\partial\Delta^1}}\ar@{=}[d]\ar@{}[drr]|{\underset{\Delta_{\ast}}{\Rightarrow}} & & \mathrm{Fun}(I^{\partial\Delta^1},\mathcal{C})\ar[d]^{\mathsf{C}^{\ast}} \\
\mathrm{Fun}(I,\mathcal{C})\ar@{=}[rr] & & \mathrm{Fun}(I,\mathcal{C}),
}\]
which at a given functor $U\colon I\rightarrow\mathcal{C}$ is the natural transformation from $U$ to
$\mathsf{C}^{\ast}(U\times U)$ given at $i\in I$ by the according diagonal $U_i\rightarrow U_i\times U_i$.

\begin{lemma}\label{lemmapasting}
Let $\mathcal{C}$ and $I$ be $\infty$-categories, and suppose $\mathcal{C}$ has finite products.
The pasting $\mathrm{paste}(\mu,\Delta_{\ast}):=(-\times -\circ(-)^{\partial\Delta^1})^{\ast}\mu\circ(\mathrm{ev}_0^{\ast})_{\ast}\Delta_{\ast}$ of the
2-cells
\begin{align}\label{diagpasting}
\begin{gathered}
\xymatrix{
\mathrm{Fun}(I,\mathcal{C})\ar[rr]^{-\times -\circ(-)^{\partial\Delta^1}}\ar@{=}[d]\ar@{}[drr]|{\underset{\Delta_{\ast}}{\Rightarrow}} & & \mathrm{Fun}(I^{\partial\Delta^1},\mathcal{C})\ar[d]^{\mathsf{C}^{\ast}}\ar@{=}[r]\ar@{}[dr]|{\underset{\mu}{\Rightarrow}}  & \mathrm{Fun}(I^{\partial\Delta^1},\mathcal{C})\ar[d]^{\mathrm{ev}_{(1,2)}^{\ast}} \\
\mathrm{Fun}(I,\mathcal{C})\ar@{=}[rr] & & \mathrm{Fun}(I,\mathcal{C})\ar[r]_{\mathrm{ev}_0^{\ast}} & \mathrm{Fun}(I^{D^1},\mathcal{C}),
}
\end{gathered}
\end{align}
computes pointwise the composition
\[\mathrm{Fun}(I,\mathcal{C})\xrightarrow{(-)^{D^1}}\mathrm{Fun}(I^{D^1},\mathcal{C}^{D^1})\xrightarrow{\pi_{\ast}}\mathrm{Fun}(I^{D^1},\mathcal{C}^{\Delta^1}).\]
\end{lemma}
\begin{proof}
As limits in functor $\infty$-categories are computed pointwise, the $\infty$-category $\mathrm{Fun}(I,\mathcal{C})$ 
has finite products. Given a diagram $U\colon I\rightarrow\mathcal{C}$, the push-forward
$\Delta_{\ast}U\colon U\rightarrow\mathsf{C}^{\ast}(U\times U)$ is exactly the diagonal of $U$ in the
$\infty$-category $\mathrm{Fun}(I,\mathcal{C})$. It follows that the natural transformation
$\mathrm{ev}_0^{\ast}(\Delta_{\ast}(U))\simeq\mathrm{ev}_0^{\ast}(\Delta(U))$ is the diagonal of $\mathrm{ev}_0^{\ast}(U)$ in $\mathrm{Fun}(I^{D^1},\mathcal{C})$.

Analogously, the functor $\infty$-category $\mathrm{Fun}(I^{D^1},\mathcal{C})$ has finite products, and
the natural transformation $\pi_{\ast}(U^{D^1})$ is the value of
$\pi\colon\mathrm{Fun}(I^{D^1},\mathcal{C})^{D^1}\rightarrow\mathrm{Fun}(I^{D^1},\mathcal{C})^{\Delta^1}$ at the span 
$U[D^1]:=(U\mathrm{ev}_1\leftarrow U\mathrm{ev}_0\rightarrow U\mathrm{ev}_2)$. We are hence to show that the triangle
\[\xymatrix{
U\mathrm{ev}_0\ar@/^/[rr]^{\Delta(U\mathrm{ev}_0)}\ar@/_/[dr]_(.4){\pi(U[D^1])} & & U\mathrm{ev}_0\times U\mathrm{ev}_0\ar@/^/[dl]^(.4){\mu(U\times U)} \\
 & U\mathrm{ev}_1\times U\mathrm{ev}_2 & 
}\]
in $\mathrm{Fun}(I^{D^1},\mathcal{C})$ commutes. By definition of $\pi(U[D^1])$, it suffices to show that it does so 
after postcomposition with the projections $U\mathrm{ev}_1\times U\mathrm{ev}_2\rightarrow U\mathrm{ev}_i$. The fact 
that the composition
\[U\mathrm{ev}_0\xrightarrow{\Delta(U\mathrm{ev}_0)}U\mathrm{ev}_0\times U\mathrm{ev}_0\xrightarrow{\mu(U\times U)}U\mathrm{ev}_1\times U\mathrm{ev}_2\rightarrow U\mathrm{ev}_i\]
is exactly the canonical map $U(\mathrm{ev}_{0}\rightarrow\mathrm{ev}_{i})$ follows from unfolding the definition of $\mu$ in terms of $\lambda_I$ in (\ref{diagmu}), from naturality of $\lambda$ in (\ref{diaglambda}) with regards to 
diagrams $U\colon I\rightarrow\mathcal{C}$, and the definition of $\lambda_{\mathcal{C}}$.

%
\end{proof}


Furthermore, whenever $\mathcal{C}$ is both cocomplete and has finite products, 
the composite cocone
\begin{align}\label{equapp1}
I\ast\Delta^0\xrightarrow{\mathsf{C}\ast\Delta^0}I^{\partial\Delta^1}\ast\Delta^0\xrightarrow{\mathrm{colim}(U\times U)}\mathcal{C}
\end{align}
induces a natural map 
\[\gamma\colon\mathrm{colim}_I(\mathsf{C}^{\ast}(U\times U))\rightarrow\mathrm{colim}_{I\times I} (U\times U)\]
in $\mathcal{C}$.

\begin{notation}\label{notationextfun}
Given an $\infty$-category $\mathcal{C}$, a functor $F\colon I\rightarrow J$ of $\infty$-categories, and a 
diagram $U\colon I\rightarrow \mathcal{C}$, we refer to the homotopy-fiber
\[\xymatrix{
\mathrm{Fun}_U(J,\mathcal{C})\ar[r]\ar[d]\ar@{}[dr]|(.3){\pbs} & \mathrm{Fun}(J,\mathcal{C})\ar[d]^{F^{\ast}} \\  
\ast\ar[r]_(.4){\ulcorner U\urcorner} & \mathrm{Fun}(I,\mathcal{C})
}\]
as the $\infty$-category of $U$-lifts to $J$ along $F$.
\end{notation}

\begin{lemma}\label{lemmacolimittop1}
Let $U\colon I\rightarrow\mathcal{C}_{/B}$ be a small decomposable diagram such 
that for all $i,j\in I$ the diagram
\begin{align}\label{equapp3}
\mathrm{Fun}_{(i,j)}(D^1,I)\xrightarrow{U_{\ast}}\mathrm{Fun}_{(U_i,U_j)}(D^1,\mathcal{C}_{/B})\xrightarrow[\simeq]{\pi|_{(U_i,U_j)}}\mathcal{C}_{/U_i\times_B U_j}
\end{align}
has a colimit in $\mathcal{C}_{/B}$. Then there is a natural transformation
\[\phi_U\colon\mathrm{Lan}_{\mathrm{ev}_{(1,2)}}(\mathrm{ev}_0^{\ast}(U))\rightarrow U\times_B U\]
which can be computed pointwise as the natural map
\begin{align}\label{equapp2}
\phi_U(i,j)\colon
\underset{i\xleftarrow{\alpha} k\xrightarrow{\beta} j}{\mathrm{colim}}U_k
\xrightarrow[\underset{i\xleftarrow{\alpha} k\xrightarrow{\beta} j}{\mathrm{colim}}(U\alpha,U\beta)]{}
U_i\times_B U_j
\end{align}
that represents the cocone (\ref{equapp3}). It comes together with a 2-cell in $\mathcal{C}_{/B}$ of the form
\begin{align}\label{equapp4}
\begin{gathered}
\xymatrix{
\mathrm{colim}_{I\times I}\left(\mathrm{Lan}_{\mathrm{ev}_{(1,2)}}(\mathrm{ev}_0^{\ast}(U))\right)\ar[r]^(.7){\simeq}\ar@/_1pc/[ddr]_{\mathrm{colim}_{I\times I}\phi_U\mbox{ }} & \mathrm{colim}_I U\ar[d]^{\mathrm{colim}_I\Delta_{\ast}(U)} \\
 & \mathrm{colim}_I(\mathsf{C}^{\ast}(U\times_B U))\ar[d]^{\gamma} \\
 & \mathrm{colim}_{I\times I}(U\times_B U)
}
\end{gathered}
\end{align}
whenever all colimits therein exist. Furthermore, whenever the canonical map
$\mathrm{colim}_{I\times I}(U\times_B U)\rightarrow\mathrm{colim}_IU\times_B\mathrm{colim}_IU$ is an equivalence, the 
vertical composition $\gamma\circ\mathrm{colim}_I \Delta_{\ast}(U)$ is equivalent to the diagonal of
$\mathrm{colim}_IU$.
\end{lemma}

\begin{proof}
Let us first assume that $\mathcal{C}$ is cocomplete and has finite products, and that $B$ is the terminal object in $\mathcal{C}$. We thus may identify $\mathcal{C}_{/B}\simeq\mathcal{C}$.
Consider the following pasting diagram in the $\infty$-cosmos $\mathrm{Cat}_{\infty}$ of small $\infty$-categories.
\begin{align}\label{diagapp4}
\begin{gathered}
\xymatrix{
\mathrm{Fun}(I^{\partial\Delta^1},\mathcal{C})\ar@{=}[r]\ar@/_/[ddr]_(.2){\mathsf{C}^{\ast}}\ar@/_/[dr]^{\mathsf{C}^{\ast}}\ar@{}[dr]|{\hspace{1.8cm}\overset{\Rightarrow}{\epsilon_{\mathsf{C}}}} & \mathrm{Fun}(I^{\partial\Delta^1},\mathcal{C})\ar@{=}[rr]\ar@{=}@/_/[dr]\ar@/_/[ddr]|(.45)\hole|(.5)\hole|(.55)\hole_(.2){\mathsf{C}^{\ast}} & & \mathrm{Fun}(I^{\partial\Delta^1},\mathcal{C})\ar@{=}@/_/[dr]\ar@{=}@/_/[ddr]|(.45)\hole & \\
 & \mathrm{Fun}(I,\mathcal{C})\ar[r]_(.45){\mathrm{Lan}_{\mathsf{C}}}\ar@{=}[d] & \mathrm{Fun}(I^{\partial\Delta^1},\mathcal{C})\ar@{=}[r]\ar[d]^{\mathsf{C}^{\ast}} & \mathrm{Fun}(I^{\partial\Delta^1},\mathcal{C})\ar@{=}[r]\ar[d]^{\mathrm{ev}_{(1,2)}^{\ast}} & \mathrm{Fun}(I^{\partial\Delta^1},\mathcal{C})\ar@{=}[d] \\
 & \mathrm{Fun}(I,\mathcal{C})\ar@{=}[r]\ar@{}[ur]|{\overset{\Rightarrow}{\eta_{\mathsf{C}}}} & \mathrm{Fun}(I,\mathcal{C})\ar[r]_{\mathrm{ev}_0^{\ast}}\ar@{}[ur]|{\overset{\Rightarrow}{\mu}} & \mathrm{Fun}(I^{D^1},\mathcal{C})\ar[r]_{\mathrm{Lan}_{\mathrm{ev}_{}(1,2)}}\ar@{}[ur]|{\overset{\Rightarrow}{\epsilon_{(1,2)}}} & \mathrm{Fun}(I^{\partial\Delta^1},\mathcal{C})
}
\end{gathered}
\end{align}
The left half of the diagram is filled by the 3-cell given by the triangle identity of the adjunction
$\mathrm{Lan}_{\mathsf{C}}\adj\mathsf{C}^{\ast}$. The right half of the diagram is filled by the obvious degenerate 
3-cell associated to the 2-cell $\epsilon_{(1,2)}\circ\mu$.

The composition of the front square 2-cells $\kappa:=\epsilon_{(1,2)}\circ\mu\circ\eta_{\mathsf{C}}$ is 
the mate of $\mu$ in (\ref{diagapp1}), and hence is a natural equivalence by virtue of the associated 
Beck-Chevalley condition (Lemma~\ref{lemmalambda}). We obtain two homotopy-commutative squares in
$\mathrm{Fun}(I^{\partial\Delta^1},\mathcal{C})$ as follows.
\begin{align}\label{diagapp4.2}
\begin{gathered}
\xymatrix{
\mathrm{Lan}_{\mathrm{ev}_{(1,2)}}(\mathrm{ev}_0^{\ast}(U))\ar[rr]^(.55){\simeq}_(.55){\kappa_U}\ar[d]_{\mathrm{Lan}_{\mathrm{ev}_{(1,2)}}(\mathrm{ev}_0^{\ast}(\Delta_{\ast}(U)))} & & \mathrm{Lan}_{\mathsf{C}}(U)\ar[d]^{\mathrm{Lan}_{\mathsf{C}}(\Delta_{\ast}(U))} \\
\mathrm{Lan}_{\mathrm{ev}_{(1,2)}}(\mathrm{ev}_0^{\ast}(\mathsf{C}^{\ast}(U\times U)))\ar[rr]^(.57){\simeq}_(.57){\kappa_{\mathsf{C}^{\ast}(U\times U)}}\ar[d]_{\mathrm{Lan}_{\mathrm{ev}_{(1,2)}}(\mu_{U\times U})} & & \mathrm{Lan}_{\mathsf{C}}(\mathsf{C}^{\ast}(U\times U))\ar[d]^{\epsilon_{\mathsf{C}}(U\times U)} \\
\mathrm{Lan}_{\mathrm{ev}_{(1,2)}}(\mathrm{ev}_{(1,2)}^{\ast}(U\times U))\ar[rr]_(.65){\epsilon_{(1,2)}(U\times U)} & & U\times U 
}
\end{gathered}
\end{align}

The upper square commutes, because
$\kappa\colon\mathrm{Fun}(I,\mathcal{C})\rightarrow\mathrm{Fun}(I^{\partial\Delta^1},\mathcal{C}^{\Delta^1})$ is a 
natural equivalence. The 2-cell making the bottom square commute comes from the composite 3-cell (\ref{diagapp4}) 
applied to $U\times U\in\mathrm{Fun}(I^{\partial\Delta^1},\mathcal{C})$. The bottom-left composite defines the 
natural transformation
\[\phi_U\colon\mathrm{Lan}_{\mathrm{ev}_{(1,2)}}(\mathrm{ev}_0^{\ast}(U))\rightarrow U\times U .\]
Let us first construct the associated 2-cell in (\ref{equapp4}) of colimits in $\mathcal{C}$. Thus, push-forward of 
(\ref{diagapp4.2}) along the colimit functor
$\mathrm{colim}\colon\mathrm{Fun}(I^{\partial\Delta^1},\mathcal{C})\rightarrow\mathcal{C}$ yields the (left hand side 
triangle of the) following diagram in $\mathcal{C}$.
\[\xymatrix{
\mathrm{colim}(\mathrm{Lan}_{\mathrm{ev}_{(1,2)}}(\mathrm{ev}_0^{\ast}(U)))\ar[r]^(.55){\simeq}_(.55){\mathrm{colim}\kappa_U}\ar@/_1pc/[ddr]_{\mathrm{colim \phi_U}} & \mathrm{colim}(\mathrm{Lan}_{\mathsf{C}}U)\ar[d]^{\mathrm{colim}(\mathrm{Lan}_{\mathsf{C}}(\Delta_{\ast}(U)))} \ar[r]^{\simeq} & \mathrm{colim}U\ar[d]^{\mathrm{colim}(\Delta_{\ast}(U))}\\
 & \mathrm{colim}(\mathrm{Lan}_{\mathsf{C}}(\mathsf{C}^{\ast}(U\times U)))\ar[d]^{\mathrm{colim}(\epsilon_{\mathsf{C}}(U\times U))}\ar[r]^(.55){\simeq} & \mathrm{colim}(\mathsf{C}^{\ast}(U\times U))\ar@/^1pc/[dl]^(.4){\gamma}  \\
 & \mathrm{colim}(U\times U)  &
}\]
Here, the upper right square comes from the fact that global left Kan extensions commute with the respective colimit 
functors. The lower right triangle commutes by definition of $\gamma$, as the vertical arrow
$\mathrm{colim}(\epsilon_{\mathsf{C}}(U\times U))$ represents exactly to the cocone (\ref{equapp1}).

Second, to compute $\phi_U$ pointwise, for $(i,j)\in I\times I$ consider the following diagram of
$\infty$-categories.
\begin{align}\label{diagapp6}
\begin{gathered}
\xymatrix{
 & & \mathcal{C}_{/U_i\times U_j}\ar@{->>}@{^(->}[r] & \mathcal{C}^{\Delta^1}\ar[dd]^{s}\\
 & \mathrm{Fun}_{(U_i,U_j)}(D^1,\mathcal{C})\ar@{^(->}[r]\ar@/^/[ur]_{\pi|_{(U_i,U_j)}}^{\simeq} & \mathcal{C}^{D^1}\ar@/^/[dr]^{\mathrm{ev}_0}\ar@/^/[ur]_{\pi} & \\
\mathrm{Fun}_{(i,j)}(D^1,I)\ar@{^(->}[r]^{\iota}\ar@/^/[ur]^(.35){U_{\ast}}\ar@{->>}[d]\ar@{}[dr]|(.3){\pbs} & I^{D^1}\ar[r]^{\mathrm{ev}_0}\ar@{->>}[d]_{\mathrm{ev}_{(1,2)}}\ar@/^/[ur]^(.35){U_{\ast}}\ar@{}[dr]|{\hspace{3pc}\Downarrow\mathrm{paste}(\mu,\Delta_{\ast})} & I\ar[r]^U & \mathcal{C} \\
\ast\ar[r]_{\ulcorner(i,j)\urcorner} & I^{\partial\Delta^1}\ar@/_2pc/[urr]_{U\times U} & & 
}
\end{gathered}
\end{align}
The 2-cell $\phi_U\colon\mathrm{Lan}_{(1,2)}(\mathrm{ev}_0^{\ast}U)\rightarrow U\times U$ is by definition the 
transpose of the pasting $\mathrm{paste}(\mu,\Delta_{\ast})(U)$.
It follows that the precomposition
\[\ulcorner\phi_U(i,j)\urcorner\simeq\ulcorner(i,j)\urcorner^{\ast}\phi_U\colon\mathrm{colim}((\mathrm{ev}_0\iota)^{\ast}U)\rightarrow U_i\times U_j\]
is the transpose of the cocone $\iota^{\ast}(\mathrm{paste}(\mu,\Delta_{\ast}))(U)\colon(\mathrm{ev}_0\iota)^{\ast}U\rightarrow U_i\times U_j$. 
The 2-cell $\mathrm{paste}(\mu,\Delta_{\ast})$ itself is equivalent to the composition $\pi\circ U_{\ast}$ on the 
upper half of Diagram~(\ref{diagapp6}) by Lemma~\ref{lemmapasting}. 
Hence, the cocone $\iota^{\ast}(\mathrm{paste}(\mu,\Delta_{\ast}))(U)$ is equivalent to the restriction
\[\pi|_{(U_i,U_j)}\circ U_{\ast}\colon\mathrm{Fun}_{(i,j)}(D^1,I)\rightarrow\mathcal{C}_{/U_i\times U_j}.\]
This finishes the proof in case $\mathcal{C}=\mathcal{C}_{/B}$ is cocomplete and has finite products.

Now, suppose $U\colon I\rightarrow\mathcal{C}_{/B}$ is a decomposable diagram respective a general $\infty$-category
$\mathcal{C}$ such that the diagrams $\pi|_{(U_i,U_j)}\circ U_{\ast}$ have a colimit for all $i,j\in I$. Localization 
of the presheaf $\infty$-category $\widehat{\mathcal{C}_{/B}}$ at the set
\[\{\mathrm{colim}(y\pi|_{(U_i,U_j)}\circ U_{\ast})\rightarrow y(\mathrm{colim}(\pi|_{(U_i,U_j)}\circ U_{\ast}))\mid i,j\in I\}.\]
yields a presentable $\infty$-category $\mathcal{D}$ together with a left exact fully faithful inclusion
$y\colon\mathcal{C}\hookrightarrow\mathcal{D}$. We obtain a natural transformation 
\[\phi_{yU}\colon\mathrm{Lan}_{\mathrm{ev}_{(1,2)}}(\mathrm{ev}_0^{\ast}(yU))\rightarrow yU\times_{yB} yU\]
pointwise computed as in (\ref{equapp2}) together with a 2-cell (\ref{equapp4}) in $\mathcal{D}$ by the above.
As every component of the natural transformation $\phi_{yU}$ lies in the 
essential image of $y\colon\mathcal{C}\hookrightarrow\mathcal{D}$, we obtain a natural transformation
$\phi_U\colon\mathrm{Lan}_{\mathrm{ev}_{(1,2)}}(\mathrm{ev}_0^{\ast}(U))\rightarrow U\times_{B} U$ as stated. 
Whenever $\mathcal{C}$ has all colimits occurring in Diagram (\ref{equapp4}), we may localize
$\mathcal{D}$ furthermore at the according sets of colimit-comparison maps and repeat the same argument to deduce the 
existence of the 2-cell (\ref{equapp4}) in $\mathcal{C}_{/B}$.
 
Lastly, suppose the canonical map
$\mathrm{colim}_{I\times I}(U\times_B U)\rightarrow\mathrm{colim}_IU\times_B\mathrm{colim}_IU$ is an equivalence. 
It follows that the vertical composition $\gamma\circ\mathrm{colim}_I\Delta_{\ast}(U)$ is equivalent to the 
diagonal of $\mathrm{colim}_I U$, because both
\[\xymatrix{
\mathrm{colim}_I U\ar[rrr]^(.4){\gamma\circ\mathrm{colim}_I\Delta_{\ast}(U)} & & & \mathrm{colim}_{I\times I}(U\times U)\ar@<.5ex>[r]^(.6){\pi_1}\ar@<-.5ex>[r]_(.6){\pi_2}& \mathrm{colim}_I U
}\]
compose to the identity (as is easily verified on associated cocones).
\end{proof}

\begin{notation}
For a decomposable diagram $U\colon I\rightarrow\mathcal{C}_{/B}$ and a pair of objects $i,j\in I$, the associated 
diagram
\[\Pi_U(i,j)\colon\mathrm{Fun}_{(i,j)}(D^1,I)\xrightarrow{U_{\ast}}\mathrm{Fun}_{(U_i,U_j)}(D^1,\mathcal{C}_{/B})\xrightarrow[\simeq]{\pi|_{(U_i,U_j)}}\mathcal{C}_{/U_i\times_{B} U_j}\]
that occurs in (\ref{equapp3}) will be referred to as the \emph{pre-diagonal} of $U$ (over $i$ and $j$).
\end{notation}

The following definition once more makes use of the notion of cofinal equivalence defined in 
Definition~\ref{defcofequiv}.

\begin{definition}\label{defcolimpretop}
Let $\mathcal{C}$ be an $\infty$-category and $T=\{T(B)\mid B\in\mathcal{C}\}$ be a class of diagrams of type
$I\rightarrow\mathcal{C}_{/B}$ for $B\in\mathcal{C}$ and $I\in\mathrm{Cat}_{\infty}$.
\begin{enumerate}
\item $T$ is \emph{(cofinally) reflexive} if for every object $B\in\mathcal{C}$ there is a diagram
$U\colon I\rightarrow\mathcal{C}_{/B}$ in $T(B)$ that is cofinally equivalent to the functor
$\Delta^0\xrightarrow{\{1_B\}}\mathcal{C}_{/B}$.
\item A class $T$ of decomposable diagrams is \emph{(cofinally) pre-diagonally closed} if for every diagram
$U\colon I\rightarrow\mathcal{C}_{/B}$ in $T$ and all $i,j\in I$ there is a diagram in $T(B)$ that is cofinally 
equivalent to the pre-diagonal
\[\Pi_U(i,j)\colon\mathrm{Fun}_{(i,j)}(D^1,I)\rightarrow\mathcal{C}_{/U_i\times_{C} U_j}.\]
\end{enumerate}
Say $T$ is a \emph{structured colimit pre-topology} if it is a reflexive and pre-diagonally closed semi-descent 
class. If $\mathcal{C}$ is small, say $T$ is small if $T$ is a set.
\end{definition}

For the following theorem we recall that a localization $L\colon\hat{\mathcal{C}}\rightarrow\mathcal{E}$ 
is said to be semi-left exact if $L$ preserves pullbacks along maps that are contained in $\mathcal{E}$ (considered 
as a full sub-$\infty$-category of $\hat{\mathcal{C}}$).

\begin{theorem}\label{thmcolimtop}
Let $\mathcal{C}$ be a small $\infty$-category, and let $T$ be a small structured colimit pre-topology (a reflexive 
semi-descent class) in $\mathcal{C}$. For $B\in\mathcal{C}$ let 
\[\mathrm{Cov}_{T}(B):=\{\mathrm{colim}yU\in\hat{\mathcal{C}}_{/yB}\mid U\in T(B)\},\]
and let $\mathrm{Cov}_{T}$ be the union of all $\mathrm{Cov}_{T}(B)$ considered as a set of maps in
$\hat{\mathcal{C}}$. Then the accessible localization $\hat{\mathcal{C}}\rightarrow\hat{\mathcal{C}}[\mathrm{Cov}_{T}^{-1}]$ is (semi-)left exact.
\end{theorem}
\begin{proof}
We show that $\mathrm{Cov}_T$ is a modulator on $\mathcal{C}$ whenever $T$ is a reflexive semi-descent class, and 
that furthermore $\mathrm{Cov}_T\cup\Delta[\mathrm{Cov}_T]\subset\mathrm{Sat}(\mathrm{Cov}_T)$ whenever $T$ is a 
structured colimit pre-topology. The statement then follows from Theorem~\ref{thmlexlocgen} (and for example from 
\cite[Proposition 1.3.2]{as_soa} in the case of semi-left exactness). The fact that the set $\mathrm{Cov}_T$ contains 
the identities $1_{yB}$ is precisely 
given by the fact that $T$ is reflexive. The fact that the set $\mathrm{Cov}_T$ is a modulator follows directly from 
the fact that $T$ is a semi-descent class, that the Yoneda embedding preserves pullbacks, and that all colimits in
$\hat{\mathcal{C}}$ are universal. Given a diagram $U\colon I\rightarrow\mathcal{C}_{/B}$ in $T$, we 
are left to show that the diagonal of the associated object $\mathrm{colim}yU\in\hat{\mathcal{C}}_{/yB}$ is contained 
in the saturation $\mathrm{Sat}(\mathrm{Cov}_T)$ whenever $T$ is pre-diagonally closed. As
$\mathrm{Sat}(\mathrm{Cov}_T)$ is closed under colimits and contains $\mathrm{Cov}_T$, it suffices to show that 
the diagonal
\[\Delta_{yB}(\mathrm{colim}yU)\colon\mathrm{colim}yU\rightarrow\mathrm{colim}yU\times_{yB}\mathrm{colim}yU\]
is a colimit of objects in $\mathrm{Cov}_T$. As $\hat{\mathcal{C}}_{/yB}$ has all small limits and small universal
colimits, by Lemma~\ref{lemmacolimittop1} it suffices to show that the natural transformation
\[\phi_{yU}(i,j)\colon
\underset{i\xleftarrow{\alpha} k\xrightarrow{\beta} j}{\mathrm{colim}}yU_k
\xrightarrow[\underset{i\xleftarrow{\alpha} k\xrightarrow{\beta} j}{\mathrm{colim}}(yU\alpha,yU\beta)]{}
yU_i\times_{yB} yU_j
\]
is contained in $\mathrm{Cov}_T$ for all $i,j\in I$. As stated in (\ref{equapp3}), each map $\phi_{yU}(i,j)$ 
represents the cocone
\[\Pi_{yU}(i,j)\colon\mathrm{Fun}_{(i,j)}(D^1,I)\xrightarrow{\Pi_U(i,j)}\mathcal{C}_{/U_i\times_C U_j}\xrightarrow{y}\hat{\mathcal{C}}_{/yU_i\times_{yB} yU_j}.\]
By definition of $\mathrm{Cov}_T$ it therefore suffices to show that each pre-diagonal $\Pi_U(i,j)$ is contained in 
$T$. This is precisely given by the assumption that $T$ is pre-diagonally closed.
\end{proof}


\begin{definition}\label{deftopsheaves}
Given a small $\infty$-category $\mathcal{C}$ and a small structured colimit topology $T$ in $\mathcal{C}$, we refer to the $\infty$-topos
\[\mathrm{Sh}_{T}(\mathcal{C}):=\hat{\mathcal{C}}[\mathrm{Cov}_{T}^{-1}]\]
as the $\infty$-topos of $T$-sheaves on $\mathcal{C}$. 
\end{definition}

By construction, a presheaf
$X\colon\mathcal{C}^{op}\rightarrow\mathcal{S}$ is a sheaf for $T$-diagrams if and only if it takes colimits of 
diagrams in $T$ to limits of spaces.

\begin{examples}\label{remexplescolimittop}
Let $\mathcal{C}$ be a small $\infty$-category and $T$ be a small structured colimit pre-topology in $\mathcal{C}$. 
\begin{enumerate}
\item Every representable presheaf over $\mathcal{C}$ is a $T$-sheaf. In other words, the localization
$\hat{\mathcal{C}}\rightarrow\mathrm{Sh}_{T}(\mathcal{C})$ is sub-canonical. 
\item Suppose $\mathcal{C}$ has pullbacks. Then the canonical indexing
$(\mathcal{C}_{/-})^{\simeq}\colon\mathcal{C}^{op}\rightarrow\mathcal{S}$ is a $T$-sheaf precisely if $T$ is a class 
of descent diagrams. 
\end{enumerate}
\end{examples}

\begin{corollary}\label{cordescentdiaganel}
Every small structured colimit pre-topology $T$ on a small $\infty$-category $\mathcal{C}$ is a descent class.
\end{corollary}
\begin{proof}
Apply Lemma~\ref{lemmadescentdiaganel} to the embedding $y\colon\mathcal{C}\rightarrow\mathrm{Sh}_T(\mathcal{C})$ 
from Theorem~\ref{thmcolimtop}. As $\mathrm{Sh}_T(\mathcal{C})$ is an $\infty$-topos, the class of all colimiting 
diagrams in $\mathrm{Sh}_T(\mathcal{C})$ is a descent class.
\end{proof}

 \subsection{Pre-diagonal stability of diagrams}\label{secsubinftydesc}

Let $\mathcal{C}$ be an $\infty$-category and let $\mathrm{SDesc}(\mathcal{C})$ denote the class of all semi-descent 
diagrams of type $U\colon I\rightarrow\mathcal{C}_{/B}$ for some small $\infty$-category $I$ and some object 
$B\in\mathcal{C}$. 
The class $\mathrm{SDesc}(\mathcal{C})$ is clearly the largest reflexive semi-descent class in $\mathcal{C}$, 
however it is generally not a structured colimit pre-topology. Yet, it contains every structured colimit pre-topology 
on $\mathcal{C}$ by definition.
In this section we characterize a structured colimit pre-topology $\mathrm{SDesc}_{\infty}(\mathcal{C})$ on
$\mathcal{C}$ which in the next section is proven to give rise to the largest such in a suitable context 
(Proposition~\ref{corhcdiscanonical}).


For the following, let $S^{\infty}$ be the poset generated by the diagram
\[\xymatrix{
x_0 & x_1\ar[l]\ar[dl] & x_2\ar[l]\ar[dl] & \dots\ar[l]\ar[dl]\\
y_0 & y_1\ar[l]\ar[ul]|(.5)\hole & y_2\ar[l]\ar[ul]|(.5)\hole & \dots\ar[l]\ar[ul]|(.5)\hole
}\]
Let $S^n$ be the truncation of $S^{\infty}$ at stage $n$, and $D^{n+1}$ be the join $\Delta^0\ast S^n$; that is, the 
poset given as follows.
\[\xymatrix{
x_0 & x_1\ar[l]\ar[dl] & x_2\ar[l]\ar[dl] & \dots\ar[l]\ar[dl] & x_n\ar[l]\ar[dl] & x_{n+1}\ar[l]\ar[dl]\\
y_0 & y_1\ar[l]\ar[ul]|(.5)\hole & y_2\ar[l]\ar[ul]|(.5)\hole & \dots\ar[l]\ar[ul]|(.5)\hole & y_n\ar[l]\ar[ul]|(.5)\hole & 
}\]
Note that for all $n\geq 0$,
\begin{align}\label{equthmcovtop}
S^{n+1} & \cong S^0\ast S^n\cong S^n\ast S^0\cong D^{n+1}\cup_{S^n} D^{n+1}, \\
\notag D^{n+1} & \cong D^n\ast S^0.
\end{align}
We obtain canonical inclusions $\iota\colon S^n\hookrightarrow S^m$ for $n\leq m$,
$\iota_n\colon S^n\hookrightarrow D^{n+1}$ given by the obvious inclusions $S^n\hookrightarrow S^m\ast S^n$, 
$S^n\hookrightarrow \Delta^0\ast S^n$. We furthermore obtain inclusions $\iota^{+}\colon S^n\rightarrow S^{n+1}$ and 
$\iota^{+}\colon D^n\rightarrow D^{n+1}$ given by the obvious inclusions $S^n\hookrightarrow S^n\ast S^0$ and 
$D^n\hookrightarrow D^n\ast S^0$.
  
\begin{notation}
Given an $\infty$-category $I$, we will refer to a functor of the form $p\colon S^n\rightarrow I$ as an
$n$-dimensional pod in $I$.


\end{notation}

For any pod $p\colon S^n\rightarrow I$ in an $\infty$-category $I$, the
$\infty$-category $\mathrm{Fun}_p(D^{n+1},I)$ is equivalent to the over-category $I_{/p}$ defined in
\cite[Section 1.2.9]{luriehtt}. Furthermore, given a pod $p\colon S^n\rightarrow I$ in a small $\infty$-category $I$, and given a diagram $U\colon I\rightarrow\mathcal{C}_{/B}$, the square
\[\xymatrix{
\mathrm{Fun}_p(D^{n+1},I)\ar[d]_{U_{\ast}}\ar[r]^(.7){\mathrm{ev}_{x_{n+1}}} & I\ar[d]^(.4)U \\
\mathrm{Fun}_{Up}(D^{n+1},\mathcal{C}_{/B})\ar[r]_(.7){\mathrm{ev}_{x_{n+1}}} & \mathcal{C}_{/B} 
}\]
commutes. The bottom left vertex $\mathrm{Fun}_{Up}(\Delta^0\ast S^n,\mathcal{C}_{/B})$ is the $\infty$-category of 
cones over the composite pod $Up$, and as such is equivalent to the over-category
$(\mathcal{C}_{/B})_{/\mathrm{lim}Up}\simeq\mathcal{C}_{/\mathrm{lim}Up}$ whenever the limit exists. Note here that 
this limit of $Up$ is computed in the slice $\mathcal{C}_{/B}$ however, not in $\mathcal{C}$. In particular, the 
limit of $Up\colon S^n\rightarrow\mathcal{C}_{/B}$ is the fiber product $U_i\times_B U_j$ whenever $p=(i,j)$ is
$0$-dimensional. Thus, in the case $n=0$, the diagram
\[\mathrm{Fun}_{(i,j)}(D^1,I)\xrightarrow{U_{\ast}}\mathrm{Fun}_{(U_i,U_j)}(D^1,\mathcal{C}_{/B})\xrightarrow{\simeq}\mathcal{C}_{/U_i\times_{B} U_j}\]
is exactly the pre-diagonal $\Pi_U(i,j)$ for any pod $p=(i,j)\colon S^0\rightarrow I$ whenever $U$ is 
decomposable.
More generally, whenever the according pullbacks exist, the limit of $Up\colon S^n\rightarrow\mathcal{C}_{/B}$ for 
pods $p$ of any dimension $n$ can be computed as follows.

\begin{lemma}\label{lemmadiaglimit}
For any diagram $U\colon I\rightarrow\mathcal{C}_{/B}$ and any pod $p$ in $I$ of dimension $n$ such that for all
$0\leq m\leq n$ the functor $\Pi^m_U(p|_{S^m})\colon\mathrm{Fun}_{p|_{S^m}}(D^{m+1},I)\rightarrow\mathcal{C}_{/\mathrm{lim}Up|_{S^m}}$ is decomposable, the limit of the composition $Up\colon S^n\rightarrow\mathcal{C}_{/B}$ is 
the iterated pullback
\begin{align}\label{defcovfunctorsitdpb}
U(p(x_n))\times_{\left(U(p(x_{n-1}))\times_{\vdots_{U(p(x_0))\times_{B}U(p(y_0))}}U(p(y_{n-1}))\right)}U(p(y_n)).
\end{align}
\end{lemma}
\begin{proof}
Induction along the dimension $n$. 
\end{proof}

\begin{definition}\label{defcovfunctors}
Let $\mathcal{C}$ be an $\infty$-category. A semi-descent diagram $U\colon I\rightarrow\mathcal{C}_{/B}$ is
\emph{pre-diagonally stable} if for all $n\geq 0$ and all pods $p\colon S^n\rightarrow I$ the
\emph{$n$-th pre-diagonal}
\begin{align}\label{equdefcovfunctors}
\Pi^n_U(p)\colon\mathrm{Fun}_p(D^{n+1},I)\xrightarrow{U_{\ast}}\mathrm{Fun}_{Up}(D^{n+1},\mathcal{C}_{/B})\xrightarrow{\simeq}\mathcal{C}_{/\mathrm{lim}Up}
\end{align}
is again a semi-descent diagram.
\end{definition}
 
Note that the limits in Definition~\ref{defcovfunctors} at dimension $n+1$ exist recursively by Lemma~\ref{lemmadiaglimit}. Indeed, $U$ is decomposable by assumption, and for all $0\leq m\leq n$ the functor
$\Pi^m_U(p|_{S^m})\colon\mathrm{Fun}_{p|_{S^m}}(D^{m+1},I)\rightarrow\mathcal{C}_{/\mathrm{lim}Up|_{S^m}}$ is again 
decomposable by assumption.
%

\begin{remark}
The case $n=-1$ in Definition~\ref{defcovfunctors} is trivial (as the limit of $Up$ is computed in the slice
$\mathcal{C}_{/B}$), so without loss of generality one may add $n=-1$ to Definition~\ref{defcovfunctors}.
\end{remark}

The following lemma states that all higher pre-diagonals can be expressed as iterated 1-dimensional pre-diagonals. 
This will be useful for later constructions.

\begin{lemma}\label{lemmaitprediag}
Let $\mathcal{C}$ be an $\infty$-category and $U\colon I\rightarrow\mathcal{C}_{/B}$ be a diagram. For any given pod 
$q\colon S^{n+1}\rightarrow I$ let $p=(\iota^{+})^{\ast}q$ be its restriction along the inclusion
$\iota^{+}\colon S^n\hookrightarrow S^n\ast S^0$. Let $(i,j)=(q(x_0),q(y_0))$. Then there is an equivalence
$f\colon\mathrm{lim}Uq\rightarrow\mathrm{lim}Up$ (computed in the according slices) together with a commutative 
diagram of associated pre-diagonals as follows.
\begin{align}\label{diaglemmaitprediag}
\begin{gathered}
\xymatrix{
\mathrm{Fun}_{q}(D^{n+1},I)\ar[rr]^{\Pi_U^{n+1}(q)}\ar[d]_{(\iota^+)^{\ast}}^{\simeq} & & \mathcal{C}_{/\mathrm{lim}Uq}\ar[d]_{\simeq}^{\Sigma_f}\\
\mathrm{Fun}_{p}(D^{n},\mathrm{Fun}_{(i,j)}(D^1,I))\ar[rr]_(.6){\Pi_{\Pi_U(i,j)}^n(p)} & & \mathcal{C}_{/\mathrm{lim}(\Pi_U(i,j) p)}
}
\end{gathered}
\end{align}
\end{lemma}
\begin{proof}
The square (\ref{diaglemmaitprediag}) unfolds by definition to the following diagram.
\begin{align*}
\xymatrix{
\mathrm{Fun}_{q}(D^{n+2},I)\ar[r]^{U_{\ast}}\ar[d]_{(\iota^+)^{\ast}}^{\simeq} & \mathrm{Fun}_{Uq}(D^{n+2},\mathcal{C}_{/B})\ar[d]_{(\iota^+)^{\ast}}^{\simeq}\ar[r]^(.6){\simeq} & \mathcal{C}_{/\mathrm{lim}Uq}\ar@{-->}[d]_{\simeq}^{\Sigma_f}\\
\mathrm{Fun}_{p}(D^{n+1},\mathrm{Fun}_{(i,j)}(D^1,I))\ar[r]_(.45){U_{\ast}} & \mathrm{Fun}_{Up}(D^{n+1},\mathrm{Fun}_{(U_i,U_j)}(D^1,\mathcal{C}_{/B}))\ar[r]_(.675){\simeq} & \mathcal{C}_{/\mathrm{lim}(\Pi_U(i,j) p)}
}
\end{align*}
Here, the square on the left hand side commutes simply by associativity.
The dashed equivalence on the right side is defined so that the right hand side square commutes. As the entire 
diagram commutes over $\mathcal{C}_{/B}$, this dashed equivalence is the push-forward along an equivalence
$f\colon\mathrm{lim}Uq\rightarrow\mathrm{lim}(\Pi_U(i,j) p)$ in $\mathcal{C}_{/B}$ by the Yoneda lemma.
\end{proof}

\begin{theorem}\label{thmcohtopex}
The class $\mathrm{SDesc}_{\infty}(\mathcal{C})$ of pre-diagonally stable semi-descent diagrams in $\mathcal{C}$ is a 
structured colimit pre-topology on any $\infty$-category $\mathcal{C}$.
\end{theorem}
\begin{proof}
The class $\mathrm{SDesc}_{\infty}(\mathcal{C})$ is reflexive as for any object $B\in\mathcal{C}$ the functor
$\{1_B\}\colon\Delta^0\rightarrow\mathcal{C}_{/B}$ is a descent diagram (Lemma~\ref{lemmadescdiagbase}.1). Its 
only pre-diagonal is $\{1_B\}\colon\Delta^0\rightarrow\mathcal{C}_{/B}$ itself. As every pre-diagonally stable
semi-descent diagram is a semi-descent diagram by definition, to show that $\mathrm{SDesc}_{\infty}(\mathcal{C})$ is 
a semi-descent class we are left to show that it is (cofinally) stable under base change. Therefore let
$U\colon I\rightarrow\mathcal{C}_{/B}$ be a pre-diagonally stable semi-descent diagram and
$f\colon C\rightarrow B$ be a morphism in $\mathcal{C}$. Then $f^{\ast}U\colon I\rightarrow\mathcal{C}_{/B}$ is again 
a semi-descent diagram. For any pod $p\colon S^n\rightarrow I$, we have the following commutative 
diagram.
\[\xymatrix{
\mathrm{Fun}_p(D^{n+1},I)\ar[r]^(.45){U_{\ast}}\ar@/_1pc/[dr]_(.3){(f^{\ast}U)_{\ast}} & \mathrm{Fun}_{Up}(D^{n+1},\mathcal{C}_{/B})\ar[r]^(.6){\simeq}\ar[d]^{f^{\ast}} & \mathcal{C}_{/\mathrm{lim}Up}\ar[d]^{f^{\ast}}  \\
 & \mathrm{Fun}_{f^{\ast}Up}(D^{n+1},\mathcal{C}_{/B})\ar[r]^(.6){\simeq} & \mathcal{C}_{/\mathrm{lim}f^{\ast}Up}
}\]
Here, the limit $\mathrm{lim}f^{\ast}Up$ is given by the pullback $f^{\ast}(\mathrm{lim}Up)$. By assumption, the top 
composition is a semi-descent diagram, and hence so is its post-composition with the base change functor $f^{\ast}$. 
This means that the bottom composition is a semi-descent diagram, which proves that the semi-descent diagram
$f^{\ast}U$ is again pre-diagonally stable.

To show that $\mathrm{SDesc}_{\infty}(\mathcal{C})$ is pre-diagonally closed, let
$U\colon I\rightarrow\mathcal{C}_{/B}$ be a pre-diagonally stable semi-descent diagram. We are to show that for all 
$i,j\in I$ the composition
\begin{align}\label{equthmcovtop1}
\Pi_U(i,j)\colon\mathrm{Fun}_{(i,j)}(D^1,I)\xrightarrow{U_{\ast}}\mathrm{Fun}_{(U_i,U_j)}(D^1,\mathcal{C}_{/B})\xrightarrow{\simeq}\mathcal{C}_{/U_i\times_{B} U_j}
\end{align}
is again a pre-diagonally stable semi-descent diagram. Thus, let $p\colon S^n\rightarrow\mathrm{Fun}_{(i,j)}(D^1,I)$ 
be a pod. We are to show that the pre-diagonal 
\begin{align*}
\Pi^n_{\Pi_U(i,j)}(p)\colon\mathrm{Fun}_p(D^{n+1},\mathrm{Fun}_{(i,j)}(D^1,I))\rightarrow\mathcal{C}_{/\mathrm{lim}(\Pi_U(i,j)p)}
\end{align*}
is a semi-descent diagram. Therefore, consider the uniquely determined pod-extension $q:=p\ast(i,j)\colon S^{n+1}\rightarrow I$ which 
restricts to $(i,j)\colon S^0\rightarrow I$ on $\iota\colon S^0\hookrightarrow S^{n+1}$, and such that the square
\begin{align}\label{diagthmcovtop+}
\begin{gathered}
\xymatrix{
S^n\ar[r]^(.3)p\ar[d]_{\iota^{+}} & \mathrm{Fun}_{(i,j)}(D^1,I)\ar[d]^{\mathrm{ev}_{x_1}} \\
S^{n+1}\ar[r]_q & I
}
\end{gathered}
\end{align}
commutes. This extension can be seen to exist via a series of transpositions and the calculations in 
(\ref{equthmcovtop}), using that the join $P\ast S^0$ (computed via the alternative join,
\cite[Section 4.2.1]{luriehtt}, which is the same on posets $P$) can be constructed as the coequalizer of the pair
\[\xymatrix{
P\times S^0\ar@<.5ex>[rr]^{1\times\iota_0}\ar@<-.5ex>[rr]_{(\{a,b\}\circ 1_P)\times\iota_0} & & P\times D^1
 }\]
for any given pair of points $a,b\in P$. We apply this to $P=D^{n+1},S^n$ and $a,b=x_0,y_0$. 
We have $p=(\iota^+)^{\ast}q$ by construction, and so we obtain a commutative diagram of higher pre-diagonals as in 
Lemma~\ref{lemmaitprediag}. The push-forward
$\Sigma_f\colon\mathcal{C}_{/\mathrm{lim}Uq}\rightarrow\mathcal{C}_{/\mathrm{lim}(\Pi_U(i,j) p)}$ in 
(\ref{diaglemmaitprediag}) is equivalent to the base change functor $(f^{-1})^{\ast}$. Thus, as the bottom horizontal 
composition of (\ref{diaglemmaitprediag}) is a semi-descent diagram by assumption, so is the top horizontal 
composition by Lemma~\ref{lemmadescdiagequiv}.
\end{proof}

%


\subsection{Higher covering diagrams}\label{secsubhcd}

In Theorem~\ref{thmcohtopex} we showed that $\mathrm{SDesc}_{\infty}(\mathcal{C})$ is always a structured colimit 
pre-topology. However, although (small) structured colimit pre-topologies $T$ are enough to construct an
$\infty$-topos of sheaves for $T$, the calculus of structured colimit pre-topologies itself appears to be fairly 
anodyne. For this reason we introduce the following strengthening in the case that $\mathcal{C}$ has pullbacks.

\begin{definition}
Say a diagram $U\colon I\rightarrow\mathcal{C}_{/B}$ is \emph{well-indexed} if $I$ has all pullbacks and
$U$ preserves them. A well-indexed pre-diagonally stable semi-descent diagram will be referred to as a
\emph{higher covering diagram}. We denote by $\mathrm{Geo}(\mathcal{C})$ the class of all small higher covering 
diagrams in $\mathcal{C}$. A \emph{well-structured colimit pre-topology} $T$ on $\mathcal{C}$ is a structured colimit 
topology of well-indexed diagrams.
\end{definition}

Well-indexed diagrams have the following advantage.

\begin{lemma}\label{lemmapbcolimittop}
Let $\mathcal{C}$ be an $\infty$-category with pullbacks. Suppose $U\colon I\rightarrow\mathcal{C}_{/B}$ is a
well-indexed diagram such that all pre-diagonals
$\Pi_U(i,j)\colon\mathrm{Fun}_{(i,j)}(D^1,\mathcal{C}_{/B})\rightarrow\mathcal{C}_{/U_i\times_B U_j}$ have a 
universal colimit. Then the natural transformation
\[\phi_U\colon\mathrm{Lan}_{\mathrm{ev}_{(1,2)}}(\mathrm{ev}_0^{\ast}(U))\rightarrow U\times_B U\]
from Lemma~\ref{lemmacolimittop1} is a cartesian natural transformation.
\end{lemma}
\begin{proof}
Given an arrow $(\alpha,\beta)\colon (i\sprime,j\sprime)\rightarrow (i,j)$ in $I\times I$, the induced push-forward
\[\Sigma_{(\alpha,\beta)}\colon \mathrm{Fun}_{(i\sprime,j\sprime)}(D^1,I)\rightarrow \mathrm{Fun}_{(i,j)}(D^1,I)\]
has a right adjoint $(\alpha,\beta)^{\ast}$ which maps a span $i\leftarrow k\rightarrow j$ to the limit
$(k\times_i i\sprime)\times_k (k\times_j j\sprime)$. Since $U\colon I\rightarrow\mathcal{C}_{/B}$ preserves 
pullbacks, the square
\[\xymatrix{
\mathrm{Fun}_{(i,j)}(D^1,I)\ar[r]^(.45){U_{\ast}}\ar[d]_{(\alpha,\beta)^{\ast}} & \mathrm{Fun}_{(U_i,U_j)}(D^1,\mathcal{C}_{/B})\ar[r]^(.6){\pi|_{(U_i,U_j)}}_(.6){\simeq} & \mathcal{C}_{/U_i\times_B U_j}\ar[d]^{(U(\alpha),U(\beta))^{\ast}} \\
\mathrm{Fun}_{(i\sprime,j\sprime)}(D^1,I)\ar[r]_(.45){U_{\ast}} & \mathrm{Fun}_{(U_{i\sprime},U_{j\sprime})}(D^1,\mathcal{C}_{/B})\ar[r]^(.6){\pi|_{(U_i\sprime,U_j\sprime)}}_(.6){\simeq} & \mathcal{C}_{/U_{i\sprime}\times_B U_{j\sprime}}
}\]
commutes up to equivalence. As the colimits of the pre-diagonals of $U$ are universal, and right adjoints are 
cofinal, 
via Lemma~\ref{lemmacolimittop1} we obtain a cartesian square of the form
\[\xymatrix{
\underset{i\sprime\leftarrow k\sprime\rightarrow j\sprime}{\mathrm{colim}}U_{k\sprime}\ar[r]\ar@{}[dr]|(.3){\pbs}\ar[d]_{\phi_U(i\sprime,j\sprime)} & \underset{i\leftarrow k\rightarrow j}{\mathrm{colim}}U_{k}\ar[d]^{\phi_U(i,j)} \\
U_{i\sprime}\times_B U_{j\sprime}\ar[r]_{(U\alpha,U\beta)} & U_{i}\times_B U_{j}.
}\]
\end{proof}

\begin{corollary}\label{corcolimittop}
Suppose $U\colon I\rightarrow\mathcal{C}_{/B}$ is a well-indexed descent diagram, and the pre-diagonal $\Pi_U(i,j)$ 
has a universal colimit for every pair $i,j\in I$. Then for all $i,j\in I$ there is a cartesian square of the form
\begin{align}\label{diagapp5}
\begin{gathered}
\xymatrix{
\underset{i\leftarrow k\rightarrow j}{\mathrm{colim}}U_k\ar[r]\ar[d]_{\phi_U(i,j)} & \mathrm{colim}U\ar[d]^{\Delta(\mathrm{colim}U)}_{\simeq} \\
U_i\times_B U_j\ar[r]_(.4){} & \mathrm{colim}U\times_B\mathrm{colim}U.
}
\end{gathered}
\end{align}
\end{corollary}
\begin{proof}
The product $U\times_B U\colon I\rightarrow\mathcal{C}_{/B}$ is again a descent diagram by
Corollary~\ref{cordescstabprod}. In particular, the canonical map
$\mathrm{colim}(U\times_B U)\rightarrow\mathrm{colim}U\times_B\mathrm{colim}U$ is an equivalence. 
Furthermore, the natural transformation
$\phi_U\colon\mathrm{Lan}_{\mathrm{ev}_{(1,2)}}(\mathrm{ev}_0^{\ast}(U))\rightarrow U\times_B U$ 
in $\mathrm{Fun}(I^{\partial\Delta^1},\mathcal{C}_{/B})$ is cartesian by Lemma~\ref{lemmapbcolimittop}. By 
Lemma~\ref{lemmacolimittop1} the colimit of $\phi_U$ is the diagonal $\Delta(\mathrm{colim}U)$. As $U\times_B U$ is a 
descent diagram, it follows that the squares (\ref{diagapp5}) are cartesian.
\end{proof}

\begin{remark}\label{remwicolimittop}
Via Corollary~\ref{corcolimittop}, the proof of Theorem~\ref{thmcolimtop} in fact shows that $\mathrm{Cov}_T$ is an 
Id-modulator (Remark~\ref{remidmod}) whenever $T$ is a well-structured colimit pre-topology on an
$\infty$-category $\mathcal{C}$.
\end{remark}

Furthermore, if $\mathcal{C}$ has pullbacks, every structured colimit pre-topology on $\mathcal{C}$ can be replaced 
by a well-structured colimit pre-topology that presents the same sheaf theory.

\begin{lemma}\label{lemmarfibwellind}
Suppose $\mathcal{C}$ has pullbacks. Then all right fibrations over $\mathcal{C}$ are well-indexed.
\end{lemma}
\begin{proof}
Every cospan $A_1\xrightarrow{a_1}B\xleftarrow{a_2} A_2$ of cartesian morphisms in the total
$\infty$-category $\mathcal{E}$ of a right fibration $p\colon\mathcal{E}\twoheadrightarrow\mathcal{C}$ gives rise to 
a cospan $(p(a_1),p(a_2))$ in $\mathcal{C}$. The projections
$\pi_i\colon p(A_1)\times_{p(C)} p(A_2)\rightarrow p(A_i)$ induce cartesian lifts $D_i\rightarrow A_i$ such that 
$p(D_i)\simeq p(A_1)\times_{p(C)} p(A_2)$. As $p(a_1)\pi_1=p(a_2)\pi_2$ in $B$, and 
cartesian lifts compose and are unique up to equivalence, the objects $D_1$ and $D_2$ are equivalent objects in the 
fiber of $p(D_i)$, and the resulting square
\[\xymatrix{
D_i\ar[r]\ar[d] & A_1\ar[d] \\
A_2\ar[r] & B
}\]
of cartesian morphisms in the total $\infty$-category $\mathcal{E}$ commutes. Such squares are automatically 
cartesian.
\end{proof}

\begin{remark}\label{remrfibwellind}
Lemma~\ref{lemmarfibwellind} has an intuitive conceptual explanation. Namely, right fibrations over $\mathcal{C}$ 
represent presheaves over $\mathcal{C}$ under the (Un)Straightening construction. Trivially, the
$\infty$-category $\mathcal{S}$ embeds into the doctrine of $\mathrm{Cat}_{\infty}^{\mathrm{pb}}$ of
$\infty$-categories with pullbacks and pullback-preserving functors. It follows that presheaves are objects with 
pullbacks in the $\infty$-cosmos of indexed $\infty$-categories over $\mathcal{C}$. Equivalently, 
right fibrations are objects with pullbacks in the $\infty$-category of cartesian fibrations over
$\mathcal{C}$ (or one can directly argue that right fibrations are the discrete objects here and such always have 
pullbacks formally). Whenever $\mathcal{C}$ has pullbacks itself, the fact that such objects are the
well-indexed ones is exactly the kind of statement often found in the categorical context, see e.g.\ 
\cite[Lemma B.1.4.1]{elephant}.
\end{remark}

\begin{proposition}\label{corhcdiscanonical}
Suppose $\mathcal{C}$ is a (small) $\infty$-category with pullbacks and suppose
$\hat{\mathcal{C}}\rightarrow\mathcal{E}$ is a sub-canonical left exact (accessible) localization. Then there is a 
(small) well-structured colimit topology $T_{\mathcal{E}}$ on $\mathcal{C}$ such that
$\mathrm{Sh}_{T_{\mathcal{E}}}(\mathcal{C})=\mathcal{E}$.
\end{proposition}
\begin{proof}
We make the argument for small $\infty$-categories and sub-canonical left exact accessible localizations. For the 
large case simply replace occurrences of ``set'' by ``class'' and drop the smallness condition on modulators.
Choose a fiberwise left exact modulator $M$ such that $\mathcal{E}=\hat{\mathcal{C}}[M^{-1}]$ and let 
$T_{\mathcal{E}}:=\mathrm{Un}[M]$. Then $M=\mathrm{Cov}_{T_{\mathcal{E}}}$ by Lemma~\ref{lemmastcolim}. The set 
$T_{\mathcal{E}}$ consists of well-indexed diagrams by Lemma~\ref{lemmarfibwellind}. It is reflexive and stable under 
base change, because $M$ is reflexive and pullback-stable. The fact that every diagram in $T_{\mathcal{E}}(B)$ is 
colimiting over $B$ follows from sub-canonicity: For all objects $B\in\mathcal{C}$ and all $m\in M(B)$, let
$\mathrm{Un}(m)\colon I\twoheadrightarrow\mathcal{C}_{/B}$ be the Unstraightening of $m\in\widehat{\mathcal{C}_{/B}}$ 
so that the diagram
\begin{align}\label{diagcorhcdiscanonical}
\begin{gathered}
\xymatrix{
\mathcal{C}(B,C)\ar[rr]^(.35){\mathrm{Un}(m)^{\ast}\mathsf{C}}\ar[d]^{\simeq}_y & & \mathrm{Fun}(I,\mathcal{C})(s\mathrm{Un}(m),\mathsf{C}(C))\ar[d]_{\simeq}^y \\
\hat{\mathcal{C}}(yB,yC)\ar[rr]^(.35){y\mathrm{Un}(m)^{\ast}\mathsf{C}}\ar@/_1pc/[drr]_(.35){\mathrm{colim}y(\mathrm{Un}(m))^{\ast}} & & \mathrm{Fun}(I,\hat{\mathcal{C}})(ys\mathrm{Un}(m),\mathsf{C}(yC))\ar[d]^{\simeq} \\
 & & \hat{\mathcal{C}}(\mathrm{colim}(ys\mathrm{Un}(m)),yC) \\
}
\end{gathered}
\end{align}
commutes. Any such $m\in M(B)$ is the colimit of the composition
\[\mathrm{El}(m)\overset{\mathrm{Un}(m)}{\twoheadrightarrow}\mathcal{C}_{/B}\xrightarrow{y}\hat{\mathcal{C}}_{/yB},\]
and as all $yC$ are assumed to be $m$-local it follows that the restriction functor
$\mathrm{colim}(y\mathrm{Un}(m))^{\ast}$ in Diagram~(\ref{diagcorhcdiscanonical}) is an equivalence. In turn, it 
follows that the top horizontal arrow in Diagram~(\ref{diagcorhcdiscanonical}) is an equivalence for all 
$C\in\mathcal{D}$, which means that the cocone $\mathrm{Un}(m)\colon\mathrm{El}(m)\rightarrow\mathcal{C}_{/B}$ is 
colimiting.

It then follows that the set $T_{\mathcal{E}}$ consists of semi-descent diagrams indeed by the fact that $M$ is pullback-stable and that $f^{\ast}\mathrm{Un}(m)\simeq\mathrm{Un}(f^{\ast}m)$ for all $m\in M(B)$ and all
$f\colon C\rightarrow B$ in $\mathcal{C}$. We are left to show that 
$T_{\mathcal{E}}$ is pre-diagonally closed. Therefore let $U\colon I\rightarrow\mathcal{C}_{/B}$ be in 
$T_{\mathcal{E}}(B)$ and $i,j\in I$ be objects. The morphism $\mathrm{colim}yU\rightarrow yB$ is contained in
$\mathrm{Cov}_{T_{\mathcal{E}}}$ by definition. Thus, as $\mathrm{Cov}_{T_{\mathcal{E}}}$ is fiberwise left exact, it 
follows that the morphism
\[\mathrm{colim}(y\Pi_U(i,j))\rightarrow y U_i\times_{yB} yU_j\]
is again contained in $\mathrm{Cov}_{T_{\mathcal{E}}}$ by Corollary~\ref{corcolimittop} applied to 
$yU\colon I\rightarrow \hat{\mathcal{C}}_{/yB}$. In particular, the right fibration
\begin{align}\label{equlemmawellstrrepl}
\mathrm{Un}(\mathrm{colim}(y\Pi_U(i,j)))\twoheadrightarrow\mathcal{C}_{/U_i\times_{B} U_j}
\end{align}
is contained in $T_{\mathcal{E}}$. The right fibration (\ref{equlemmawellstrrepl}) and the pre-diagonal
$\Pi_U(i,j)\colon I\rightarrow\mathcal{C}_{/U_i\times_B U_j}$ have the same colimit after post-composition with the 
Yoneda embedding $\mathcal{C}_{/B}\rightarrow\hat{\mathcal{C}}_{/yB}$ by construction.
It follows that $\Pi_U(i,j)$ and the right fibration (\ref{equlemmawellstrrepl}) are cofinally equivalent.
\end{proof}

In particular, it follows that all structured colimit pre-topologies can be presented by a well-structured one 
whenever $\mathcal{C}$ has pullbacks.

\begin{corollary}\label{lemmawellstrrepl}
Suppose $\mathcal{C}$ has pullbacks. Then for every (small) structured colimit pre-topology $T$ there is a (small)
well-structured colimit pre-topology $T\sprime$ such that
$\mathrm{Sh}_T(\mathcal{C})=\mathrm{Sh}_{T\sprime}(\mathcal{C})$.
\end{corollary}
\begin{proof}
If $T$ is a small structured colimit pre-topology, then $\hat{\mathcal{C}}\rightarrow\mathrm{Sh}_T(\mathcal{C})$ is 
accessible, left exact and sub-canonical. We thus may apply Proposition~\ref{corhcdiscanonical}. 
\end{proof}

\begin{lemma}\label{lemmaequivmaxtop}
Every well-structured colimit pre-topology $T$ on a (small) $\infty$-category $\mathcal{C}$ with pullbacks is 
contained in a largest well-structured colimit pre-topology $\bar{T}$ such that
$\mathrm{Sh}_T(\mathcal{C})=\mathrm{Sh}_{\bar{T}}(\mathcal{C})$. The class $\bar{T}$ is closed under cofinal 
equivalence, i.e.\ whenever $U\colon I\rightarrow\mathcal{C}_{/B}$ and $V\colon I\rightarrow\mathcal{C}_{/B}$ are 
well-indexed diagrams such that $U$ and $V$ are cofinally equivalent, then $U$ is contained in $\bar{T}$ if and only 
if $V$ is contained in $\bar{T}$.
\end{lemma}
\begin{proof}
Let $\mathrm{Sh}_T(\mathcal{C})=\hat{\mathcal{C}}[\mathrm{Cov}_T^{-1}]$ be the $\infty$-topos generated by $T$ from 
Theorem~\ref{thmcolimtop}, which, if $\mathcal{C}$ and $T$ are large, is to be taken as the $\infty$-topos of
$T$-sheaves valued in an accordingly larger universe of spaces. 

For each $B\in\mathcal{C}$, let $M(B)\subseteq\hat{\mathcal{C}}_{/yB}$ be the large fiberwise left exact modulator 
of maps over $yB$ which are inverted by $\hat{\mathcal{C}}\rightarrow\mathrm{Sh}_T(\mathcal{C})$.
For each $B\in\mathcal{C}$, let $\bar{T}(B)$ be the class of well-indexed diagrams
$U\colon I\rightarrow\mathcal{C}_{/B}$ such that $\mathrm{colim}yU\rightarrow yB$ is contained in $M(B)$. Then
$\bar{T}$ is the largest class of well-indexed diagrams such that
$\mathrm{Sh}_T(\mathcal{C})=\mathrm{Sh}_{\bar{T}}(\mathcal{C})$ by construction. In particular, $T$ is 
contained in $\bar{T}$ and so the latter is reflexive. Furthermore, we have seen in the proof of 
Proposition~\ref{corhcdiscanonical} that $\mathrm{Un}[M]$ is a well-structured colimit pre-topology. Every diagram
$U\colon I\rightarrow\mathcal{C}_{/B}$ in $T(B)$ is cofinally equivalent to the right fibration
$\mathrm{Un}(\mathrm{colim}yU)\twoheadrightarrow\mathcal{C}_{/B}$ via Lemma~\ref{lemmastcolim}. As
$\mathrm{colim}yU$ is contained in $M(B)$ by assumption, its Unstraightening is contained in $\mathrm{Un}[M]$ and so 
is a semi-descent diagram. In particular, every diagram $U$ in $\bar{T}$ is a semi-descent diagrams by 
Lemma~\ref{lemmadesccofstable}. The class $\bar{T}$ is stable under base change for essentially the same reason.

To show that $\bar{T}$ is pre-diagonally closed, let $U\colon I\rightarrow\mathcal{C}_{/B}$ be contained in $\bar{T}$ 
and consider the composition
\[\Pi_{yU}(i,j)\colon\mathrm{Fun}_{(i,j)}(D^1,I)\xrightarrow{\Pi_U(i,j)}\mathcal{C}_{/U_i\times_{C} U_j}\xrightarrow{y}\hat{\mathcal{C}}_{/yU_i\times_{yC} yU_j}.
\]
The pre-diagonal $\Pi_U(i,j)$ is contained in $\bar{T}$ if the morphism
$\mathrm{colim}(y\Pi U(i,j))\rightarrow yU_i\times_{yB}yU_j$ in $\hat{\mathcal{C}}$ is inverted by the left exact 
localization $\hat{\mathcal{C}}\rightarrow\mathrm{Sh}_T(\mathcal{C})$. By assumption, the morphism
$\mathrm{colim}yU\rightarrow yB$ is inverted, and hence so is its diagonal
$\Delta_{\mathrm{colim}yU}\colon\mathrm{colim}yU\rightarrow\mathrm{colim}yU\times_{yB}\mathrm{colim}yU$. By 
Corollary~\ref{corcolimittop} applied to the well-indexed diagram $yU\colon I\rightarrow\hat{\mathcal{C}}_{/yB}$, we 
see that the morphisms $\mathrm{colim}(\Pi_{yU}(i,j))\rightarrow yU_i\times_{yB}yU_j$  are inverted as well. 

We are left to show that $\bar{T}$ is cofinally closed. However, given cofinally equivalent diagrams
$U\colon I\rightarrow\mathcal{C}_{/B}$ and $V\colon J\rightarrow\mathcal{C}_{/B}$, by definition
$\mathrm{colim}_I yU\rightarrow yB$ is contained in $\mathrm{Cov}_{\bar{T}}$ if and only if
$\mathrm{colim}_J yV\rightarrow yB$ is so. 
\end{proof}

\begin{theorem}\label{prophcdmaxtop}
Let $\mathcal{C}$ be an $\infty$-category with pullbacks. Then $\mathrm{Geo}(\mathcal{C})$ is 
the largest well-structured colimit pre-topology on $\mathcal{C}$. 
\end{theorem}
\begin{proof}
We have seen that $\mathrm{SDesc}_{\infty}(\mathcal{C})$ is a structured colimit pre-topology in 
Theorem~\ref{thmcohtopex}. As well-indexedness is stable under all relevant operations, it follows that the class 
$\mathrm{Geo}(\mathcal{C})$ of higher covering diagrams is a structured colimit pre-topology as well.
Any other well-structured colimit pre-topology $T$ is contained in a well-structured colimit pre-topology $\bar{T}$ 
that is closed under cofinal equivalence by Lemma~\ref{lemmaequivmaxtop}. This class $\bar{T}$ consist of
semi-descent diagrams and is (actually) closed under pre-diagonals. It follows that the pre-diagonal of any diagram 
contained in $\bar{T}$ is again a semi-descent diagram. As all higher pre-diagonals of any diagram in $\bar{T}$ are 
likewise again contained in $\bar{T}$ via their iterative description in Lemma~\ref{lemmaitprediag}, it follows that 
every diagram in $\bar{T}$ is a higher covering diagram.
\end{proof}

Theorem~\ref{prophcdmaxtop} in conjunction with Proposition~\ref{corhcdiscanonical} morally states that 
$\mathrm{Sh}_{\mathrm{Geo}}(\mathcal{C})$ is the largest sub-canonical sheaf theory on any small $\infty$-category
$\mathcal{C}$ with pullbacks, and that as such it presents the canonical sheaf theory on such $\infty$-categories. 
However, as the structured colimit pre-topology $\mathrm{Geo}(\mathcal{C})$ is not small even if $\mathcal{C}$ is 
small, it a priori only gives rise to a left exact localization of large presheaves. Hence, the according canonicity 
property has to be stated subject to a size caveat. There are multiple ways to do so, one of which is formulated in 
Remark~\ref{remjointcanonical}.

\begin{corollary}\label{corhcfcofstab}
Let $\mathcal{C}$ be an $\infty$-category with pullbacks. The class $\mathrm{Geo}(\mathcal{C})$ of higher covering 
diagrams is closed under cofinal equivalence among well-indexed diagrams.
\end{corollary}
\begin{proof}
By Lemma~\ref{lemmaequivmaxtop} there is a well-structured colimit pre-topology on $\mathcal{C}$ that is cofinally 
closed and contains $\mathrm{Geo}(\mathcal{C})$. By Theorem~\ref{prophcdmaxtop}, this well-structured 
colimit pre-topology is in turn contained in $\mathrm{Geo}(\mathcal{C})$.
\end{proof}


We end this section with some examples of higher covering diagrams. Therefore, we first note that higher covering 
diagrams are ubiquitous whenever $\mathcal{C}$ has descent.

\begin{lemma}\label{lemmadeschcd}
Suppose $\mathcal{C}$ has pullbacks, and $U\colon I\rightarrow\mathcal{C}_{/B}$ is a well-indexed descent diagram 
such that for all $p\colon S^n\rightarrow I$, the pre-diagonal
$\Pi^n_U(p)\colon\mathrm{Fun}_p(D^{n+1},I)\rightarrow\mathcal{C}_{/\mathrm{lim}Up}$ factors through a descent 
diagram
\[\Pi^n_U(p)\colon\mathrm{Fun}_p(D^{n+1},I)\rightarrow\mathcal{C}_{/\mathrm{colim}\Pi^n_U(p)}.\]
Then $U$ is a higher covering diagram.
\end{lemma}
\begin{proof}
We may apply Corollary~\ref{corcolimittop} to the diagram $U\colon I\rightarrow\mathcal{C}_{/B}$. 
The diagonal $B\rightarrow B\times_B B$ is always an equivalence, which shows that $\phi_U$ is a natural equivalence. 
In other words, each pre-diagonal
$\Pi_U(i,j)\colon\mathrm{Fun}_{(i,j)}(D^{1},I)\rightarrow\mathcal{C}_{/U_i\times_B U_j}$ is again colimiting (and hence in particular a descent diagram). Applying this argument inductively shows that all higher pre-diagonals
\[\Pi_U^n(p)\colon\mathrm{Fun}_p(D^{n},I)\rightarrow\mathcal{C}_{/\mathrm{lim}Up}\]
are (semi-)descent diagrams via Lemma~\ref{lemmaitprediag}.
\end{proof}

\begin{corollary}\label{corhcdlogos}
Suppose $\mathcal{C}$ is finitely complete and cocomplete and has descent (Example~\ref{explelogoi}). Then every
well-indexed small diagram $U\colon I\rightarrow\mathcal{C}_{/\mathrm{colim}U}$ is a higher covering diagram.
\end{corollary}
\begin{proof}
This follows immediately from Lemma~\ref{lemmadeschcd}. Alternatively, one can show directly that in this case
the class of all small well-indexed diagrams is a well-structured colimit pre-topology. The statement then follows 
from Theorem~\ref{prophcdmaxtop}.
\end{proof}

Vice versa, we have the following implication in the other direction.

\begin{lemma}\label{lemmahcddesc}
For any $\infty$-category $\mathcal{C}$ with pullbacks, the class $\mathrm{Geo}(\mathcal{C})$ is a descent class. In 
particular, if $U\colon I\rightarrow\mathcal{C}_{/B}$ is a small higher covering diagram such that for all cartesian 
natural transformations $V\rightarrow U$, the domain $V\colon I\rightarrow\mathcal{C}_{/B}$ also factors through a 
higher covering diagram $V\colon I\rightarrow\mathcal{C}_{/\mathrm{colim}V}$, then $U$ is a descent diagram. 
\end{lemma}
\begin{proof}
It suffices to show that for all regular cardinals $\kappa$ large enough, the set
$\mathrm{Geo}_{\kappa}(\mathcal{C})$ of $\kappa$-small higher covering diagrams is a descent class. However, each set 
$\mathrm{Geo}_{\kappa}(\mathcal{C})$ is a small well-structured colimit pre-topology on $\mathcal{C}$ via 
Theorem~\ref{prophcdmaxtop}. The statement hence follows from Corollary~\ref{cordescentdiaganel}.
\end{proof}

\begin{remark}\label{remdeschcdequiv}
Lemma~\ref{lemmadeschcd} and Lemma~\ref{lemmahcddesc} together show an interesting interplay between descent diagrams 
and higher covering diagrams. Higher covering diagrams are always closed under higher pre-diagonals, while descent 
diagrams are always closed under cartesian natural transformations. If a higher covering diagram is such that every 
cartesian natural transformation over it factors again through a higher covering diagram, then it is a descent 
diagram. If a descent diagram is such that its higher pre-diagonals factor through descent diagrams, then it is a 
higher covering diagram. 
Thus, whenever enough colimiting diagrams in $\mathcal{C}$ are descent diagrams, or whenever enough colimiting 
diagrams in $\mathcal{C}$ are higher covering diagrams for that matter, then one implies the other. 
This will be the case in the $\infty$-categories considered in Section~\ref{secexttop} as well as in 
Section~\ref{secsites}.
\end{remark}

\begin{example}\label{explecovprod}
Suppose $\mathcal{C}$ is an $\infty$-category with pullbacks. Suppose $I$ is an $\infty$-category that has both 
pullbacks and non-empty finite products. Then a well-indexed semi-descent diagram
$U\colon I\rightarrow\mathcal{C}_{/B}$ is a higher covering diagram if and only if it preserves non-empty finite 
products. In particular, if $I$ has all finite limits and hence a terminal object $t$, then a diagram
$U\colon I\rightarrow\mathcal{C}_{/B}$ is higher covering if and only if $U$ is left exact.
\end{example}
\begin{proof}
If $I$ has non-empty finite products, then for every pod $p\colon S^n\rightarrow I$ the $\infty$-category
$\mathrm{Fun}_p(D^{n+1},I)$ has a terminal object given by the iterated pullback
\[\mathrm{lim}p\simeq p(x_n)\times_{\vdots_{p(x_0)\times p(y_0)}} p(y_n).\]
In this case, the colimit of any diagram with domain $\mathrm{Fun}_p(D^{n+1},I)$ is computed by evaluation at this 
terminal object. Thus, if $U\colon I\rightarrow\mathcal{C}_{/B}$ is a well-indexed semi-descent diagram that 
preserves non-empty finite products, then the higher pre-diagonal
$\Pi^n_U(p)\colon\mathrm{Fun}_p(D^{n+1},I)\rightarrow\mathcal{C}_{/\mathrm{lim}Up}$ is colimiting. Furthermore, the 
diagram
\[\xymatrix{
I\ar[r]^U\ar[d]_{-\times\mathrm{lim}p} & \mathcal{C}_{/B}\ar[r]^{-\times_B\mathrm{lim}Up} &
\mathcal{C}_{/\mathrm{lim}Up}\\
I_{/\mathrm{lim}p}\ar[r]_(.4){\simeq} & \mathrm{Fun}_p(D^{n+1},I)\ar@/_1pc/[ur]_(.6){\Pi^n_U(p)} & 
}\]
commutes. As $U$ is assumed to be a semi-descent diagram, so is the top horizontal composition
$-\times_B\mathrm{lim}Up\circ U$. The left vertical functor $-\times\mathrm{lim}p$ is a right adjoint and hence 
cofinal. So is the equivalence on the bottom, which means that $\Pi^n_U(p)$ is a semi-descent diagram by 
Lemma~\ref{lemmadesccofstable}.

Vice versa, if $U\colon I\rightarrow\mathcal{C}_{/B}$ is a higher covering diagram, then for every pair $i,j\in I$ 
the cocone $\Pi_U(i,j)\colon\mathrm{Fun}_{(i,j)}(D^1,I)\rightarrow\mathcal{C}_{/U_i\times_B U_j}$ is 
colimiting by assumption. Again, this means that
$\Pi_U(i,j)(i\leftarrow i\times j\rightarrow j)\simeq 1_{U_i\times_B U_j}$, and hence that
$U(i\times j)\simeq U_i\times_B U_j$. 

Suppose $I$ is left exact. Then if $U\colon I\rightarrow\mathcal{C}_{/B}$ is higher covering, to show that $U$ is 
left exact we are left to show that $U$ preserves the terminal object. But the terminal cocone
$\mathrm{id}_I\rightarrow t$ is an absolute colimit, and so $\mathrm{colim}U\simeq U(t)$ as well. The 
colimit of $U$ is the terminal object $1_B\in\mathcal{C}_{/B}$ by assumption that $U$ is colimiting. 
Thus, $U$ preserves the terminal object. Vice versa, if $U$ is left exact, by the above we are left to show that $U$ 
is a semi-descent diagram. But $\mathrm{colim}U\simeq U(t)\simeq 1_{B}$, and so $U$ is colimiting.
Furthermore, for all $f\colon C\rightarrow B$ the base change functor
$f^{\ast}\colon\mathcal{C}_{/B}\rightarrow\mathcal{C}_{/C}$ is left exact itself, and so the composition
$f^{\ast}U\colon I\rightarrow\mathcal{C}_{/C}$ is again colimiting by the same argument.
\end{proof}

\begin{example}\label{explecovid}
Suppose $\mathcal{C}$ has pullbacks. The identity functor on $\mathcal{C}$ admits a colimit if and only if
$\mathcal{C}$ has a terminal object $t$. In this case, the identity functor
$1_{\mathcal{C}}\colon\mathcal{C}\rightarrow\mathcal{C}_{/t}$ is always a higher covering diagram. Indeed,
$\mathcal{C}$ has all finite limits and hence by Example~\ref{explecovprod} the identity on $\mathcal{C}$ is a higher 
covering diagram.
\end{example}

\begin{example}\label{explecovgen}
Suppose $\mathcal{C}$ is a presentable $\infty$-category with universal colimits. Let $I\hookrightarrow\mathcal{C}$ 
be the fully faithful inclusion of some generating set which is closed under pullbacks. Then for all objects 
$B\in\mathcal{C}$, the fully faithful inclusion
\[\iota_C\colon I_{/B}\hookrightarrow\mathcal{C}_{/B}\]
is a higher covering diagram. Indeed, $\iota_B$ is a pre-descent diagram for all objects $B$, and all pre-descent 
diagrams in $\mathcal{C}$ are semi-descent diagrams by assumption. We are hence to show that for all
$B\in\mathcal{C}$ and all pods $p\in\mathrm{Fun}(S^n,I_{/B})$, the pre-diagonal 
$\Pi_{\iota_B}^n(p)\colon\mathrm{Fun}_p(D^n,I_{/B})\hookrightarrow\mathcal{C}_{/\mathrm{lim}\iota_B p}$ is again 
colimiting. However, by virtue of fully faithfulness of the embedding $\iota_B$, the $\infty$-category
$\mathrm{Fun}_p(D^n,I_{/B})$ is equivalent to the slice $I_{/\mathrm{lim}\iota_B p}$, and the pre-diagonal
$\Pi_{\iota_B}^n(p)$ itself is equivalent to the fully faithful embedding
\[\iota_{\mathrm{lim}\iota_B p}\colon I_{/\mathrm{lim}\iota_B p}\hookrightarrow \mathcal{C}_{/\mathrm{lim}\iota_B p}.\]
This embedding is colimiting by assumption.
\end{example}

\begin{example}\label{explecovposet}
Suppose $\mathcal{C}$ has pullbacks. Whenever $I$ is a poset with non-empty finite meets, a well-indexed
semi-descent diagram $U\colon I\rightarrow\mathcal{C}_{/B}$ is higher covering if and only if it factors through a 
meet-preserving morphism of posets
\[\xymatrix{
 & \mathrm{Sub}(B)\ar@{^(->}[d]\\
I\ar[r]_U\ar@{-->}[ur] & \mathcal{C}_{/B}.
}\]
This follows from Example~\ref{explecovprod} and the fact that for all $i\in I$ the degenerate triple $i=i=i$ is 
terminal in the category of spans $i\leftarrow k\rightarrow i$ whenever $I$ is a poset. Hence,
$U_i\xrightarrow{\simeq}U_i\times_{B}U_i$ if $U$ is higher covering. 
\end{example}

\begin{example}\label{explecovgrpd}
Suppose $\mathcal{C}$ is an $\infty$-groupoid. Then every diagram $U\colon I\rightarrow\mathcal{C}_{/B}$ is
pullback-preserving, and all colimits that exist in $\mathcal{C}$ are universal. Furthermore, as all slices of an
$\infty$-groupoid $\mathcal{C}$ are contractible, all diagrams of type $I\rightarrow\mathcal{C}_{/B}$ for any object 
$B\in\mathcal{C}$ are pre-descent diagrams trivially. It follows that every diagram of type
$I\rightarrow\mathcal{C}_{/B}$ (where $I$ has pullbacks) is a pre-diagonally stable semi-descent diagram (a higher 
covering diagram). 
\end{example}

The next two sections are concerned specifically with the following two examples.

\begin{example}
Every ($\kappa$-small) discrete colimiting diagram $U\colon I\rightarrow \mathcal{C}_{/B}$ is higher covering 
whenever $\mathcal{C}$ is ($\kappa$-)extensive. See Section~\ref{secexttop}.
\end{example}

\begin{example}
For any morphism $f\colon E\rightarrow B$ in a regular $\infty$-category $\mathcal{C}$, the symmetric \v{C}ech nerve
$\check{\Sigma}(f)\colon\mathrm{FinSet}_{+}^{op}\rightarrow\mathcal{C}_{/B}$ is a higher covering diagram. More 
generally, for any $\kappa$-coherent cover $F=(E_i\rightarrow B\mid i\in\kappa)$ in a $\kappa$-coherent
$\infty$-category $\mathcal{C}$, the associated symmetric \v{C}ech nerve
$\check{\Sigma}(F)\colon\mathrm{FS}_{+}(I)^{op}\rightarrow\mathcal{C}_{/B}$ is a higher covering diagram. See 
Section~\ref{secregtop}.
\end{example}


\section{Extensive $\infty$-categories and their sheaves}\label{secexttop}

We formulate the definitions in this section for the finite case only, but everything up to 
Proposition~\ref{propexthyper} can be phrased for arbitrary regular cardinals $\kappa$ in a straight-forward fashion.

\begin{definition}
Let $\mathcal{C}$ be a small $\infty$-category with finite coproducts and pullbacks along coproduct injections. Given 
a finite collection of objects $\{U_i\in\mathcal{C}\mid i\leq n\}$, for a pair $i,j\leq n$ consider the pullback
\[\xymatrix{
P_{i,j}\ar[r]\ar[d]\ar@{}[dr]|(.3){\pbs} & U_i\ar[d]^{\iota_i} \\
U_j\ar[r]_(.4){\iota_j} & \coprod_{i\leq n}U_i.
}\]
Coproducts in $\mathcal{C}$ are \emph{disjoint} if for every such finite collection 
of objects and every pair $i,j\leq n$, we have $P_{i,j}\simeq\emptyset$ whenever $i\not= j$. Coproducts in
$\mathcal{C}$ are \emph{universal} if for any such finite collection of objects and every map
$C\rightarrow B$ where $B\simeq\coprod_{i\in I}U_i$, the induced map $\coprod_{i\in I}(U_i\times_{B} C)\rightarrow C$ 
is an equivalence. The $\infty$-category $\mathcal{C}$ is \emph{extensive} if coproducts in $\mathcal{C}$ are both 
disjoint and universal.
\end{definition}

\begin{remark}
Universality of coproducts in an extensive $\infty$-category $\mathcal{C}$ implies that for every finite 
collection $U_i$ of objects in $\mathcal{C}$, the pullbacks $P_{i,i}$ are equivalent to $U_i$ (over $U_i$). This 
means that the coproduct injections $\iota_i\colon U_i\rightarrow\coprod_{i\leq n}U_i$ are $(-1)$-truncated.
Furthermore, the $0$-ary coproduct in an extensive $\infty$-category $\mathcal{C}$ exists by assumption. It is always 
a strict initial object as can be shown along the lines of the same statement for ordinary extensive categories, see 
e.g.\ \cite[Proposition 2.8]{clwextdist}.
\end{remark}

Note that an $\infty$-category $\mathcal{C}$ with finite coproducts has pullbacks along coproduct injections if and 
only if every colimiting discrete diagram $U\colon I\rightarrow\mathcal{C}_{/B}$ is a pre-descent diagram.

\begin{proposition}\label{propextdesc}
Let $\mathcal{C}$ be an $\infty$-category with finite coproducts and pullbacks along coproduct injections. Then the following are equivalent.
\begin{enumerate}
\item $\mathcal{C}$ is extensive.
\item Every finite discrete colimiting diagram $U\colon I\rightarrow\mathcal{C}_{/B}$ is a descent diagram. 
\item Every finite discrete colimiting diagram $U\colon I\rightarrow\mathcal{C}_{/B}$ is a higher covering diagram.
\end{enumerate}
\end{proposition}
\begin{proof}
By definition, $U\colon I\rightarrow\mathcal{C}_{/U}$ is a descent diagram if and only if the coproduct
$B=\coprod_{i\in I}U_i$ is universal and effective. To show the equivalence of 1 and 2, we hence are to compare 
effectiveness of this coproduct to disjointness thereof (under assumption of their universality). Therefore, if it is 
effective, for a given $j\in I $ we may consider the (cartesian) natural transformation $C_{i,j}\rightarrow U_i$ for 
$i\in I$ defined by $C_{i,j}=U_j$ if $i=j$ and $C_{i,j}=\emptyset$ otherwise. Each colimit $\coprod_{i\in I}C_{i,j}$ 
is equivalent to $U_j$ over $B$; by virtue of effectiveness, it follows that $C_{i,j}\simeq P_{i,j}$ over $U_j$ 
for all $j\in I$. This means that the coproduct $B$ is disjoint. Vice versa, given any (cartesian) natural 
transformation $\{C_i\rightarrow U_i\mid i\in I\}$, we are to show that for all $j\in I$ the natural map 
$C_j\rightarrow U_j\times_{B}\coprod_{i\in I}C_i$ is an equivalence. As the coproduct $B$ is assumed to be universal, 
the natural map $C_j\rightarrow\coprod_{i\in I}(U_i\times_{B}C_j)$ is an equivalence for all $j\in I$. But we have
$U_i\times_{B}C_j=C_j$ whenever $i=j$ as the coproduct injection $U_i\rightarrow B$ is $(-1)$-truncated, and we have  
$U_i\times_{B}C_j=\emptyset$ otherwise as the coproduct $B$ is disjoint and the initial object in $\mathcal{C}$ is 
strict.

To show the equivalence of 1 and 3, we note that every discrete diagram is well-indexed as all squares in a set 
trivial, all trivial squares are cartesian, and all diagrams preserve triviality of a square. Let us assume that 
coproducts in $\mathcal{C}$ are universal, i.e.\ equivalently that all discrete colimiting diagrams in
$\mathcal{C}$ are semi-descent diagrams. Under this assumption, we are to compare pre-diagonal stability of a 
discrete colimiting diagram $U\colon I\rightarrow\mathcal{C}_{/B}$ to disjointness of its coproduct. Such a diagram 
$U$ is pre-diagonally stable if and only if for all $i,j\in I$ the cocone
\[\Pi(U)(i,j)\colon \mathrm{Fun}_{(i,j)}(D^1,I)\rightarrow\mathcal{C}_{/U_i\times_B U_j}\]
is a semi-descent diagram. Indeed, the $\infty$-category $\mathrm{Fun}_{i,j}(D^1,I)$ is the singleton
$\Delta^0$ whenever $i=j$ and is empty otherwise. In particular, the higher pre-diagonals of $U$ trivialize. If 
$i=j$, the pre-diagonal $\Pi(U)(i,j)$ is the diagram
\[\{1_{U_i}\}\colon\Delta^0\rightarrow\mathcal{C}_{/U_i}\]
which is always a descent diagram. If $i\neq j$, the pre-diagonal $\Pi(U)(i,j)$ is the empty diagram
\[\emptyset\rightarrow\mathcal{C}_{/U_i\times_B U_j}.\]
The colimit of the empty diagram is the initial object $\emptyset\in\mathcal{C}_{/U_i\times_B U_j}$. This means that 
it is colimiting if and only if $\emptyset\simeq U_i\times_B U_j$.
\end{proof}

We obtain the following characterization of extensivity, which gives rise to a definitional pattern that applies 
directly to the structures considered in Sections~\ref{secregtop} and \ref{seccohtop} as well.

\begin{corollary}\label{corextdescalt}
For an $\infty$-category $\mathcal{C}$ the following are equivalent.
\begin{enumerate}
\item The $\infty$-category $\mathcal{C}$ is extensive.
\item For every finite discrete diagram
$U\colon I\rightarrow\mathcal{C}_{/C}$ there is a (unique) factorization
\[I\xrightarrow{U}\mathcal{C}_{/B}\xrightarrow{\Sigma_f}\mathcal{C}_{/C}\]
of $U$ such that $U\colon I\rightarrow\mathcal{C}_{/B}$ is a descent diagram.
\item For every finite discrete diagram
$U\colon I\rightarrow\mathcal{C}_{/C}$ there is a (unique) factorization
\[I\xrightarrow{U}\mathcal{C}_{/B}\xrightarrow{\Sigma_f}\mathcal{C}_{/C}\]
of $U$ such that $U\colon I\rightarrow\mathcal{C}_{/B}$ is a higher covering diagram.
\end{enumerate}
\end{corollary}
\begin{proof}
Straight-forward by Proposition~\ref{propextdesc}.
\end{proof}

For an object $B\in\mathcal{C}$ in an $\infty$-category $\mathcal{C}$ consider the class
\[\mathrm{Ext}(B)=\{U\colon I\rightarrow\mathcal{C}_{/B}\mid I\in\mathrm{Set}_{\omega}, U\text{ is colimiting}\},\]
and let $\mathrm{Ext}$ be the union of the classes $\mathrm{Ext}(B)$ for objects $B\in\mathcal{C}$.

\begin{proposition}
If $\mathcal{C}$ is an extensive $\infty$-category, the class $\mathrm{Ext}$ is a small well-structured colimit 
pre-topology.
\end{proposition}
\begin{proof}
Whenever $\mathcal{C}$ is extensive, we have
\[\mathrm{Ext}(B)=\{U\colon I\rightarrow\mathcal{C}_{/B}\mid I\in\mathrm{Set}_{\omega}, U\text{ is a higher covering diagram}\}\]
by Proposition~\ref{propextdesc}.
Thus, to show that $\mathrm{Ext}$ is a well-structured colimit pre-topology, by Theorem~\ref{prophcdmaxtop} we 
only are to show that discreteness of a diagram preserves reflexivity, stability under base change and
pre-diagonal stability. But the $\infty$-category $\Delta^0$ is discrete, and post-composition with a base change 
functor does not vary the domain of a diagram. Lastly, whenever $I$ is a finite set, then so is
$\mathrm{Fun}_{(i,j)}(D^1,I)$ for all $i,j\in I$. 
\end{proof}

It follows from Theorem~\ref{thmcolimtop} that for all extensive $\infty$-categories $\mathcal{C}$ the localization
$\hat{\mathcal{C}}\rightarrow\mathrm{Sh}_{\mathrm{Ext}}(\mathcal{C})$ is left exact (Definition~\ref{deftopsheaves}). 
While we made the effort to show that the results of Section 3 do apply, they certainly are an overkill to prove left 
exactness of this localization. Indeed, in the following we show that the localization
$\hat{\mathcal{C}}\rightarrow\mathrm{Sh}_{\mathrm{Ext}}(\mathcal{C})$ coincides with the localization at the 
extensive Grothendieck topology on $\mathcal{C}$ (to be defined below) and that the $\infty$-topos
$\mathrm{Sh}_{\mathrm{Ext}}(\mathcal{C})$ is hypercomplete whenever $\mathcal{C}$ is extensive. As a 
corollary we obtain that the $\infty$-topos $\mathrm{Sh}_{\mathrm{Ext}}(\mathcal{C})$ has enough points in the sense 
of \cite[Section 4]{luriedag7} whenever $\mathcal{C}$ itself is furthermore left exact.

\begin{definition}\label{defextgrtop}
Let $\mathcal{C}$ be an extensive $\infty$-category. A sieve $S$ over an object $B\in\mathcal{C}$ is an 
\emph{extensive cover} if it contains a finite family $(U_i\rightarrow B)_{i\leq n}$ of arrows which exhibits $B$ as 
the coproduct of the components $U_i$. The extensive Grothendieck topology on $\mathcal{C}$ is defined to be the 
smallest Grothendieck topology on $\mathcal{C}$ that contains all extensive covering sieves.
\end{definition}

By definition, the set of extensive covering sieves associated to an extensive $\infty$-category $\mathcal{C}$ is 
exactly the set of $(-1)$-truncations of the maps contained in the modulator $\mathrm{Cov}_{\mathrm{Ext}}$. 

\begin{lemma}\label{lemmaextsheaveschar}
Let $\mathcal{C}$ be an $\infty$-category. Then a presheaf $X\in\hat{\mathcal{C}}$ is 
$\mathrm{Cov}_{\mathrm{Ext}}$-local if and only if $X\colon\mathcal{C}^{op}\rightarrow\mathcal{S}$ preserves finite 
products. Thus, whenever $\mathcal{C}$ is small and extensive, the localization
$\mathrm{Sh}_{\mathrm{Ext}}(\mathcal{C})$ consists exactly of the sheaves for the extensive Grothendieck topology. In 
particular, the localization $\hat{\mathcal{C}}\rightarrow\mathrm{Sh}_{\mathrm{Ext}}(\mathcal{C})$ is topological.
\end{lemma}
\begin{proof}
On the one hand, the fact that the $\mathrm{Cov}_{\mathrm{Ext}}$-local presheaves are exactly the finite product 
preserving functors $X\colon\mathcal{C}^{op}\rightarrow\mathcal{S}$ holds by construction. On the other hand, if
$\mathcal{C}$ is extensive, the extensive Grothendieck topology on $\mathcal{C}$ is generated by the covering 
families given by finite coproduct injections $\{U_i\rightarrow B\mid i\leq n\}$ whenever
$B\simeq\coprod_{i\leq n}U_i$. As the corresponding sieves $S\hookrightarrow yB$ are exactly the $(-1)$-truncations 
of the maps $m\colon\coprod_{i\leq n}y(U_i)\rightarrow yB$, they can be computed by geometric realization of 
the \v{C}ech-nerves $\check{C}(m)$ of $m$ \cite[Proposition 6.2.3.4, Lemma 6.2.3.18]{luriehtt}. It follows that a 
presheaf $X$ is a sheaf for the extensive Grothendieck topology if and only if the natural map
\[X(B)\rightarrow\mathrm{lim}X(\hat{C}(m))\]
is an equivalence for all $m\in \mathrm{Cov}_{\mathrm{Ext}}(B)$, $B\in\mathcal{C}$. Using that coproducts in
$\mathcal{C}$ are disjoint and that all coproduct injections are monomorphisms, one shows that the limit on the right 
hand side is the product $\prod_{i\leq n}X(U_i)$. Thus, $X\colon\mathcal{C}^{op}\rightarrow\mathcal{S}$ is a sheaf 
for the extensive Grothendieck topology if and only if it preserves finite products.
\end{proof}

\begin{proposition}\label{propexthyper}
Let $\mathcal{C}$ be a small extensive $\infty$-category. Then the geometric inclusion
$\iota\colon\mathrm{Sh}_{\mathrm{Ext}}(\mathcal{C})\hookrightarrow\hat{\mathcal{C}}$ preserves sifted colimits. In 
particular, it preserves effective epimorphisms.
\end{proposition}

\begin{proof}
By Lemma~\ref{lemmaextsheaveschar} we are to show that a sifted colimit of finite limit preserving presheaves
is again finite limit preserving. As colimits of presheaves are computed pointwise, this reduces to the fact that finite limits commute with sifted colimits in the $\infty$-category of spaces \cite[Remark 5.5.8.12]{luriehtt}.
\end{proof}

\begin{corollary}\label{corexthyper}
Let $\mathcal{C}$ be a small extensive $\infty$-category. Then the $\infty$-topos
$\mathrm{Sh}_{\mathrm{Ext}}(\mathcal{C})$ is hypercomplete.
\end{corollary}
\begin{proof}
Recall that a map in an $\infty$-topos is $\infty$-connected if and only if all its higher diagonals are effective 
epimorphisms (this follows from \cite[Proposition 6.5.1.19]{luriehtt}). The inclusion
$\iota\colon\mathrm{Sh}_{\mathrm{Ext}}(\mathcal{C})\hookrightarrow\hat{\mathcal{C}}$ preserves finite limits, and we 
have seen in Proposition~\ref{propexthyper} that it preserves effective epimorphisms as well. Thus, if
$f\in\mathrm{Sh}_{\mathrm{Ext}}(\mathcal{C})$ is $\infty$-connected, then so is $\iota(f)\in\hat{\mathcal{C}}$. But 
presheaf $\infty$-toposes are hypercomplete (since an $\infty$-connected map in a presheaf $\infty$-category is 
pointwise $\infty$-connected and hence a (pointwise) equivalence by Whitehead's Theorem, see \cite[Remark 6.5.4.7]{luriehtt}), and so 
$\iota(f)$ is an equivalence. Thus, $f\in\mathrm{Sh}_{\mathrm{Ext}}(\mathcal{C})$ is an equivalence as well.
\end{proof}

\begin{corollary}\label{corextpoints}
Let $\mathcal{C}$ be a small lextensive $\infty$-category, i.e.\ $\mathcal{C}$ is extensive and left exact. Then the
$\infty$-topos $\mathrm{Sh}_{\mathrm{Ext}}(\mathcal{C})$ has enough points. These are up to equivalence exactly the 
left exact and finite coproduct preserving functors of type $F\colon\mathcal{C}\rightarrow\mathcal{S}$.
\end{corollary}
\begin{proof}
The first statement follows immediately from Corollary~\ref{corexthyper} together with
\cite[Corollary 3.22]{luriedag7} and \cite[Theorem 4.1]{luriedag7}. The second statement is a standard argument via 
left Kan extension along the Yoneda embedding, see
\cite[Lemma 5.1.5.5, Proposition 5.5.4.20 and Proposition 6.1.5.2]{luriehtt}.
\end{proof}


\section{Coherent $\infty$-categories and their sheaves}\label{secregtop}

The most classical example of a Grothendieck topology is the ($\kappa$-)coherent topology on a ($\kappa$-)coherent
category, and hence the regular topology on a regular category in particular \cite{elephant}.
In this section we define straight-forward generalizations of these notions in the $\infty$-categorical context. We
show that the associated sheaf theories are each generated by a well-structured colimit pre-topology as defined in 
Section~\ref{secdesc}, specifically given by the class of all higher covering diagrams indexed by the sorted 
Lawvere theory of $I$-indexed collections of objects for $\kappa$-small sets $I$. 

\begin{remark}
In this section and the following sections, we depart in our definitions from the ``global'' conventions of the 
literature by dropping the assumption of a terminal object unless explicitly stated. We simply do so because the vast 
majority of the constructions merely requires left exactness of the slices of $\mathcal{C}$ rather than of
$\mathcal{C}$ itself. 
\end{remark}

\begin{definition}\label{defregcat}
An $\infty$-category $\mathcal{C}$ with pullbacks is \emph{locally regular} if for every morphism
$f\colon E\rightarrow B$ in $\mathcal{C}$ its \v{C}ech nerve
$\check{C}(f)\colon\Delta^{op}\rightarrow\mathcal{C}_{/B}$ admits a universal colimit $|\check{C}(f)|\rightarrow B$ such that the induced factorization $\check{C}(f)\colon\Delta^{op}\rightarrow\mathcal{C}_{/|\check{C}(f)|}$ is the
\v{C}ech nerve of the factorization $f\colon E\rightarrow|\check{C}(f)|$.
\end{definition}

Universality of the colimit of the \v{C}ech nerve $\check{C}(f)\colon\Delta^{op}\rightarrow\mathcal{C}_{/B}$ 
associated to a morphism $f\colon E\rightarrow B$ expresses that for every map 
$g\colon C\rightarrow B$, the natural map $|\check{C}(g^{\ast}f)|\rightarrow g^{\ast}|\check{C}(f)|$ is an 
equivalence. In other words, for every morphism $f\colon E\rightarrow B$ in $\mathcal{C}$ its \v{C}ech nerve
$\check{C}(f)\colon\Delta^{op}\rightarrow\mathcal{C}_{/|\check{C}(f)|}$ is a semi-descent diagram. The latter 
condition in Definition~\ref{defregcat} states that the underlying internal groupoid of any \v{C}ech nerve in
$\mathcal{C}$ is effective \cite[Definition 6.1.2.14]{luriehtt}.
We note that the following standard definitions can be expressed in any locally regular $\infty$-category.

\begin{definition}\label{defeffepi}
A map $f\colon E\rightarrow B$ in a locally regular $\infty$-category $\mathcal{C}$ is an \emph{effective 
epimorphism} if $|\check{C}(f)|\rightarrow B$ is an equivalence. A map $f$ in a locally regular $\infty$-category
$\mathcal{C}$ is \emph{$\infty$-connected} if all its higher diagonals (including the $0$-th) are effective 
epimorphisms. An object $C$ in $\mathcal{C}$ is \emph{hypercomplete} if $C$ is local with respect to all
$\infty$-connected maps in $\mathcal{C}$.
\end{definition}

Certainly, every $\infty$-topos $\mathcal{C}$ is regular \cite[Remark 6.2.3.2]{luriehtt}. In this case, a map $f$ in
$\mathcal{C}$ is $\infty$-connected if and only if it is an effective epimorphism and all its internal homotopy 
groups vanish \cite[Proposition 6.5.1.18]{luriehtt}. The latter condition is in fact the definition of
$\infty$-connectedness in \cite{luriehtt}. Following the proof of \cite[Proposition 6.2.3.4]{luriehtt}, one 
sees that the natural map $|\check{C}(f)|\rightarrow B$ is always $(-1)$-truncated. It also follows that the map
$E\rightarrow |\check{C}(f)|$ is an effective epimorphism for every $f\colon E\rightarrow B$ in $\mathcal{C}$.
In particular, the class of effective epimorphisms in a locally regular $\infty$-category $\mathcal{C}$ is stable 
under base change and contains all equivalences. Furthermore, the pair of effective
epimorphisms and $(-1)$-truncated maps form a factorization system on $\mathcal{C}$. The $(-1)$-truncated morphism 
$f_{-1}\colon|\check{C}(f)|\hookrightarrow B$ associated to an arrow $f\colon E\rightarrow B$ is the
$(-1)$-truncation of $f$ in $\mathcal{C}$. That means, the inclusion $\mathrm{Sub}(B)\hookrightarrow\mathcal{C}_{/B}$ 
of the poset of $(-1)$-truncated objects in $\mathcal{C}_{/B}$ exhibits a right adjoint
\begin{align}\label{equdeftruncfunc}
\tau_{-1}\colon\mathcal{C}_{/B}\rightarrow\mathrm{Sub}(B)
\end{align}
that maps an arrow $f$ to the $(-1)$-truncation $f_{-1}$.

More generally, one considers the following $\kappa$-many-object version of regularity.
For a fixed regular cardinal $\kappa$ and any $\kappa$-small set $I$ (i.e.\ a set of size \emph{strictly} less than
$\kappa$) consider the simplicial $\kappa$-small set
\[ I^{|(\cdot)|}\colon\Delta^{op}\rightarrow\mathrm{Set}_{\kappa}\]
and its Grothendieck construction $\sum_{[n]\in\Delta^{op}}I^{|[n]|}$ discretely fibered over $\Delta$. 
Let $\mathrm{FinSet}_{+}$ denote the category of non-empty finite sets and let
$\sigma\colon\Delta\rightarrow\mathrm{FinSet}_{+}$ be the canonical inclusion (which is bijective on objects). 
The simplicial set $I^{|[\cdot]|}\colon\Delta^{op}\rightarrow\mathrm{Set}_{\kappa}$ admits an extension along
$\sigma$ to a symmetric simplicial set
\[I^{|[\cdot]|}\colon\mathrm{FinSet}_{+}^{op}\rightarrow\mathrm{Set}_{\kappa}\]
as can be directly seen via \cite[Theorem 4.2]{grandissym} by mapping the main transpositions of a non-empty finite 
set $[n]$ to the according permutations of components of tuples in $I^{|[n]|}$. (Existence of this extension however 
also follows from Lemma~\ref{lemmasymnerverkan} below, as the simplicial object
$I^{|[\cdot]|}\colon\Delta^{op}\rightarrow\mathrm{Set}_{\kappa}$ itself is the \v{C}ech nerve of the function 
$I\rightarrow\ast$ in $\mathrm{Set}_{\kappa}$, and the latter has pullbacks.) We obtain the following pullback of 
discretely fibered Grothendieck constructions.
\begin{align}\label{diagpropcohcovcotoptop}
\begin{gathered}
\xymatrix{
\sum\limits_{[n]\in\Delta^{op}}I^{|[n]|}\ar@{^(->}[r]^{}\ar@{->>}[d]\ar@{}[dr]|(.3){\pbs} & \sum\limits_{[n]\in\mathrm{FinSet}_{+}^{op}}I^{|[n]|}\ar@{->>}[d] \\
\Delta^{op}\ar@{^(->}[r]_{\sigma} & \mathrm{FinSet}_{+}^{op}
}
\end{gathered}
\end{align}

\begin{notation}
In the following we denote the (bijective on objects) top inclusion by
$\sigma\colon \Delta(I)^{op}\hookrightarrow\mathrm{FS}(I)_+^{op}$.
\end{notation}

\begin{lemma}\label{lemmamultinerve}
Let $\mathcal{C}$ be an $\infty$-category with pullbacks and let $I$ be a set. To every family
\begin{align}\label{equlemmamultinervefam}
F=\{E_i\rightarrow B\mid i\in I\}
\end{align}
of arrows with a common base $B$ in $\mathcal{C}$ -- considered as a discrete diagram
$F\colon I\rightarrow\mathcal{C}_{/B}$ -- the right Kan extension of $F$ along the canonical 
embeddings
\begin{align}\label{equlemmamultinerve}
\xymatrix{
\{0\}\times I\ar@{^(->}[r]_{\iota_{\Delta}}\ar@/^2pc/@{^(->}[rr]^{\iota_{\mathrm{FS}}} & \Delta(I)^{op}\ar@{^(->}[r]_{\sigma} & \mathrm{FS}_+(I)^{op}
}
\end{align}
exist and restrict back to $F$ up to equivalence. In both cases, if we denote the according right Kan extension again 
by $F$, the natural map
\begin{align}\label{explecohtopequ1}
F([n],\vec{i})\rightarrow F([0],i_0)\times_{C}\dots\times_{C}F([0],i_n)
\end{align}
induced by the points $\{j\}\colon[0]\rightarrow [n]$ for $j\leq n$ is an equivalence for all $n\geq 0$ and
$\vec{i}\in I^{|[n]|}$. A given functor $U\colon\mathrm{FS}_+(I)^{op}\rightarrow\mathcal{C}_{/B}$
is the right Kan extension of its restriction along $\iota_{\mathrm{FS}}$ if and only if it preserves non-empty 
finite products.
\end{lemma}

\begin{proof}
We formulate the proof for the composite inclusion $\iota_{\mathrm{FS}}$ in (\ref{equlemmamultinerve}); the proof for 
the inclusion $\iota_{\Delta}$ is completely analogous (barring the last statement). Thus, we first 
note that the inclusion $\iota_{\mathrm{FS}}\colon I\hookrightarrow\mathrm{FS}_+(I)^{op}$ is fully faithful. As $I$ 
is discrete, so is the under-category
$([n],\vec{i})_{/\iota_{\mathrm{FS}}}:=(\{[0]\}\times I)\times_{\mathrm{FS}(I)}\mathrm{FS}(I)_{([n],\vec{i})/}$ for 
every object $([n],\vec{i})\in\mathrm{FS}_+(I)^{op}$. This under-category is furthermore finite, because the tuple
$\vec{i}$ has finite length. Since $\mathcal{C}$ has pullbacks, the slice $\mathcal{C}_{/C}$ has products. For any 
given family (\ref{equlemmamultinervefam}), it follows that for 
every $([n],\vec{i})\in\mathrm{FS}_+(I)^{op}$, the functor
\[([n],\vec{i})_{/\iota}\twoheadrightarrow \{[0]\}\times I\xrightarrow{F([0],-)}\mathcal{C}_{/C}\]
has a limit in $\mathcal{C}_{/C}$. By \cite[Lemma 4.3.2.13]{luriehtt} it follows that $F$ admits a pointwise right 
Kan extension $F\colon\mathrm{FS}_+(I)^{op}\rightarrow\mathcal{C}_{/C}$ along $\iota_{\mathrm{FS}}$. 
By \cite[Definition 4.3.2.2]{luriehtt}, for all tuples $([n],\vec{i})$ we have equivalences
\begin{align}\label{diagpropcohcovcotoptop2}
\notag F([n],\vec{i}) & \simeq F(\mathrm{lim}\left(([n],\vec{i})_{/\iota_{\mathrm{FS}}}\rightarrow \{[0]\}\times I\hookrightarrow\mathrm{FS}_+(I)^{op}\right))\\
& \simeq \mathrm{lim}\left(([n],\vec{i})_{/\iota_{\mathrm{FS}}}\rightarrow \{[0]\}\times I\hookrightarrow\mathrm{FS}_+(I)^{op}\xrightarrow{F}\mathcal{C}_{/C}\right) \\
\notag & \simeq F([0],i_0)\times_C\dots\times_C F([0],i_n).
\end{align}
In particular, the restriction $F|_{\{0\}\times I}$ is equivalent to the original family $F$.
Furthermore, for any tuple $([n],\vec{i})$, the limit of the composition
\[([n],\vec{i})_{/\iota_{\mathrm{FS}}}\rightarrow \{[0]\}\times I\hookrightarrow\mathrm{FS}_+(I)^{op}\]
is just $([n],\vec{i})$ itself. In other words, the full sub-$\infty$-category $\{0\}\times I$ generates
$\mathrm{FS}_+(I)^{op}$ under non-empty finite products. It follows that whenever
$U\colon\mathrm{FS}_+(I)^{op}\rightarrow\mathcal{C}$ is any non-empty finite product-preserving functor, then
$U$ is the right Kan extension of its restriction $U|_{\{0\}\times I}$.
\end{proof}

\begin{notation}\label{notsymcnerve}
We will refer to the right Kan extension of a family $F=\{E_i\rightarrow B\mid i\in I\}$ from 
Lemma~\ref{lemmamultinerve} along $\iota_{\Delta}\colon I\hookrightarrow\Delta(I)^{op}$ as the \v{C}ech nerve
$\check{C}(F)\colon\Delta(I)^{op}\rightarrow\mathcal{C}_{/B}$ of $F$. We will refer to its right Kan extension along
$\iota_{\mathrm{FS}}\colon I\hookrightarrow\mathrm{FS}_+(I)^{op}$ as the symmetric \v{C}ech nerve
$\check{\Sigma}(F)\colon\Delta(I)^{op}\rightarrow\mathcal{C}_{/B}$ of $F$. 
\end{notation}

By construction, Notation~\ref{notsymcnerve} recovers (defines) the (symmetric) \v{C}ech nerve of a single arrow
$f\colon E\rightarrow B$ in $\mathcal{C}$ whenever $I$ has cardinality $1$. 

\begin{lemma}\label{lemmasymnerverkan}
Let $\mathcal{C}$ be an $\infty$-category with pullbacks and $I$ be a set. For any family
$F=\{E_i\rightarrow B\mid i\in I\}$ of arrows in $\mathcal{C}$, the symmetric 
\v{C}ech nerve $\check{\Sigma}(F)\colon\mathrm{FS}_+(I)^{op}\rightarrow\mathcal{C}_{/B}$ is the right Kan extension 
of the \v{C}ech nerve $\check{C}(F)\colon\Delta(I)^{op}\rightarrow\mathcal{C}_{/B}$ of $F$ along
$\sigma\colon \Delta(I)^{op}\hookrightarrow\mathrm{FS}(I)_+^{op}$. In particular,
$\sigma^{\ast}\check{\Sigma}(F)\simeq\check{C}(F)$.
\end{lemma}
\begin{proof}
Both $\check{\Sigma}(F)$ and $\check{C}(F)$ are the global right Kan extension of
$F\colon I\rightarrow\mathcal{C}_{/B}$ along the inclusions $\iota_{\mathrm{FS}}$ and $\iota_{\Delta}$, respectively.
It formally follows that $\check{\Sigma}(F)$ is the pointwise right Kan extension of $\check{C}(F)$ simply because
$\iota_{\mathrm{FS}}=\sigma\iota_{\Delta}$. That is, briefly, because for any pair of functors
$L_1\colon\mathcal{D}\rightarrow\mathcal{E}$ and $L_2\colon\mathcal{E}\rightarrow\mathcal{F}$, if $L_2$ has a right 
adjoint $R_2$ and the composition $L_2L_1$ has a right adjoint $R_{12}$, then $R_{12}$ is an $R_1$-relative right 
adjoint of $L_1$. It follows that $\sigma^{\ast}\check{\Sigma}(F)\simeq\check{C}(F)$, as $\sigma^{\ast}$ preserves 
limits and both the right Kan extension $\check{C}(F)$ as well as the right Kan extension $\check{\Sigma}(F)$ are 
determined by (\ref{explecohtopequ1}).
\end{proof}

\begin{definition}\label{explecohtop}
Let $\mathcal{C}$ be an $\infty$-category with pullbacks and let $I$ be a set. A family
$F=\{E_i\rightarrow B\mid i\in I\}$ of arrows in $\mathcal{C}$ with common base $B$ is \emph{jointly 
effective epic} if $\check{C}(F)\colon \Delta(I)^{op}\rightarrow \mathcal{C}_{/B}$ is colimiting. The family $F$ is 
\emph{universally jointly effective epic} if $\check{C}(F)\colon \Delta(I)^{op}\rightarrow \mathcal{C}_{/B}$ is a 
semi-descent diagram.
\end{definition}

\begin{proposition}\label{lemmaocdran}
Let $\mathcal{C}$ be an $\infty$-category with pullbacks. For every set $I$ and every object $B\in\mathcal{C}$, 
restriction along the inclusion
$\sigma\colon\Delta(I)^{op}\hookrightarrow\mathrm{FS}(I)_+^{op}$ induces a bijection between 
the class
\[\{\check{C}(F)\colon\Delta(I)^{op}\rightarrow\mathcal{C}_{/B}\mid F\colon I\rightarrow\mathcal{C}_{/B}\text{ is universally jointly effective epic}\},\] 
and
the class
\[\{U\colon\mathrm{FS}(I)_+^{op}\rightarrow\mathcal{C}_{/B}\mid U\text{ is a higher
covering diagram}\}.\] 
\end{proposition}
\begin{proof}
By Lemma~\ref{lemmamultinerve} we are to show that, first, a family $F=\{E_i\rightarrow B\mid i\in I\}$
of arrows with a common base $B$ is jointly effective epic if and only if its right Kan extension along the composite
embedding
\[\iota\colon\{0\}\times I\hookrightarrow\Delta(I)^{op}\hookrightarrow\mathrm{FS}_+(I)^{op}\]
is a higher covering diagram, and, second, that every $\mathrm{FS}_+(I)^{op}$-indexed higher covering diagram arises 
in this way.

The inclusion $\sigma\colon\Delta^{op}\hookrightarrow\mathrm{FinSet}_{+}^{op}$ is cofinal as can be shown by the same 
proof of \cite[Lemma 6.5.3.7]{luriehtt}. Since Kan fibrations are smooth \cite[Proposition 4.1.2.15]{luriehtt}, it 
follows that the pullback $\sigma\colon \Delta(I)^{op}\hookrightarrow\mathrm{FS}(I)^{op}$ in 
(\ref{diagpropcohcovcotoptop}) is cofinal as well \cite[Remark 4.1.2.10]{luriehtt}. In particular, by 
Lemma~\ref{lemmadesccofstable} restriction along $\sigma$ preserves and reflects semi-descent diagrams.

Let us first show that every higher covering diagram $U\colon\mathrm{FS}_+(I)^{op}\rightarrow\mathcal{C}_{/B}$
is the symmetric \v{C}ech nerve of the family $\sigma^{\ast}U\colon I\rightarrow\mathcal{C}_{/B}$, and that the latter is universally jointly effective epic over $B$.
Therefore, we note that the presheaf $I^{|[\cdot]|}\colon\mathrm{FinSet_+^{op}}\rightarrow\mathrm{Set}$ is an indexed 
category with pullbacks and non-empty finite products whose domain $\mathrm{FinSet_+^{op}}$ has pullbacks 
and non-empty finite products. Along the lines of Remark~\ref{remrfibwellind} it follows that the associated total 
category $\mathrm{FS}_+(I)^{op}$ has finite non-empty products and pullbacks, too. 
Hence, the $\mathrm{FS}_+(I)^{op}$-indexed higher covering diagrams are exactly the non-empty finite product and 
pullback-preserving semi-descent diagrams (Example~\ref{explecovprod}). In particular, every higher covering diagram 
$U\colon\mathrm{FS}_+(I)^{op}\rightarrow\mathcal{C}_{/B}$ is the right Kan extension
$\check{\Sigma}(U|_{\{0\}\times I})$ (Lemma~\ref{lemmamultinerve}). The \v{C}ech nerve
$\check{C}(U|_{\{0\}\times I})\colon\Delta(I)^{op}\rightarrow\mathcal{C}_{/B}$ is the restriction of $U$ along
$\sigma$ by Lemma~\ref{lemmasymnerverkan}. Thus, as $\sigma$ is cofinal and $U$ is a semi-descent diagram
by assumption, the family $U|_{\{0\}\times I}$ is universally jointly effective epic by Lemma~\ref{lemmadesccofstable}.

We are left to show that $\check{\Sigma}(F)\colon\mathrm{FS}_+(I)^{op}\rightarrow\mathcal{C}_{/B}$ is higher covering 
whenever $F\colon I\rightarrow\mathcal{C}_{/B}$ is a family of arrows such that
$\check{C}(F)\colon\Delta(I)^{op}\rightarrow\mathcal{C}_{/B}$ is a semi-descent diagram. Therefore, we use 
that all objects in the discrete full subcategory $\{0\}\times I$ are small-injective in $\mathrm{FS}_+(I)^{op}$ 
with respect to non-empty finite products and pullbacks. That means, given any diagram 
$G\colon J\rightarrow\mathrm{FS}_+(I)^{op}$ for $J$ a finite non-empty set or the free co-span, let
$G_{/\iota_{\mathrm{FS}}}$ be the composite
\[J^{op}\xrightarrow{G^{op}}\mathrm{FS}_+(I)\xrightarrow{-_{/\iota_{\mathrm{FS}}}}\mathrm{Set}.\]
Then the natural map
\[\mathrm{colim}(G_{/\iota_{\mathrm{FS}}})\rightarrow(\mathrm{lim}G)_{/\iota_{\mathrm{FS}}}\]
is an equivalence (of sets). Via the formula (\ref{diagpropcohcovcotoptop2}), it follows that the right Kan extension 
$\check{\Sigma}(F)\colon\mathrm{FS}_+(I)^{op}\rightarrow\mathcal{C}_{/B}$ of any family
$F\colon I\rightarrow\mathcal{C}_{/B}$ of arrows preserves all such limits. Furthermore, $\check{\Sigma}(F)$ is a 
semi-descent diagram whenever $F$ universally jointly effective epic by cofinality of $\sigma$. Thus, $\check{\Sigma}(F)$ is a 
higher covering diagram whenever $F$ is universally jointly effective epic by Example~\ref{explecovprod}.
\end{proof}

\begin{theorem}\label{thmocdsheaves}
Suppose $\mathcal{C}$ is a small $\infty$-category with pullbacks and $\kappa$ is a regular cardinal. Then the family 
of sets
\[\mathrm{Coh}_{\kappa}(B):=\{U\colon\mathrm{FS}(I)_+^{op}\rightarrow\mathcal{C}_{/B}\mid I\in\mathrm{Set}_{\kappa}, U\text{ is a higher covering diagram}\}\] 
for $B\in\mathcal{C}$ is a small well-structured colimit pre-topology on $\mathcal{C}$. The according localization
$\hat{\mathcal{C}}\rightarrow\mathrm{Sh}_{\mathrm{Coh}_{\kappa}}(\mathcal{C})$ is generated by the $\kappa$-small 
universally jointly effective epic families in $\mathcal{C}$. In particular, the localization 
$\hat{\mathcal{C}}\rightarrow\mathrm{Sh}_{\mathrm{Coh}_{\kappa}}(\mathcal{C})$ is sub-canonical and topological.
\end{theorem}
\begin{proof}
The set $\mathrm{Coh}_{\kappa}$ consists of well-indexed semi-descent diagrams by construction.
Stability under base change follows directly from stability under base change of the class of higher covering 
diagrams (Theorem~\ref{prophcdmaxtop}). For reflexivity, we note that $\mathrm{FS}_+(I)^{op}$ is connected, so 
the colimit of the composition
\[\{1_{yB}\}\colon\mathrm{FS}_+(I)^{op}\rightarrow\Delta^0\xrightarrow{\{1_{B}\}}\mathcal{C}_{/C}\xrightarrow{y}\hat{\mathcal{C}}_{/yB}\]
is the object $1_{yB}$ for every object $B\in\mathcal{C}$. Thus, for all $B\in\mathcal{C}$, the diagrams
$\{1_B\}\colon\Delta^0\rightarrow\mathcal{C}_{/B}$ and
$\mathrm{FS}_+(I)^{op}\rightarrow\Delta^0\xrightarrow{1_{B}}\mathcal{C}_{/B}$ are cofinally equivalent. To show that 
$\mathrm{Coh}_{\kappa}$ is closed under pre-diagonals, it suffices to show that for all pairs
$([n],\vec{i}),([m],\vec{j})\in\mathrm{FS}_+(I)^{op}$, there is a finite product-preserving cofinal
functor
\[\mathrm{FS}_+(I)^{op}\rightarrow\mathrm{Fun}_{(([n],\vec{i}),([m],\vec{j}))}(D^1,\mathrm{FS}_+(I)^{op}).\]
Therefore we may simply use that $\mathrm{FS}_+(I)^{op}$ has all non-empty finite products. Indeed, for 
any $\infty$-category $J$ with non-empty finite products, and any two objects $i,j\in J$, the
$\infty$-category $\mathrm{Fun}_{(i,j)}(D^1,J)$ is equivalent to the slice $J_{/i\times j}$. The projection
$J_{/i\times j}\rightarrow J$ has a right adjoint
$J\rightarrow J_{/i\times j}$, and right adjoints preserve all limits and are cofinal. This finishes the proof of the 
fact that $\mathrm{Coh}_{\kappa}$ is a well-structured colimit topology on $\mathcal{C}$.

For the second statement, we show that the modulator $\mathrm{Cov}_{\mathrm{Coh}_{\kappa}}$ generates the sheaf 
theory for the $\kappa$-coherent Grothendieck topology.
Therefore, we note that the modulator $\mathrm{Cov}_{\mathrm{Coh}_{\kappa}}$ consists of monomorphisms, and in fact 
is the usual set of generating covering sieves for the $\kappa$-coherent Grothendieck topology on
$\mathcal{C}$. Indeed, given a universally jointly effective epic family $F$ over $B$, the colimit of the composition
$y\check{C}(F)=\check{C}(yF)\colon\mathrm{\Delta}(I)^{op}\rightarrow\hat{\mathcal{C}}_{/y(B)}$ can be computed by the 
colimit of its (global) left Kan extension along the cocartesian fibration
$p\colon\Delta(I)^{op}\twoheadrightarrow\Delta^{op}$. We thus compute that the colimit
$\mathrm{colim}(y\check{C}(F))\rightarrow yB$ is the colimit of the simplicial diagram
\begin{align}\label{explecohtopequ2}
\xymatrix{
\mathrm{Lan}_p(y\check{C}(F))_0\ar[r] & \mathrm{Lan}_p(y\check{C}(F))_1\ar@<.5ex>@/^/[l]\ar@<-.5ex>@/_/[l] \ar@/^/[r]\ar@/_/[r]& \mathrm{Lan}_p(y\check{C}(F))_2\ar[l]\ar@<1ex>@/^/[l]\ar@<-1ex>@/_/[l]\ar@{-->}@<.5ex>[r] & \dots\ar@{-->}@<.5ex>[l]
}
\end{align}
over $yB$. By \cite[Proposition 4.3.3.10]{luriehtt}, each $\mathrm{Lan}_p(y\check{C}(F))_n$ is the colimit of the 
restriction of $\check{C}(F)$ to the fiber $p^{-1}([n])=I^{|[n]|}$. I.e.,
$\mathrm{Lan}_p(y\check{C}(F))_n\simeq\coprod_{\vec{i}\in I^{|[n]|}}y\check{C}(F)([n],\vec{i})$. Using Condition
(\ref{explecohtopequ1}), we see that the simplicial object (\ref{explecohtopequ2}) is equivalent to the \v{C}ech 
nerve of $\mathrm{Lan}_p(y\check{C}(F))_0\simeq\coprod_{i\in I}y\check{C}(F)([0],i)$ over $yB$. Thus,
$\mathrm{colim}(y\check{C}(F))\simeq\left(\coprod_{i\in I}y F_i\right)_{-1}$ over $yB$, which is exactly the sieve 
generated by the $\kappa$-coherent cover $F=\{E_i\rightarrow B\mid i\in I\}$ via \cite[Lemma 6.2.3.18]{luriehtt}.

\end{proof}

The $\kappa$-coherent Grothendieck topology is most commonly considered on categories which themselves are 
$\kappa$-coherent. We therefore make the following definition.

\begin{definition}\label{defordgeocat}
An $\infty$-category $\mathcal{C}$ with pullbacks is \emph{locally $\kappa$-coherent} for some regular 
cardinal $\kappa$ if for every ($\kappa$-)small set $I$ and every family
$F=\{E_i\rightarrow B\mid i\in I\}$ of objects over some $B\in\mathcal{C}$ the \v{C}ech nerve
$\check{C}(F)\colon\Delta(I)^{op}\rightarrow\mathcal{C}_{/B}$ has a universal colimit $|\check{C}(F)|\rightarrow B$ 
such that the induced factorization $\check{C}(F)\colon\Delta(I)^{op}\rightarrow\mathcal{C}_{/|\check{C}(F)|}$ is the 
\v{C}ech nerve of the family $F=\{E_i\rightarrow |\check{C}(F)|\mid i\in I\}$. 
A locally $\kappa$-coherent $\infty$-category $\mathcal{C}$ is \emph{$\kappa$-coherent} if it has a terminal 
object. An $\infty$-category $\mathcal{C}$ is (locally) \emph{infinitary-coherent} if it is (locally)
$\kappa$-coherent for all regular cardinals $\kappa$.
\end{definition}

\begin{remark}\label{remdefordgeocat}
Definition~\ref{defordgeocat} is chosen so that the generalization of Definition~\ref{defregcat} is obvious. 
Similar to the single arrow case, the colimit $|\check{C}(F)|\rightarrow B$ for any family
$F=\{E_i\rightarrow B\mid i\in I\}$ is $(-1)$-truncated (whenever it exists).
Thereby one can show that an $\infty$-category $\mathcal{C}$ with pullbacks is locally ($\kappa$-)coherent if and 
only if it is locally regular and the subobject-posets $\mathrm{Sub}(B)$ for objects 
$B\in\mathcal{C}$ have pullback-stable ($\kappa$-)small unions. We will omit a proof and work directly with 
Definition~\ref{defordgeocat} in the following instead.
\end{remark}

In particular, all results for locally $\kappa$-coherent $\infty$-categories to be stated below apply to 
locally regular $\infty$-categories by considering $\kappa=2$.


\begin{theorem}\label{corcharordgeocat}
For an $\infty$-category $\mathcal{C}$ with pullbacks the following are equivalent.
\begin{enumerate}
\item The $\infty$-category $\mathcal{C}$ is locally $\kappa$-coherent.
\item For every $\kappa$-small set $I$ and every family $F\colon I\rightarrow\mathcal{C}_{/B}$ of arrows there is a (unique) factorization
\[\Delta(I)^{op}\xrightarrow{\check{C}(F)}\mathcal{C}_{/B}\xrightarrow{\Sigma_f}\mathcal{C}_{/C}\]
such that $\check{C}(F)\colon\Delta(I)^{op}\rightarrow\mathcal{C}_{/B}$ has descent with respect to the class of \v{C}ech nerves of $\kappa$-small families (Definition~\ref{defdescentclass}).
\item For every $\kappa$-small set $I$ and every non-empty finite product-preserving functor
$U\colon\mathrm{FS}_+(I)^{op}\rightarrow\mathcal{C}_{/C}$ there is a (unique) factorization
\[\mathrm{FS}_+(I)^{op}\xrightarrow{U}\mathcal{C}_{/B}\xrightarrow{\Sigma_f}\mathcal{C}_{/C}\]
such that $U\colon\mathrm{FS}_+(I)^{op}\rightarrow\mathcal{C}_{/B}$ is a higher covering diagram.
\end{enumerate}
\end{theorem}
\begin{proof}
Regarding the equivalence of 1 and 3, the only non-trivial step left to show is in the ``if'' direction. Namely, that 
under the given assumption, for every $\kappa$-small family $F=\{E_i\rightarrow C\mid i\in I\}$, the factorization of
$\check{\Sigma}(F)\colon\mathrm{FS}_+(I)^{op}\rightarrow\mathcal{C}_{/C}$ through a higher covering diagram
$\check{\Sigma}(F)\colon\mathrm{FS}_+(I)^{op}\rightarrow\mathcal{C}_{/B}$ exhibits the restriction
$\check{C}(F)\colon\Delta(I)^{op}\rightarrow\mathcal{C}_{/B}$ as the \v{C}ech nerve of $F=\{E_i\rightarrow B\mid i\in I\}$. This however follows directly from Lemma~\ref{lemmamultinerve} as
$\check{\Sigma}(F)\colon\mathrm{FS}_+(I)^{op}\rightarrow\mathcal{C}_{/B}$ preserves non-empty finite products.

The fact that 3 implies 2 follows from Theorem~\ref{thmocdsheaves}, Corollary~\ref{cordescentdiaganel}, 
and the fact that every \v{C}hech-nerve can be functorially and cofinally right Kan extended to a symmetric
\v{C}hech-nerve by Lemma~\ref{lemmamultinerve}. The right Kan extension preserves cartesianness of natural 
transformations, given that the squares induced by permutations of non-empty finite sets are automatically cartesian.

Let's show that 2 implies 3. Let $\check{C}(\mathcal{C})$ be the class of \v{C}ech nerves in $\mathcal{C}$. That is, 
the class of diagrams $\check{C}(F)\colon\Delta(I)^{op}\rightarrow\mathcal{C}_{/D}$ for $D\in\mathcal{C}$, $I$ a
$\kappa$-small set, and $F=\{E_i\rightarrow D\mid i\in I\}$ a family of objects.
Let $F\colon I\rightarrow\mathcal{C}_{/C}$ be a $\kappa$-small family of 
objects. By assumption, its \v{C}ech nerve factors through a semi-descent diagram
$\check{C}(F)\colon\Delta(I)^{op}\rightarrow\mathcal{C}_{/B}$ such that 
\[\mathrm{res}_{\check{C}(F)}\colon\mathcal{C}_{/B}\rightarrow\mathrm{Desc}_{\check{C}(\mathcal{C})}(\check{C}(F))\]
is an equivalence. Via Lemma~\ref{lemmadesccofstable} and Proposition~\ref{lemmaocdran}, one shows that the associated 
symmetric \v{C}ech nerve
$\check{\Sigma}(F)\colon\mathrm{FS}_+(I)^{op}\rightarrow\mathcal{C}_{/C}$ induces an equivalence
\begin{align}\label{diagcharordgeocat}
\mathrm{res}_{\check{\Sigma}(F)}\colon\mathcal{C}_{/B}\rightarrow\mathrm{Desc}_{\check{\Sigma}(\mathcal{C})}(\check{\Sigma}(F)),
\end{align}
where $\check{\Sigma}(\mathcal{C})$ is the class of symmetric \v{C}ech nerves in $\mathcal{C}$.
We are to show that the colimiting factorization
$\check{\Sigma}(F)\colon\mathrm{FS}_+(I)^{op}\rightarrow\mathcal{C}_{/B}$ is a higher covering diagram. Therefore, we 
first note that $\check{\Sigma}(F)\colon\mathrm{FS}_+(I)^{op}\rightarrow\mathcal{C}_{/B}$ is well-indexed. Indeed,
the original diagram $\check{\Sigma}(F)\colon\mathrm{FS}_+(I)^{op}\rightarrow\mathcal{C}_{/C}$ preserves both
non-empty finite products and pullbacks, and $\mathcal{C}_{/B}\rightarrow\mathcal{C}_{/C}$ reflects connected limits; 
it follows that $\check{\Sigma}(F)\colon\mathrm{FS}_+(I)^{op}\rightarrow\mathcal{C}_{/B}$ preserves pullbacks. 
Thus, to show that it is higher covering it suffices to show that the digram also preserves non-empty finite products 
(Example~\ref{explecovprod}). Therefore, let $([n],\vec{i})\in\mathrm{FS}_+(I)^{op}$ be an object. We show that the 
canonical natural transformation
\[\varepsilon\colon\check{\Sigma}(F)(([n],\vec{i})\times -)\rightarrow\check{\Sigma}(F)([n],\vec{i})\times \check{\Sigma}(F)(-)\]
in $\mathrm{Fun}(\mathrm{FS}_+(I)^{op},\mathcal{C}_{/B})$ is an equivalence. Therefore, we note that the product 
projections induce a triangle
\[\xymatrix{
\check{\Sigma}(F)(([n],\vec{i})\times -)\ar[rr]^{\varepsilon}\ar@/_1pc/[dr]_{F\pi_2} & & \check{\Sigma}(F)([n],\vec{i})\times \check{\Sigma}(F)(-)\ar@/^1pc/[dl]^{\pi_2} \\
 & \check{\Sigma}(F)(-) & 
}\]
of cartesian natural transformations in $\mathrm{Fun}(\mathrm{FS}_+(I)^{op},\mathcal{C}_{/B})$. The colimit of the 
domain
$\check{\Sigma}(F)(([n],\vec{i})\times -)$ is the colimit of the composition
\[\mathrm{FS}_+(I)^{op}\xrightarrow{([n],\vec{i})\times -}\mathrm{FS}_+(I)^{op}_{/([n],\vec{i})}\xrightarrow{\check{\Sigma}(F)_{/([n],\vec{i})}}\mathcal{C}_{/\check{\Sigma}(F)([n],\vec{i})}\rightarrow\mathcal{C}_{/B}.\]
The product functor $([n],\vec{i})\times -$ is a right adjoint and hence cofinal. The slice $\mathrm{FS}_+(I)^{op}_{/([n],\vec{i})}$ has a terminal object. It follows that the colimit of this composition is exactly
$1_{\check{\Sigma}(F)([n],\vec{i})}$ in $\mathcal{C}_{/\check{\Sigma}(F)([n],\vec{i})}$, or
$\check{\Sigma}(F)([n],\vec{i})$ in $\mathcal{C}_{/B}$ equivalently.

The colimit of the codomain $\check{\Sigma}(F)([n],\vec{i})\times \check{\Sigma}(F)(-)$ of $\varepsilon$ is 
also $\check{\Sigma}(F)([n],\vec{i})$, because $\check{\Sigma}(F)$ is a semi-descent diagram. As $\varepsilon$ 
itself is factors through a natural transformation in
$\mathrm{Fun}(\mathrm{FS}_+(I)^{op},\mathcal{C}_{/\check{\Sigma}(F)([n],\vec{i})})$, the colimit of $\varepsilon$ 
is the identity on $\check{\Sigma}(F)([n],\vec{i})$. Thus, $\varepsilon$ is a natural equivalence itself by 
the fact that the functor (\ref{diagcharordgeocat}) is an equivalence, presuming that both domain and codomain of
$\varepsilon$ are \v{C}ech nerves themselves. However, on the one hand, the domain factors through the composition
\[\mathrm{FS}_+(I)^{op}\xrightarrow{([n],\vec{i})\times -}\mathrm{FS}_+(I)^{op}_{/([n],\vec{i})}\xrightarrow{\check{\Sigma}(F)_{/([n],\vec{i})}}\mathcal{C}_{/\check{\Sigma}(F)([n],\vec{i})}\]
as noted above. This composition preserves non-empty finite products and hence is a symmetric \v{C}ech nerve by 
Lemma~\ref{lemmamultinerve}. On the other hand, the codomain factors through the composition
\[\mathrm{FS}_+(I)^{op}\xrightarrow{\check{\Sigma}(F)}\mathcal{C}_{/B}\xrightarrow{\check{\Sigma}(F)([n],\vec{i})\times -}\mathcal{C}_{/\check{\Sigma}(F)([n],\vec{i})}\]
which preserves non-empty finite products, too, and hence is a \v{C}ech nerve as well.
\end{proof}

\begin{remark}\label{remocdlvthy}
By Lemma~\ref{lemmamultinerve} and Proposition~\ref{lemmaocdran}, a jointly effective epic family $F$ over an object $B$ in 
a locally $\kappa$-coherent $\infty$-category $\mathcal{C}$ is essentially the same structure as a non-empty finite 
product preserving diagram $\mathrm{FS}_+(I)^{op}\rightarrow\mathcal{C}_{/B}$. Every such diagram can be 
extended uniquely to a finite product preserving functor from the $I$-sorted Lawvere theory $\mathrm{FS}(I)^{op}$ of
$I$-indexed collections of objects simply by mapping the terminal object $\emptyset\in\mathrm{FS}(I)^{op}$ to the 
terminal object $1_B\in\mathcal{C}_{/B}$. The latter is exactly a $\mathrm{FS}(I)^{op}$-algebra in
$\mathcal{C}_{/B}$. Thus, according to Theorem~\ref{corcharordgeocat}, an $\infty$-category $\mathcal{C}$ with 
pullbacks is $\kappa$-coherent if and only if for all $I\in\mathrm{Set}_{\kappa}$, for all $B\in\mathcal{C}$, and for 
all $\mathrm{FS}(I)^{op}$-algebras $T$ in $\mathcal{C}_{/B}$, the colimit of the restriction
$T_+\colon\mathrm{FS}_+(I)^{op}\rightarrow\mathcal{C}_{/B}$ exists, and the canonical extension
$T\colon\mathrm{FS}(I)^{op}\rightarrow\mathcal{C}_{/\mathrm{colim}T_+}$ is again an $\mathrm{FS}(I)^{op}$-algebra.
\end{remark}

\begin{corollary}
Suppose $\mathcal{C}$ is a small locally $\kappa$-coherent $\infty$-category. Then the localization
$\mathrm{Sh}_{\mathrm{Coh}_{\kappa}}(\mathcal{C})$ is the $\infty$-topos of $\kappa$-coherent sheaves 
associated to the $\kappa$-coherent Grothendieck topology on $\mathcal{C}$.
\end{corollary}
\begin{proof}
Immediate by Theorem~\ref{thmocdsheaves}.
\end{proof}

In contrast to Corollary~\ref{corextpoints} and Corollary~\ref{corexthyper} in the extensive case, we have the 
following proposition. Therefore, if $\mathcal{C}$ is a $\kappa$-coherent $\infty$-category, we note 
that the points of $\mathrm{Sh}_{\mathrm{Coh}_{\kappa}}(\mathcal{C})$ are (up to 
equivalence) exactly the left exact functors $\mathcal{C}\rightarrow\mathcal{S}$ which preserve jointly effective 
epic families of size less than $\kappa$.

\begin{proposition}\label{propregpts}
Let $\mathcal{C}$ be a small locally $\kappa$-coherent $\infty$-category for some regular cardinal
$\kappa\geq 2$. Every non-trivial $\infty$-connected map $f$ in $\mathcal{C}$ induces a non-trivial
$\infty$-connected map $yf$ in $\mathrm{Sh}_{\mathrm{Coh}_{\kappa}}(\mathcal{C})$. In particular, the
$\infty$-topos of $\kappa$-coherent sheaves on $\mathcal{C}$ is generally not hypercomplete, and hence 
does generally not have enough points.
\end{proposition}

\begin{proof}
We note that the Yoneda embedding $y\colon\mathcal{C}\rightarrow\mathrm{Sh}_{\mathrm{Coh}_{\kappa}}(\mathcal{C})$ 
preserves both pullbacks and $\kappa$-small jointly effective epic families. In particular, it preserves effective 
epimorphisms. Indeed, for an effective epimorphism $f\colon E\rightarrow B$ in
$\mathcal{C}$, the sequence $yE\rightarrow|\check{C}(yf)|\rightarrow yB$ factors $yf$ in $\hat{\mathcal{C}}$ into an 
effective epimorphism followed by a $\mathrm{Cov}_{\mathrm{Coh}_{\kappa}}$-local monomorphism. Since the 
localization $\hat{\mathcal{C}}\rightarrow\mathrm{Sh}_{\mathrm{Coh}_{\kappa}}(\mathcal{C})$ preserves pullbacks and 
colimits, it preserves effective epimorphisms, and so the map $yf$ is equivalent to an effective epimorphism in
$\mathrm{Sh}_{\mathrm{Coh}_{\kappa}}(\mathcal{C})$. In particular, it preserves $\infty$-connected maps. 
Furthermore, the localization $\mathrm{Sh}_{\mathrm{Coh}_{\kappa}}(\mathcal{C})$ is sub-canonical by 
Theorem~\ref{thmocdsheaves}. Thus, whenever $\mathcal{C}$ exhibits a non-hypercomplete object $E$, the 
representable $yE$ is non-hypercomplete in $\mathrm{Sh}_{\mathrm{Coh}_{\kappa}}(\mathcal{C})$. Such a
$\kappa$-coherent $\infty$-category $\mathcal{C}$ is given for instance by the $\infty$-category of
$\lambda$-compact object in the Dugger-Hollander-Isaksen $\infty$-topos \cite[Section 11.3]{rezkhtytps} for any 
regular cardinal $\lambda\geq\kappa$ large enough. As hypercompleteness is a necessary condition for an
$\infty$-topos to have enough points \cite[Remark 6.5.4.7]{luriehtt}, the second statement follows.
\end{proof}

\section{Higher geometric sheaves}\label{seccohtop}

We have seen in Theorem~\ref{prophcdmaxtop} that the class of higher covering diagrams is the largest
well-structured colimit pre-topology on any $\infty$-category $\mathcal{C}$ with pullbacks. In this section we study 
the basic properties of the ``higher $\kappa$-geometric ''sheaf theory associated to the class
$\mathrm{Geo}_{\kappa}$ of $\kappa$-small higher covering diagrams in suitable $\infty$-categories $\mathcal{C}$. We 
show that it is generally neither topological nor hypercomplete. Instead, its topological part is given by the
$\infty$-topos of $\kappa$-coherent sheaves (whenever $\mathcal{C}$ is locally $\kappa$-coherent and $\kappa$ is 
uncountable). When $\mathcal{C}$ is an $\infty$-topos, we show that it recovers 
Lurie's notion of sheaves on an $\infty$-topos \cite[Notation 6.3.5.16]{luriehtt}. We will show however that there 
are $\infty$-toposes $\mathcal{C}$ which admit infinitary-coherent sheaves over themselves which are not higher 
geometric. This in particular shows that the infinitary-coherent sheaf theory on an $\infty$-topos is generally not 
canonical. \\

\begin{definition}
Let $\mathcal{C}$ be an $\infty$-category with pullbacks. We refer to the $\infty$-category
$\mathrm{Sh}_{\mathrm{Geo}}(\mathcal{C})\subseteq\hat{\mathcal{C}}$ of
$\mathrm{Cov}_{\mathrm{Geo}(\mathcal{C})}$-local presheaves as the higher geometric sheaf theory of $\mathcal{C}$.
Its objects will be referred to as higher geometric sheaves on $\mathcal{C}$.
\end{definition}

Even if $\mathcal{C}$ is a small $\infty$-category, the structured colimit pre-topology
$\mathrm{Geo}(\mathcal{C})$ may still be large, and so the $\infty$-category
$\mathrm{Sh}_{\mathrm{Geo}}(\mathcal{C})\subseteq\hat{\mathcal{C}}$ may not arise as a reflective localization. Yet, 
the set $\mathrm{Geo}_{\kappa}(\mathcal{C})$ of higher covering diagrams with $\kappa$-small domain is a small 
well-structured colimit pre-topology for any given infinite regular cardinal $\kappa$ (infinity and 
regularity assure pre-diagonal closure). Via Theorem~\ref{thmcolimtop}, we obtain left exact accessible localizations
\[\hat{\mathcal{C}}\rightarrow\mathrm{Sh}_{\mathrm{Geo}_{\kappa}}(\mathcal{C}).\]
For every pair $\kappa_1\leq \kappa_2$ of infinite regular cardinals, there are canonical inclusions
$\mathrm{Sh}_{\mathrm{Geo}_{\kappa_2}}(\mathcal{C})\subseteq\mathrm{Sh}_{\mathrm{Geo}_{\kappa_1}}(\mathcal{C})$. Thus, for every cofinal sequence of infinite regular cardinals $\{\kappa_i\mid i\in\mathrm{Ord}\}$ we have
\[\mathrm{Sh}_{\mathrm{Geo}}(\mathcal{C})=\bigcap_{i\in\mathrm{Ord}}\mathrm{Sh}_{\mathrm{Geo}^{\kappa_i}}(\mathcal{C}).\]

\begin{remark}\label{remjointcanonical}
Whenever $\mathcal{C}$ is a small $\infty$-category with pullbacks and $\hat{\mathcal{C}}\rightarrow\mathcal{E}$ is a 
sub-canonical left exact accessible localization, we may thus reformulate canonicity of
$\mathrm{Geo}(\mathcal{C})$ as stated in Proposition~\ref{corhcdiscanonical} as follows. 
Whenever $\mathcal{C}$ is a small $\infty$-category with pullbacks and $\hat{\mathcal{C}}\rightarrow\mathcal{E}$ is a 
sub-canonical left exact accessible localization, then the small well-structured colimit pre-topology 
$T_{\mathcal{E}}$ on $\mathcal{C}$ is contained in the class $\mathrm{Geo}_{\kappa}(\mathcal{C})$ of higher covering 
diagrams with $\kappa$-small domain for some large enough cardinal $\kappa$. It follows that
$\mathrm{Sh}_{\mathrm{Geo}_{\kappa}}(\mathcal{C})\subseteq\mathcal{E}$, and so for any cofinal sequence of cardinals 
$\kappa$ the localizations $\hat{\mathcal{C}}\rightarrow\mathrm{Sh}_{\mathrm{Geo}_{\kappa}}(\mathcal{C})$ are 
``jointly'' canonical.
\end{remark}

\begin{proposition}\label{propcohcotoptopfac}
Let $\kappa$ be an uncountable regular cardinal and let $\mathcal{C}$ be a small locally $\kappa$-coherent
$\infty$-category. Then there is a sequence
\[\hat{\mathcal{C}}\rightarrow\mathrm{Sh}_{\mathrm{Coh}_{\kappa}}(\mathcal{C})\rightarrow\mathrm{Sh}_{\mathrm{Geo}_{\kappa}}(\mathcal{C})\]
of left exact accessible localizations where the first localization is topological and the second localization is 
cotopological.
\end{proposition}

\begin{proof}
We show that $(\mathrm{Cov}_{\mathrm{Geo}_{\kappa}})_{-1}\subseteq \mathrm{Cov}_{\mathrm{Coh}_{\kappa}}$ and
$\text{Cov}_{\mathrm{Coh}_{\kappa}}\subseteq \text{Cov}_{\mathrm{Geo}_{\kappa}}$. It then follows from
Lemma~\ref{propmodbase} that the Grothendieck topology generated by the set of $(-1)$-truncations
$(\text{Cov}_{\mathrm{Geo}_{\kappa}})_{-1}$ is exactly the $\kappa$-coherent Grothendieck topology on
$\mathcal{C}$, and so the statement follows from Corollary~\ref{cormodbase}.

To construct the inclusion
$(\mathrm{Cov}_{\mathrm{Geo}_{\kappa}})_{-1}\subseteq \mathrm{Cov}_{\mathrm{Coh}_{\kappa}}$, let
$U\colon I\rightarrow\mathcal{C}_{/B}$ be a $\kappa$-small higher covering diagram. Then the $(-1)$-truncation of
$\mathrm{colim}yU\rightarrow yB$ in $\hat{\mathcal{C}}$ is the sieve 
generated by the family $\{U_i\rightarrow B\mid i\in I\}$ by \cite[Lemma 6.2.3.13]{luriehtt} 
and \cite[Lemma 6.2.3.18]{luriehtt}. To show that this sieve is a $\kappa$-coherent covering sieve, we are 
to show that the family $\{U_i\rightarrow B\mid i\in I\}$ is jointly effective epic. This however follows from 
Proposition~\ref{lemmaocdran}.
The other direction also follows directly from Proposition~\ref{lemmaocdran}.
\end{proof}

\begin{remark}
In Proposition~\ref{propcohcotoptopfac} we assumed the cardinal $\kappa$ to be uncountable so that the cardinal 
occurring in both $\mathrm{Sh}_{\mathrm{Geo}_{\kappa}}(\mathcal{C})$ and
$\mathrm{Sh}_{\mathrm{Coh}_{\kappa}}(\mathcal{C})$ is the same. In the case $\kappa=\aleph_0$, the proof of 
Proposition~\ref{propcohcotoptopfac} only generates factorizations of the form
\[\xymatrix{
\hat{\mathcal{C}}\ar[r]\ar[dr] & \mathrm{Sh}_{(\mathrm{Geo}_{\aleph_0})_{-1}}(\mathcal{C})\ar[r]\ar[d] & \mathrm{Sh}_{\mathrm{Geo}_{\aleph_0}}(\mathcal{C})\ar[d]\\
 & \mathrm{Sh}_{\mathrm{Coh}_{\aleph_0}}(\mathcal{C})\ar[r] & \mathrm{Sh}_{\mathrm{Geo}_{\aleph_1}}(\mathcal{C}).
}\]
The increase in cardinality is caused by the fact that the higher covering diagram
$\check{\Sigma}(F)\colon\mathrm{FS}_+(I)^{op}\rightarrow\mathcal{C}_{/B}$ associated to a finite cover
$F=\{E_i\rightarrow B\mid i\in I\}$ in $\mathcal{C}$ has countably infinite domain $\mathrm{FS}_+(I)^{op}$. Although 
it still has finite ``width'', it invariably has countably infinite ``length''. In this sense, the finite case is 
somewhat singular.
\end{remark}

\begin{corollary}
Let $\kappa$ be an uncountable regular cardinal, and let $\mathcal{C}$ be a small $\kappa$-coherent
$\infty$-category. Then the $\infty$-toposes $\mathrm{Sh}_{\mathrm{Coh}_{\kappa}}(\mathcal{C})$ and
$\mathrm{Sh}_{\mathrm{Geo}_{\kappa}}(\mathcal{C})$ have the same class of points. By construction, these are the left 
exact functors $M\colon\mathcal{C}\rightarrow\mathcal{S}$ which preserve colimits of $\kappa$-small higher covering 
diagrams.
\end{corollary}
\begin{proof}
This follows immediately from the fact that the localization
$\mathrm{Sh}_{\mathrm{Coh}_{\kappa}}(\mathcal{C})\rightarrow\mathrm{Sh}_{\mathrm{Geo}_{\kappa}}(\mathcal{C})$ is 
cotopological, together with the general observations that (the left adjoint part of) points preserve
$\infty$-connected maps, and that $\mathcal{S}$ is hypercomplete.
\end{proof}



\begin{corollary}\label{corcovhyper}
Let $\mathcal{C}$ be a small locally $\kappa$-coherent $\infty$-category for some uncountable regular 
cardinal $\kappa$. Every non-trivial $\infty$-connected map $f$ in $\mathcal{C}$ induces a non-trivial $\infty$-
connected map $yf$ in $\mathrm{Sh}_{\mathrm{Geo}_{\kappa}}(\mathcal{C})$. In particular, the $\infty$-topos
of higher $\kappa$-geometric sheaves on $\mathcal{C}$ is generally not hypercomplete, and hence does generally 
not have enough points.
\end{corollary}
\begin{proof}
As the localization $\hat{\mathcal{C}}\rightarrow\mathrm{Sh}_{\mathrm{Geo}_{\kappa}}(\mathcal{C})$ is sub-canonical, 
a map $f\colon E\rightarrow B$ in $\mathcal{C}$ is an equivalence if and only if $yf$ is an equivalence in 
$\mathrm{Sh}_{\mathrm{Geo}_{\kappa}}(\mathcal{C})$. Furthermore, given an $\infty$-connected map
$f\colon E\rightarrow B$ in $\mathcal{C}$, the representable $yf\in\hat{\mathcal{C}}$ is again $\infty$-connected in 
$\mathrm{Sh}_{\mathrm{Coh}_{\kappa}}(\mathcal{C})$ by Proposition~\ref{propregpts}. As the localization
$\mathrm{Sh}_{\mathrm{Coh}_{\kappa}}(\mathcal{C})\rightarrow\mathrm{Sh}_{\mathrm{Geo}_{\kappa}}(\mathcal{C})$ 
preserves $\infty$-connected maps, the map $yf\in\hat{\mathcal{C}}$ is still $\infty$-connected in
$\mathrm{Sh}_{\mathrm{Geo}_{\kappa}}(\mathcal{C})$.
\end{proof}

\begin{corollary}
For any uncountable regular cardinal $\kappa$, any locally $\kappa$-coherent $\infty$-category
$\mathcal{C}$, and any finite integer $n\geq -2$, the inclusion
$\mathrm{Sh}_{\mathrm{Geo}_{\kappa}}(\mathcal{C})\subseteq\mathrm{Sh}_{\mathrm{Coh}_{\kappa}}(\mathcal{C})$ induces 
an equivalence
$\tau_n(\mathrm{Sh}_{\mathrm{Geo}_{\kappa}}(\mathcal{C}))\simeq\tau_n(\mathrm{Sh}_{\mathrm{Coh}_{\kappa}}(\mathcal{C}))$ between the Grothendieck $n$-toposes of $n$-truncated sheaves. Thus, 
an $n$-truncated presheaf $F\colon\mathcal{C}^{op}\rightarrow\tau_n(\mathcal{S})$ is $\kappa$-coherent if 
and only if it is higher $\kappa$-geometric.
\end{corollary}
\begin{proof}
Any cotopological localization of $\infty$-toposes induces an equivalence on the according $n$-toposes of
$n$-truncated objects. 
\end{proof}

In the following we will see that $\kappa$-coherent and higher $\kappa$-geometric sheaves of arbitrary homotopy type 
however generally differ. 
In fact, recall that every 1-topos is equivalent to the category of (set-valued) sheaves for the geometric site over 
itself \cite[Proposition C.2.2.7]{elephant}. That means, the geometric (set-valued) sheaves on a 1-topos
$\mathcal{C}$ are exactly the small limit preserving functors $\mathcal{C}^{op}\rightarrow\mathrm{Set}$.
Whenever $\mathcal{C}$ is an $\infty$-topos, the latter notion is captured by \cite[Notation 6.3.5.16]{luriehtt} 
which defines the $\infty$-category $\mathrm{Sh}_{\mathcal{D}}(\mathcal{C})$ of small limit-preserving functors
$\mathcal{C}^{op}\rightarrow\mathcal{D}$ for an $\infty$-category $\mathcal{D}$. Lurie refers to such functors as
\emph{$\mathcal{D}$-valued sheaves on the $\infty$-topos $\mathcal{C}$}. We recover this sheaf condition over
$\infty$-toposes as follows.

\begin{proposition}\label{prophcdtoplim}
Let $\mathcal{E}$ be an $\infty$-topos and $\mathcal{D}$ be an $\infty$-category which admits all small limits. Then 
a functor $\mathcal{E}^{op}\rightarrow\mathcal{D}$ preserves all small limits if and only if it takes colimits of 
small higher covering diagrams in $\mathcal{E}$ to limits in $\mathcal{D}$.
\end{proposition}

\begin{proof}
One direction is trivial. We show the other direction in two steps. First, let $\mathcal{C}$ be a small $\infty$-
category with pullbacks, and suppose $\mathcal{E}\simeq\hat{\mathcal{C}}$. For every $X\in\hat{\mathcal{C}}$, the 
canonical inclusion
\[\mathcal{C}_{/X}\xrightarrow{y}\hat{\mathcal{C}}_{/X}\]
is well-indexed and colimiting. It is a higher covering diagram by Example~\ref{explecovgen} (or alternatively by
Corollary~\ref{corhcdlogos} as $\hat{\mathcal{C}}$ has descent). Thus, whenever $F\colon\hat{\mathcal{C}}^{op}\rightarrow\mathcal{D}$ 
takes colimits of small higher covering diagrams in $\hat{\mathcal{C}}$ to 
limits in $\mathcal{D}$, it follows that $F$ is the pointwise right Kan extension of its restriction along
$y\colon\mathcal{C}^{op}\rightarrow\hat{\mathcal{C}}^{op}$. By \cite[Lemma 5.1.5.5]{luriehtt} it follows that
$F$ preserves all small limits.

Second, suppose $\mathcal{E}$ is a general $\infty$-topos. By \cite[Proposition 6.1.5.3]{luriehtt} there is a small $\infty$-category  
$\mathcal{C}$ with pullbacks together with a left exact accessible localization functor
$L\colon\hat{\mathcal{C}}\rightarrow\mathcal{E}$.
Suppose $F\colon\mathcal{E}^{op}\rightarrow\mathcal{D}$ takes colimits of small higher covering diagrams in $\mathcal{E}$ to limits in 
$\mathcal{D}$. Since $L\colon\hat{\mathcal{C}}\rightarrow\mathcal{E}$ preserves both pullbacks and colimits, and 
higher covering diagrams are just well-indexed colimiting diagrams by virtue of descent, every higher covering 
diagram $G\colon I\rightarrow\hat{\mathcal{C}}_{/X}$ yields a higher covering diagram
$LG\colon I\rightarrow\mathcal{E}_{/LX}$ by push-forward along $L$. Thus, the composition
\[FL\colon\hat{\mathcal{C}}^{op}\rightarrow\mathcal{D}\]
takes colimits of higher covering diagrams in $\hat{\mathcal{C}}$ to limits in $\mathcal{D}$. By the first part of 
the proof it follows that $FL\colon\hat{\mathcal{C}}^{op}\rightarrow\mathcal{D}$ preserves all small limits. By
\cite[Proposition 5.5.4.20]{luriehtt} and fully faithfulness of the right adjoint
$\mathcal{E}\hookrightarrow\hat{\mathcal{C}}$, it follows that $F\colon\mathcal{E}^{op}\rightarrow\mathcal{D}$ is 
small limit preserving itself.
\end{proof}

\begin{theorem}\label{remcovtoposesexple}
Every $\infty$-topos is the $\infty$-category of higher geometric sheaves over itself. More precisely, whenever
$\mathcal{E}$ is an $\infty$-topos, we have the following.
\begin{enumerate}
\item A presheaf $\mathcal{E}^{op}\rightarrow\mathcal{S}$ (of small spaces) is higher geometric if and only if it is representable. In 
particular, the Yoneda embedding
\[y\colon\mathcal{E}\rightarrow\mathrm{Sh}_{\mathrm{Geo}}(\mathcal{E})=\mathrm{Sh}_{\mathcal{S}}(\mathcal{E})\]
is essentially surjective and hence an equivalence.
\item Suppose $\mathcal{E}$ is contained in some Grothendieck universe $\mathcal{U}$, and let $\mathcal{S}^+$ be the
$\infty$-category of large spaces. Then a presheaf $\mathcal{E}^{op}\rightarrow\mathcal{S}^+$ (of large spaces) is 
higher geometric if and only if it preserves all $\mathcal{U}$-small limits. I.e.\,
$\mathrm{Sh}_{\mathcal{S}^+}(\mathcal{E})$ is the $\infty$-category of large higher geometric sheaves on
$\mathcal{E}$.
\end{enumerate}
\end{theorem}

\begin{proof}
Both statements follow from Proposition~\ref{prophcdtoplim} for $\mathcal{D}=\mathcal{S}$ in the first case and
$\mathcal{D}=\mathcal{S}^+$ in the second.
\end{proof}

%
%

In the next proposition we construct a class of examples of (small) $\infty$-categories whose $\kappa$-coherent
and higher $\kappa$-geometric sheaf theories provably differ. In particular, it will show that 
Theorem~\ref{remcovtoposesexple} does not hold for the infinitary-coherent sheaf theory of an $\infty$-topos. In that 
sense, it follows that the infinitary-coherent Grothendieck topology on an $\infty$-topos $\mathcal{E}$ is 
insufficient to recover $\mathcal{E}$ as a sheaf theory over itself. Therefore, we first state and prove one more 
general lemma.

\begin{lemma}\label{lemmaprecorcovnotequcoh}
Let $\mathcal{E}$ be an $\infty$-topos. Then the hypercompletion endofunctor $\tau_{\infty}\colon\mathcal{E}\rightarrow\mathcal{E}$ 
associated to the left exact localization $\mathcal{E}\rightarrow\tau_{\infty}(\mathcal{E})$
\cite[Section 6.5.2]{luriehtt} preserves effective epimorphisms and coproducts.
\end{lemma}
\begin{proof}
First, to see that hypercompletion in an $\infty$-topos $\mathcal{E}$ always preserves effective 
epimorphisms, let $f\colon E\rightarrow B$ be an effective epimorphism in $\mathcal{E}$. We obtain the following map 
of hypercompletions in $\mathcal{E}$. 
\[\xymatrix{
E\ar[r]^(.4){\eta_E}\ar[d]_f & \tau_{\infty}(E)\ar[d]^{\tau_{\infty}(f)} \\
B\ar[r]_(.4){\eta_B} & \tau_{\infty}(B)
}\]
The map $f$ is an effective epimorphism by assumption, the two vertical maps are $\infty$-connected 
and as such in particular effective epimorphisms as well. It follows that $\tau_{\infty}(f)$ is an effective epimorphism 
by compositionality and right cancellability of effective epimorphisms \cite[Corollary 6.2.3.12]{luriehtt}.

To see that $\tau_{\infty}$ preserves coproducts, it suffices to show that the class of hypercomplete objects in
$\mathcal{E}$ is closed under coproducts. Therefore suppose that $I$ is a set and that we are given a collection
$\{X_i\mid i\in I\}$ of hypercomplete objects in $\mathcal{E}$. Let $f\colon A\rightarrow B$ be $\infty$-connected 
and $g\colon A\rightarrow X$ for $X\simeq\coprod_{i\in I}X_i$ be a map. We are to show that $g$ lifts along $f$ in 
essentially unique fashion. Therefore, note that since $\mathcal{E}$ is extensive, for $A_i\simeq A\times_X X_i$ we 
obtain a collection of maps $g_i\colon A_i\rightarrow X_i$ for $i\in I$ together with an equivalence
\[\xymatrix{
A\ar@{}[d]|{\rotatebox[origin=c]{-90}{$\simeq$}}\ar@/^/[dr]^{g} & \\
\coprod_{i\in I} A_I\ar[r]_{\coprod_{i\in I}{g_i}} & \coprod_{i\in I}X_i.
}\]
Furthermore, we obtain maps $f_i\colon A_i\rightarrow B$ such that $f\simeq (f_i)_{i\in I}$. Since $f$ is
$\infty$-connected, its $0$-truncation
\[\xymatrix{
\coprod_{i\in I}A_i\ar[r]^{\eta^0_A}\ar[d]_{(f_i)_{i\in I}} & \tau_0(\coprod_{i\in I}A_i)\ar[d]_{\rotatebox[origin=c]{90}{$\simeq$}}^{\tau_0((f_i)_{i\in I})} \\
B\ar[r]_{\eta^0_B} & \tau_0(B)
}
\]
is an equivalence. Now, $n$-truncation $\tau_n$ for $n\geq 0$ preserves coproducts, because the localization
$\mathcal{E}\rightarrow\tau_n\mathcal{E}$ is generated by the tensors
$E\otimes\partial\Delta^{n+1}\rightarrow E\otimes\Delta^0$ for $E\in\mathcal{E}$, and the $(n+1)$-sphere for
$n\geq 0$ is connected \cite[Proposition 5.5.6.18]{luriehtt}. Again using that $\mathcal{E}$ is extensive, for
$B_i\simeq B\times_{\tau_0(B)}\tau_0(A_i)$ the map $f\colon A\rightarrow B$ is equivalent to the coproduct
$\coprod_{i\in I}f_i\colon\coprod_{i\in I}A_i\rightarrow \coprod_{i\in I}B_i$. Since $f$ is $\infty$-connected and 
the class of $\infty$-connected maps is closed under pullback, each $f_i$ is $\infty$-connected as well. We thus are 
given a lifting problem of the form
\[\xymatrix{
\coprod_{i\in I}A_i\ar[d]_{\coprod_{i\in I}f_i}\ar[r]^{\coprod_{i\in I}g_i} & \coprod_{i\in I}X_i. \\
\coprod_{i\in I}B_i
}\]
As each $X_i$ is hypercomplete and each $f_i$ is $\infty$-connected, this admits a solution. This solution is 
essentially unique whenever 
every map of type $B_i\rightarrow X$ extending $g_i\colon A_i\rightarrow X_i$ factors through the component
$X_i\hookrightarrow X$. This indeed is satisfied, since the inclusion $X_i\hookrightarrow\coprod_{i\in I}X_i$ is
$(-1)$-truncated and $f_i\colon A_i\rightarrow B_i$ is $(-1)$-connected, and so the square
\[\xymatrix{
A_i\ar[d]_{f_i}\ar[r]^{g_i} & X_i\ar@{^(->}[d] \\
B_i\ar[r] & \coprod_{i\in I} X_i
}\]
exhibits a lift.
\end{proof}

\begin{remark}
Although not needed here, the proof of Lemma~\ref{lemmaprecorcovnotequcoh} applies not only to the 
hypercompletion endofunctor $\tau_{\infty}\colon\mathcal{E}\rightarrow\mathcal{E}$, but as well to the finite truncation 
functors $\tau_{\leq n}\colon\mathcal{E}\rightarrow\mathcal{E}$ for every $n\geq -1$ in the case of effective 
epimorphisms, and for every $n\geq 0$ in the case of coproducts. For such natural numbers $n<\infty$, we only need that 
$\mathcal{E}$ is presentable and regular for the first case, and furthermore extensive for the second case.
\end{remark}

\begin{proposition}\label{propcovnotequcoh}
There are $\infty$-toposes $\mathcal{E}$ such that the canonical inclusion
\[\mathrm{Sh}_{\mathrm{Geo}}(\mathcal{C})\hookrightarrow \mathrm{Sh}_{\mathrm{Coh}}(\mathcal{C})\]
is non-trivial. Accordingly, for all uncountable regular cardinals $\kappa$ large enough there is a small
$\infty$-category $\mathcal{C}$ (with finite limits, $\kappa$-small colimits and descent) such that 
the cotopological localization
\[\mathrm{Sh}_{\mathrm{Coh}_{\kappa}}(\mathcal{C})\rightarrow\mathrm{Sh}_{\mathrm{Geo}_{\kappa}}(\mathcal{C})\]
from Proposition~\ref{propcohcotoptopfac} is non-trivial. 
\end{proposition}
\begin{proof}
Let $\mathcal{E}$ be an $\infty$-topos with the following two properties.
\begin{enumerate}
\item $\mathcal{E}$ is generated by a set $G$ of objects which is closed under fiber products and such that each 
$g\in G$ is hypercomplete.
\item $\mathcal{E}$ is not hypercomplete itself.
\end{enumerate}
Let $\kappa\geq |G|$ be any regular cardinal such that there is a $\kappa$-compact non-hypercomplete object 
$E\in\mathcal{E}$, and such that the $\infty$-category $G_{/E}$ is $\kappa$-small. Let $\mu$ be a regular 
cardinal sharply larger than $\kappa$ \cite[Definition 5.4.2.8]{luriehtt} such that the accessible endofunctor 
$T_{\infty}\colon\mathcal{E}\rightarrow\mathcal{E}$ takes $\mu$-small objects to $\mu$-small objects
\cite[Lemma 8.3.4]{thesis}. Let $\mathcal{C}\subset\mathcal{E}$ be the full sub-$\infty$-category of $\mu$-compact 
objects. Then $\mathcal{C}$ is small, $\kappa$-cocomplete, left exact, and has descent for $\kappa$-small diagrams. 

Now, for every $\kappa$-coherent sheaf $X$ on $\mathcal{C}$, the precomposition
$X\tau_{\infty}^{op}\colon\mathcal{C}^{op}\rightarrow\mathcal{S}$ with the endofunctor
$\tau_{\infty}\colon\mathcal{C}\rightarrow\mathcal{C}$ is a $\kappa$-coherent sheaf again by
Lemma~\ref{lemmaprecorcovnotequcoh} and by the fact that $\tau_{\infty}\colon\mathcal{C}\rightarrow\mathcal{C}$ 
preserves finite limits. In the following we show that the composition
$yE\tau_{\infty}^{op}\colon\mathcal{C}^{op}\rightarrow\mathcal{S}$ is not higher $\kappa$-geometric. Since the 
representable $yE$ however is $\kappa$-coherent, this proves the statement (assuming that the
$\infty$-topos $\mathcal{E}$ exists).

Therefore, we use that the inclusion $G_{/E}\rightarrow\mathcal{C}_{/E}$ is colimiting, and consider the 
induced map
\begin{align}\label{equcorcovnotequcoh}
\underset{g\in G_{/E}}{\mathrm{colim}}(\tau_{\infty}(g))\rightarrow \tau_{\infty}(\underset{g\in G_{/E}}{\mathrm{colim}}g).
\end{align}
As $\tau_{\infty}(g)\simeq g$ for all $g\in G$ by Property 1, the domain of (\ref{equcorcovnotequcoh}) is equivalent 
to $E$ itself, while its codomain is the hypercompletion of $E$ in $\mathcal{E}$ by construction.
If the representable $yE$ applied to the map (\ref{equcorcovnotequcoh}) was an equivalence of spaces, we 
would obtain a retract to the map (\ref{equcorcovnotequcoh}) in $\mathcal{C}$. Since the collection of hypercomplete 
objects is closed under retracts, that would imply that $E$ is hypercomplete as well, which is contrary to our 
assumption.
But $yE$ preserves colimits itself, and so it follows that the presheaf
$yE\tau_{\infty}\colon\mathcal{C}^{op}\rightarrow\mathcal{S}$ does not preserve the colimit of the inclusion 
$G_{/E}\rightarrow\mathcal{C}_{/E}$. But the $\infty$-category $G_{/E}$ has pullbacks by virtue of 
Property 1. These pullbacks are furthermore preserved by $G_{/E}\rightarrow\mathcal{C}_{/E}$. It follows that
$G_{/E}\rightarrow\mathcal{C}_{/E}$ is a higher covering diagram by Example~\ref{explecovgen} (or again via 
Corollary~\ref{corhcdlogos} as $\mathcal{C}$ has descent for $\kappa$-small diagrams and $G_{/E}$ is $\kappa$-small). 
Consequently, the presheaf $yE\tau_{\infty}$ is not a higher $\kappa$-geometric sheaf. It yet is
$\kappa$-coherent by the observations put forward at the beginning of the proof.

In order to finish the proof, we are left to present an $\infty$-topos $\mathcal{E}$ which has the 
properties listed in 1 and 2.
Therefore, we simply note that the $\infty$-topos of sheaves $\mathrm{Sh}_{J}(\mathcal{C})$ on any small
sub-canonical 1-site $(\mathcal{C},J)$ where $\mathcal{C}$ has pullbacks satisfies Property 1. Indeed, due 
to sub-canonicity $\mathrm{Sh}_{J}(\mathcal{C})$ is generated by the representables
$\mathcal{C}\xrightarrow{y}\mathrm{Sh}_{J}(\mathcal{C})$. 
As $\mathcal{C}$ is of finite homotopy type, each representable is of finite homotopy type and thus is in particular 
hypercomplete. An example of such an $\infty$-topos which is not hypercomplete itself is the localic
Dugger-Hollander-Isaksen-topos we used for other examples as well \cite[Section 11.3]{rezkhtytps}.
\end{proof}


\section{Higher geometric $\infty$-categories}\label{secsites}


In this section we propose a definition of the $\infty$-category of higher ($\kappa$-)geometric $\infty$-categories 
(equipped with their canonical higher sites) and relate it to the $\infty$-category of $\infty$-toposes 
\cite{luriehtt}.

\begin{definition}\label{defhgeocat}
An $\infty$-category $\mathcal{C}$ with pullbacks is \emph{locally higher $\kappa$-geometric} for some regular 
cardinal $\kappa$ if every well-indexed diagram $U\colon I\rightarrow\mathcal{C}_{/C}$ with $\kappa$-small index $I$ 
admits a factorization
\[I\xrightarrow{U}\mathcal{C}_{/B}\xrightarrow{\Sigma_f}\mathcal{C}_{/C}\]
through a higher covering diagram. An $\infty$-category $\mathcal{C}$ with pullbacks is \emph{locally higher 
geometric} if it is higher $\kappa$-geometric for all regular cardinals $\kappa$. A locally higher
($\kappa$-)geometric $\infty$-category $\mathcal{C}$ is \emph{higher ($\kappa$-)geometric} if it has a terminal 
object.
\end{definition}

\begin{lemma}\label{lemmageoisextord}
Every locally $\kappa$-geometric (geometric) $\infty$-category $\mathcal{C}$ is in particular $\kappa$-extensive (for 
all $\kappa$) as well as locally $\kappa$-coherent (for all uncountable $\kappa$).
\end{lemma}
\begin{proof}
Follows directly from Corollary~\ref{corextdescalt} and Theorem~\ref{corcharordgeocat}.
\end{proof}

\begin{lemma}\label{lemmahgeocatalt}
For any $\infty$-category $\mathcal{C}$, the following are equivalent. 
\begin{enumerate}
\item $\mathcal{C}$ is locally higher ($\kappa$-)geometric.
\item All slices $\mathcal{C}_{/B}$ of $\mathcal{C}$ are higher ($\kappa$-)geometric.
\item $\mathcal{C}$ has pullbacks, and every ($\kappa$-)small diagram $U\colon I\rightarrow\mathcal{C}_{/C}$ admits a 
factorization
\[I\xrightarrow{U}\mathcal{C}_{/B}\xrightarrow{\Sigma_f}\mathcal{C}_{/C}\]
through a descent diagram.
\end{enumerate}
In particular, a higher geometric $\infty$-category is precisely an $\infty$-category with finite limits, small 
colimits and descent.
\end{lemma}
\begin{proof}
The equivalence of Parts 1 and 2 is straight-forward. The equivalence of Parts 2 and 3 follows from 
Lemma~\ref{lemmadeschcd} and Lemma~\ref{lemmahcddesc} (as summarized in Remark~\ref{remdeschcdequiv}).
\end{proof}



Following the 1-categorical tradition captured by \cite[Proposition 1.4.8]{caramellobook}, we define locally
higher ($\kappa$-)geometric functors to be the pullback-preserving functors which preserve ($\kappa$-)geometric 
covers. 

%

\begin{definition}
Let $\kappa$ be a regular cardinal. A functor $F\colon\mathcal{C}\rightarrow\mathcal{D}$ between locally
higher $\kappa$-geometric $\infty$-categories is \emph{locally higher $\kappa$-geometric} if it preserves pullbacks 
as well as $\kappa$-small higher covering diagrams. That is to say, whenever $U\colon I\rightarrow\mathcal{C}_{/B}$ 
is a $\kappa$-small higher covering diagram, then so is $FU\colon I\rightarrow\mathcal{D}_{/FB}$. A locally higher
$\kappa$-geometric functor $F\colon\mathcal{C}\rightarrow\mathcal{D}$ between higher $\kappa$-geometric
$\infty$-categories is \emph{higher $\kappa$-geometric} if $F$ preserves the terminal object. A functor between 
(locally) higher geometric $\infty$-categories is (locally) higher geometric if it is (locally) higher
$\kappa$-geometric for all $\kappa$.
\end{definition}

\begin{definition}
The $\infty$-category
$\mathrm{GeoCat}_{\kappa}\subset\mathrm{Cat}$
is the sub-$\infty$-category of small higher $\kappa$-geometric $\infty$-categories, higher $\kappa$-geometric 
functors and all higher cells. Accordingly, the (superlarge) $\infty$-category
$\mathrm{GeoCAT}_{\kappa}\subset\mathrm{CAT}$
is the sub-$\infty$-category of all higher $\kappa$-geometric $\infty$-categories, higher $\kappa$-geometric functors 
and all higher cells. The (superlarge) $\infty$-category $\mathrm{GeoCAT}$
denotes the $\infty$-category of large higher geometric $\infty$-categories and higher geometric functors.
\end{definition}

\begin{lemma}\label{lemmageofunctors}
For any pullback-preserving functor $F\colon\mathcal{C}\rightarrow\mathcal{D}$ between locally higher
($\kappa$-)geometric $\infty$-categories the following are equivalent.
\begin{enumerate}
\item $F$ is locally higher ($\kappa$-)geometric.
\item $F$ preserves colimits of ($\kappa$-)small higher covering diagrams.
\end{enumerate}
Furthermore, either condition is equivalent to the following (whenever $\kappa$ is regular and large enough so
$\mathcal{C}$ itself is $\kappa$-small).
\begin{enumerate}
\item[3.] $F$ preserves ($\kappa$-)small colimits.
\end{enumerate}
\end{lemma}
\begin{proof}
The equivalence of Parts 1 and 2 is immediate by the fact that $F$ preserves pullbacks, that higher covering diagrams 
of type $I\rightarrow\mathcal{C}_{/B}$ are by definition colimiting over $B$, and that colimits are unique up to 
equivalence whenever they exist.

To show the equivalence of Parts 2 and 3, suppose $F\colon\mathcal{C}\rightarrow\mathcal{D}$ is a pullback-preserving 
functor between higher ($\kappa$-)geometric $\infty$-categories. By Lemma~\ref{lemmahgeocatalt} 
we are to show that $F$ preserves colimits of all small diagrams in $\mathcal{C}$ if and only if it preserves 
colimits of all small well-indexed diagrams in $\mathcal{C}$.
Therefore, suppose $F$ preserves all ($\kappa$-)small well-indexed colimits and let $U\colon I\rightarrow\mathcal{C}$ 
be any ($\kappa$-)small diagram. We may factor $U$ into a right anodyne inclusion $I\hookrightarrow RI$ followed by 
right fibration $RU\colon RI\twoheadrightarrow\mathcal{C}$. This factors $U$ into a cofinal functor followed by a 
well-indexed small diagram by Lemma~\ref{lemmarfibwellind}. The latter is again $\kappa$-small whenever both $I$ 
and $\mathcal{C}$ are $\kappa$-small, as in this case each fiber
$RU^{-1}(C)\simeq\mathrm{colim}_{i\in I}\mathcal{C}(C,U_i)$ is $\kappa$-small. Hence, $F$ preserves the 
colimit of $RU$ by assumption, and hence preserves the colimit of $U$ by virtue of cofinality of the inclusion 
$I\hookrightarrow RI$.

\end{proof}

Let $\mathrm{LTop}$ denote the (superlarge) $\infty$-category of $\infty$-toposes and left exact left adjoints 
\cite[Definition 6.3.1.5]{luriehtt}. 

\begin{proposition}\label{propgeofunctors}
Every $\infty$-topos is a higher geometric $\infty$-category. A functor
$F\colon\mathcal{C}\rightarrow\mathcal{D}$ between $\infty$-toposes is higher geometric if and only if it is
left exact and cocontinuous. In particular, there is a fully faithful forgetful functor
$U\colon\mathrm{LTop}\rightarrow\mathrm{GeoCAT}$.
\end{proposition}
\begin{proof}
The fact that $\infty$-toposes are higher geometric follows directly from Corollary~\ref{corhcdlogos} and the fact 
that $\infty$-toposes are cocomplete. The second statement follows directly from Lemma~\ref{lemmageofunctors}.
\end{proof}

\begin{theorem}\label{propgeocat1}
Let $\mathcal{C}$ be a small higher $\kappa$-geometric $\infty$-category for some cardinal $\kappa$. Then the 
sheafified Yoneda embedding
\begin{align}\label{equgeocatyoneda}
y\colon\mathcal{C}\rightarrow\mathrm{Sh}_{\mathrm{Geo}_{\kappa}}(\mathcal{C})
\end{align}
is higher $\kappa$-geometric. For all $\infty$-toposes $\mathcal{D}$, the induced restriction
\[y^{\ast}\colon\mathrm{LTop}(\mathrm{Sh}_{\mathrm{Geo}_{\kappa}}(\mathcal{C}),\mathcal{D})\rightarrow\mathrm{GeoCAT}_{\kappa}(\mathcal{C},\mathcal{D})\]
along $y$ is an equivalence of hom-spaces.
\end{theorem}
\begin{proof}
The embedding (\ref{equgeocatyoneda}) is left exact and preserves colimits of $\kappa$-small higher covering diagrams 
by construction. Hence, the functor (\ref{equgeocatyoneda}) is higher $\kappa$-geometric by 
Lemma~\ref{lemmageofunctors}. Let
$L\colon\hat{\mathcal{C}}\rightarrow\mathrm{Sh}_{\mathrm{Geo}_{\kappa}}(\mathcal{C})$ denote the left adjoint to 
the canonical inclusion in converse direction. Given an $\infty$-topos $\mathcal{D}$, consider the diagram
\[\xymatrix{
\mathrm{LTop}(\hat{\mathcal{C}},\mathcal{D})\ar@/^.5pc/@<.5ex>[r]^{y^{\ast}} & \mathrm{Fun}^{\mathrm{lex}}(\mathcal{C},\mathcal{D})^{\simeq}\ar@<.5ex>@/^.5pc/[l]^{y_!} \\
\mathrm{LTop}(\mathrm{Sh}_{\mathrm{Geo}_{\kappa}}(\mathcal{C}),\mathcal{D})\ar@{^(->}[u]^{L^{\ast}}\ar@{-->}@/^.5pc/@<.5ex>[r]^{y^{\ast}} & \mathrm{GeoCAT}_{\kappa}(\mathcal{C},\mathcal{D})\ar@{^(->}[u]\ar@{-->}@<.5ex>@/^.5pc/[l]^{y_!}
}\]
associated to the embedding $y\colon\mathcal{C}\rightarrow\hat{\mathcal{C}}$ and to its corestriction
$y\colon\mathcal{C}\rightarrow\mathrm{Sh}_{\mathrm{Geo}_{\kappa}}(\mathcal{C})$. The two vertical functors are fully 
faithful. The top horizontal pair $(y_!,y^{\ast})$ is an equivalence by
\cite[Theorem 5.1.5.6, Proposition 6.1.5.2]{luriehtt}. The restriction of
the top horizontal functor $y^{\ast}$ along the inclusion $L^{\ast}$ is equivalent to the restriction
$y^{\ast}\colon\mathrm{LTop}(\mathrm{Sh}_{\mathrm{Geo}_{\kappa}}(\mathcal{C}),\mathcal{D})\rightarrow\mathrm{Fun}^{\mathrm{lex}}(\mathcal{C},\mathcal{D})$. It factors through
$\mathrm{GeoCat}_{\kappa}(\mathcal{C},\mathcal{D})$ because 
$y\colon\mathcal{C}\rightarrow\mathrm{Sh}_{\mathrm{Geo}_{\kappa}}(\mathcal{C})$ is higher $\kappa$-geometric,
higher geometric morphisms between $\infty$-toposes are higher $\kappa$-geometric, and higher $\kappa$-geometric 
functors are closed under composition. The restriction of the top horizontal left Kan extension $y_!$ along the 
inclusion 
$\mathrm{GeoCat}_{\kappa}(\mathcal{C},\mathcal{D})\hookrightarrow\mathrm{Fun}^{\mathrm{lex}}(\mathcal{C},\mathcal{D})$ factors through $\mathrm{LTop}(\mathrm{Sh}_{\mathrm{Geo}_{\kappa}}(\mathcal{C}),\mathcal{D})$ via 
\cite[Proposition 5.5.4.20]{luriehtt}. It readily follows that the thereby induced bottom horizontal pair
$(y_!,y^{\ast})$ is an equivalence as well.
\end{proof}

For a given cardinal $\kappa$ let
$U_{\kappa}\colon\mathrm{GeoCAT}\hookrightarrow\mathrm{GeoCAT}_{\kappa}$ denote the obvious forgetful 
functor, and let $\iota_{\kappa}\colon\mathrm{GeoCat}_{\kappa}\rightarrow\mathrm{GeoCAT}_{\kappa}$ 
denote the canonical inclusion. We end this section with the following corollary which shows that
$\mathrm{Sh}_{\mathrm{Geo}_{\kappa}}(\mathcal{C})$ is the free $\infty$-topos generated by a small higher
$\kappa$-geometric $\infty$-category $\mathcal{C}$.

\begin{corollary}\label{propfreegeotop}
The composite forgetful functor
\begin{align}\label{equpropfreegeotop}
\mathrm{LTop}\xrightarrow{U}\mathrm{GeoCAT}\xrightarrow{U_{\kappa}}\mathrm{GeoCAT}_{\kappa}
\end{align}
has a $\iota_{\kappa}$-relative left adjoint 
\[\mathrm{Sh}_{\mathrm{Geo}_{\kappa}}(-)\colon\mathrm{GeoCat}_{\kappa}\rightarrow\mathrm{LTop}\]
for every cardinal $\kappa$.
\end{corollary}
\begin{proof}
Given a small higher $\kappa$-geometric $\infty$-category $\mathcal{C}$, the embedding
$y\colon\mathcal{C}\rightarrow\mathrm{Sh}_{\mathrm{Geo}_{\kappa}}(\mathcal{C})$ is initial in
$\mathcal{C}_{/\mathrm{LTop}}$ by Theorem~\ref{propgeocat1}, and hence is a unit which exhibits the composition 
(\ref{equpropfreegeotop}) as a $\iota_{\kappa}$-relative right adjoint.
\end{proof}

\section{Appendix on cofinal equivalence}\label{secappcof}

\begin{definition}\label{defcofequiv}
Say two diagrams $U\colon I\rightarrow\mathcal{C}$ and $V\colon J\rightarrow\mathcal{C}$ are \emph{cofinally 
equivalent} if the two post-compositions $yU\colon I\rightarrow\hat{\mathcal{C}}$ and
$yV\colon J\rightarrow\hat{\mathcal{C}}$ have equivalent colimits.
\end{definition}

The following lemma is shown in \cite[Proposition 3.9]{perronetholen} in the context of ordinary category theory. Cofinally equivalent diagrams are loc.\ cit.\ referred to as ``mutually confinal'' diagrams.

\begin{lemma}\label{lemmacofequiv}
Let $\mathcal{C}$ be a small $\infty$-category and let $U\colon I\rightarrow\mathcal{C}$ and
$V\colon J\rightarrow\mathcal{C}$ be small diagrams. Then the following are equivalent.
\begin{enumerate}
\item $U$ and $V$ are cofinally equivalent.
\item For any functor $F\colon\mathcal{C}\rightarrow\mathcal{D}$ there is an equivalence
$\mathrm{colim}_I FU\simeq\mathrm{colim}_J FV$ in $\mathcal{D}$ whenever either of the colimits exists.
\item There is a right fibration $W\colon K\twoheadrightarrow\mathcal{C}$ together with cofinal functors
$\phi\colon I\rightarrow K$ and $\psi\colon J\rightarrow K$ such that the following two triangles commute.
\[\xymatrix{
I\ar@/^1pc/[drr]^U\ar[dr]^{\phi} & & \\
 & K\ar@{->>}[r]^W & \mathcal{C} \\
J\ar@/_1pc/[urr]_V\ar[ur]^{\psi} & &  
}\]
\end{enumerate}
\end{lemma}
\begin{proof}
First, if there is a right fibration $W\colon K\twoheadrightarrow\mathcal{C}$ as stated in Part 3, every horizontal 
functor in the canonical diagram
\[\xymatrix{
\mathcal{D}_{FU/}\ar@{->>}@/_1pc/[drr]\ar@{}[r]|{\simeq} & \mathcal{D}_{FW\phi/}\ar@{->>}@/_/[dr] & \mathcal{D}_{FW/}\ar@{->>}[d]\ar[l]\ar[r] & \mathcal{D}_{FW\psi/}\ar@{->>}@/^/[dl]\ar@{}[r]|{\simeq} & \mathcal{D}_{FV/}\ar@{->>}@/^1pc/[dll] \\
 & & \mathcal{D} & & 
}\]
of left fibrations over $\mathcal{C}$ is an equivalence by \cite[Proposition 4.1.1.8]{luriehtt}. In particular, 
whenever either of the diagrams $FU$ or $FV$ has a colimit, it yields an initial object in the according
over-category, and hence induces an initial object in the respectively other over-category. The fact that the diagram 
commutes over $\mathcal{D}$ implies that the two resulting colimits coincide. Thus, Part 3 implies Part 2.

Part 2 in particular applies to the Yoneda embedding $y\colon\mathcal{C}\rightarrow\hat{\mathcal{C}}$, and so it 
implies Part 1. To prove Part 3 from Part 1, we may factor the diagram $U\colon I\rightarrow\mathcal{C}$ into a right 
anodyne inclusion
$\iota_U\colon I\rightarrow E_U$ followed by a right fibration
$\pi_U\colon E_U\twoheadrightarrow\mathcal{C}$. As right anodyne inclusions are cofinal
\cite[Proposition 4.1.1.3.]{luriehtt}, we obtain an equivalence
\[\mathrm{colim}_I(yU)\simeq\mathrm{colim}_I(y\pi_U\iota_U)\simeq\mathrm{colim}_{E_F}(y\pi_U)\]
of presheaves. For any right fibration $p\colon E\twoheadrightarrow\mathcal{C}$, the colimit of the composition
$\mathrm{colim}_E(yp)\colon E\rightarrow\hat{\mathcal{C}}$ computes the Straightening
$\mathrm{St}(p)\in\hat{\mathcal{C}}$ of $p$ (as shown explicitly in the proof of Lemma~\ref{lemmastcolim}).
The same construction applied to the diagram $V\colon J\rightarrow\mathcal{C}$ thus induces a composite equivalence
\[\mathrm{St}(\pi_U)\simeq\mathrm{St}(\pi_V)\]
in $\hat{\mathcal{C}}$. By subsequent Unstraightening we obtain an equivalence $\pi_U\simeq\pi_V$ of right fibrations 
over $\mathcal{C}$ in return, and hence the following commutative diagram.
\[\xymatrix{
 & E_U\ar[rr]^{\simeq}\ar@{->>}[dr]_{\pi_U} &  & E_V\ar@{->>}[dl]^{\pi_V} &  \\
 I\ar[rr]_U\ar[ur]^{\iota_U} & & \mathcal{C} & & J\ar[ll]^V\ar[ul]_{\iota_V}
}\]
The inclusions $\iota_U$ and $\iota_V$ are cofinal, and the post-composition of a cofinal functor with a categorical 
equivalence is again cofinal \cite[Corollary 4.1.1.9]{luriehtt}. Here, we note that contravariant equivalences 
between right fibrations indeed induce categorical equivalences of total $\infty$-categories. Both functors $\pi_U$ 
and $\pi_V$ are right fibrations by construction.
\end{proof}

Given a pre-descent diagram $U\colon I\rightarrow\mathcal{C}_{/B}$ in an $\infty$-category $\mathcal{C}$ 
(Definition~\ref{defpredescdiag}), we discussed in (\ref{equpredefpredescent}) a canonical functor of
$\infty$-categories of the form
\begin{align*}
\mathrm{res}_U\colon\mathcal{C}_{/B}\rightarrow\mathrm{Desc}(U),
\end{align*}
where $\mathrm{Desc}(U)$ denotes the full sub-$\infty$-category of $\mathrm{Fun}(I,\mathcal{C}_{/C})_{/U}$ spanned by 
the cartesian natural transformation over $U$. It maps an object $f\colon C\rightarrow B$ to the cartesian natural 
transformation $\mathrm{res}_U(f)$ given pointwise by its associated pullbacks along the arrows
$U_i\colon sU_i\rightarrow B$. 
Whenever $\mathcal{C}$ has pullbacks, we can describe $\mathrm{res}_U$ alternatively as follows.

Given any diagram $U\colon I\rightarrow\mathcal{C}_{/B}$, we may consider the slice functor
\[(\mathcal{C}_{/B})_{/-}\colon(\mathcal{C}_{/B})^{op}\rightarrow\mathrm{Cat}_{\infty}\] 
as well as its pre-composition with $U\colon I\rightarrow\mathcal{C}_{/B}$. We obtain a canonical functor
\begin{align}\label{equresalt}
\mathrm{res}_U\colon\mathrm{lim}(\mathcal{C}_{/B})_{/-}\rightarrow{\mathrm{lim}}U^{\ast}(\mathcal{C}_{/B})_{/-}
\end{align}
between the limits. As $\mathcal{C}_{/B}$ has a terminal object given by the identity $1_B$, the limit of
$(\mathcal{C}_{/B})_{/-}$ is just the $\infty$-category $(\mathcal{C}_{/B})_{/1_B}\simeq\mathcal{C}_{/B}$. 
Furthermore, there is a natural equivalence $(\mathcal{C}_{/B})_{/-}\simeq\mathcal{C}_{/s(-)}$, and the limit
$\mathrm{lim}_{i\in I}\mathcal{C}_{/sU_i}$ is exactly the $\infty$-category 
$\mathrm{Desc}(U)$ of cartesian natural transformations over $U$ by \cite[Corollary 3.3.3.2]{luriehtt}. One shows 
that the functor $\mathrm{res}_U$ in (\ref{equresalt}) is naturally equivalent to the functor $\mathrm{res}_U$ in 
(\ref{equpredefpredescent}) by computing that both represent the same cone
$\mathcal{C}_{/B}\rightarrow U^{\ast}(\mathcal{C}_{/B})_{/-}$.

\begin{lemma}\label{lemmaresfunct}
Let $\mathcal{C}$ be an $\infty$-category with pullbacks, let $B\in\mathcal{C}$ be an object, and
$U\colon I\rightarrow\mathcal{C}_{/B}$ be a diagram.
Then every functor $\phi\colon J\rightarrow I$ of $\infty$-categories induces a commutative triangle
\begin{align}\label{diaglemmaresfunct1}
\begin{gathered}
\xymatrix{
 & \mathrm{Desc}(U)\ar[dd]^{\mathrm{res}_{\phi}}\\
\mathcal{C}_{/B}\ar[ur]^{\mathrm{res}_U}\ar[dr]_{\mathrm{res}_{U\phi}} & \\
 & \mathrm{Desc}(U\phi)
}
\end{gathered}
\end{align}
of $\infty$-categories. The functor $\mathrm{res}_{\phi}$ is an equivalence whenever $\phi$ is cofinal.
\end{lemma}
\begin{proof}
The triangle in the statement is given by
\begin{align}\label{diaglemmaresfunct2}
\begin{gathered}
\xymatrix{
 & \underset{i\in I}{\mathrm{lim}}(\mathcal{C}_{/B})_{/U_i}\ar[dd]^{\mathrm{res}_{\phi}} \\
\underset{f\in \mathcal{C}_{/B}}{\mathrm{lim}}(\mathcal{C}_{/B})_{/f}\ar[ur]^{\mathrm{res}_U}\ar[dr]_{\mathrm{res}_{U\phi}} & \\
 & \underset{j\in J}{\mathrm{lim}}(\mathcal{C}_{/B})_{/U\phi_j}
}
\end{gathered}
\end{align}
The fact that the triangle commutes is easily seen by the fact that both $\mathrm{res}_{U\phi}$ and
$\mathrm{res}_{\phi}\circ\mathrm{res}_U$ represent the same cone. Whenever $\phi$ is cofinal, the functor
$\mathrm{res}_{\phi}$ is an equivalence by Lemma~\ref{lemmacofequiv}.2 applied to the slice functor
$(\mathcal{C}_{/B})_{/-}\colon\mathcal{C}_{/B}\rightarrow\mathrm{Cat}_{\infty}^{op}$. 

\end{proof}

\bibliographystyle{amsplain}
\bibliography{BSBib}
\Address

\end{document}